\documentclass[11pt, letterpaper]{amsart}

\usepackage{amssymb,amscd}
\usepackage{mathrsfs}
\usepackage{hyperref}
\usepackage[svgnames]{xcolor}
\hypersetup{colorlinks,breaklinks,
            linkcolor=DarkBlue,urlcolor=DarkBlue,
            anchorcolor=DarkBlue,citecolor=DarkBlue}
\usepackage{xypic}
\usepackage{tikz}
\usetikzlibrary{decorations.pathreplacing,shapes.geometric}
\usetikzlibrary{calc,positioning}
\usepackage{graphicx}
\usepackage{euscript}
\usepackage[verbose,letterpaper,tmargin=1in,bmargin=1in,lmargin=1in,rmargin=1in]
{geometry}

\newtheorem{theorem}{Theorem}[section]

\newtheorem{proposition}[theorem]{Proposition}
\newtheorem{corollary}[theorem]{Corollary}
\newtheorem{lemma}[theorem]{Lemma}

\theoremstyle{definition}
\newtheorem{definition}[theorem]{Definition}
\newtheorem*{acknowledgement}{Acknowledgements}
\newtheorem{example}[theorem]{Example}
\newtheorem{remark*}[theorem]{}

\theoremstyle{remark}
\newtheorem{remark}[theorem]{Remark}

\newcommand{\Q}{\mathbb Q}
\newcommand{\R}{\mathbb R}
\newcommand{\C}{\mathbb C}
\newcommand{\Z}{\mathbb Z}
\newcommand{\N}{\mathbb N}
\newcommand{\T}{\mathbb T}
\newcommand{\e}{\mathbf e}
\newcommand{\tr}{\mathrm{tr}}
\newcommand{\Tr}{\mathrm{Tr}}
\DeclareMathOperator{\Hom}{Hom}
\DeclareMathOperator{\Der}{Der}
\DeclareMathOperator{\Ker}{Ker}
\newcommand{\ms}[1]{\mathscr{#1}}
\newcommand{\mc}[1]{\mathcal{#1}}
\newcommand{\mb}[1]{\mathbb{#1}}
\newcommand{\frk}[1]{\mathfrak{#1}}
\newcommand{\eu}[1]{\EuScript{#1}}

\renewcommand{\labelenumi}{(\arabic{enumi})}
\renewcommand{\labelenumii}{(\roman{enumii})}

\tikzset{Box/.style={very thick, rounded corners}}
\tikzset{marked/.style={star, star point height = .75mm, star points =5, fill=black,minimum size=2mm, inner sep=0mm} }
\tikzset{verythickline/.style = {line width=7pt}}
\tikzset{thickline/.style = {line width=5pt}}
\tikzset{medthick/.style = {line width=3pt}}
\tikzset{med/.style = {line width=2pt}}
\tikzset{count/.style = {fill=white,circle,draw,thin, inner sep=2pt}}
\tikzset{rcount/.style = {fill=white,rectangle,draw,thin,inner sep=2pt, rounded corners}}
\tikzset{cpr/.style = {draw,fill=white,rectangle,thin, rounded corners}}
\begin{document}

\title{On the planar algebra of Ocneanu's asymptotic inclusion}
\author{Stephen Curran$^\dagger$}\thanks{\noindent $\dagger$: Research supported by an NSF postdoctoral fellowship, NSF grant DMS-0900776 and DARPA Award 0011-11-0001.}
\address{Department of Mathematics, UCLA, Los Angeles, CA 90095.}
\email{\href{mailto:curransr@math.ucla.edu}{curransr@math.ucla.edu}}

\begin{abstract}
In recent joint work with V. Jones and D. Shlyakhtenko, we have given a diagrammatic description of Popa's symmetric enveloping inclusion for planar algebra subfactors.  In this paper we give a diagrammatic construction of the associated Jones tower, in the case that the planar algebra is finite-depth.  We then use this construction to describe the planar algebra of the symmetric enveloping inclusion, which is known to be isomorphic to the planar algebra of Ocneanu's asymptotic inclusion by a result of Popa.  As an application we give a planar algebraic computation of the (reduced) fusion algebra of the asymptotic inclusion, recovering some well-known results of Ocneanu and Evans-Kawahigashi.
\end{abstract}

\maketitle

\section*{Introduction}

Let $M_0 \subset M_1$ be a finite-index inclusion of AFD $II_1$ factors, and let $M_0 \subset M_1 \subset M_2 \subset \dotsb \subset M_\infty$ be the associated Jones tower \cite{jon}.  The \textit{asymptotic inclusion} is the subfactor $M_1 \vee (M_1 ' \cap M_\infty) \subset M_\infty$, which has finite index if and only if $M_0 \subset M_1$ has finite depth.  This construction was introduced by Ocneanu \cite{ocn}, who argued that the asymptotic inclusion could be viewed as the subfactor analogue of Drinfeld's quantum double construction.  This connection has since been clarified by a number of authors, including Evans-Kawahigashi \cite{evk}, Izumi \cite{izumi, izumi2} and M\"{u}ger \cite{mug}.  A related construction for type $III$ factors has been developed by Longo and Rehren \cite{lrehr}. 

In \cite{pop1}, Popa introduced the \textit{symmetric enveloping algebra} associated to a finite index inclusion $M_0 \subset M_1$ of arbitrary $II_1$ factors.  This is the (unique) $II_1$ factor $M_1 \boxtimes_{e_0} M_1^{op}$ which is generated by a copy of $M_1 \otimes M_1^{op}$ and a projection $e_0$ which is simultaneously the Jones projection for both $M_0 \subset M_1$ and $M_0^{op} \subset M_1^{op}$.  In the case that $M_0 \subset M_1$ is a finite-depth inclusion of AFD $II_1$ factors, Popa proved that the \textit{symmetric enveloping inclusion} $M_1 \otimes M_1^{op} \subset M_1 \boxtimes_{e_0} M_1^{op}$ is conjugate to Ocneanu's asymptotic inclusion.  In the infinite-depth case, Popa \cite{pop2} used this inclusion to analyze a number of important analytic properties of the original subfactor $M_0 \subset M_1$.

Planar algebras were introduced by Jones \cite{palg} in the late 90s.  Any finite-index subfactor gives rise to a planar algebra as its \textit{standard invariant}.  By a fundamental result of Popa \cite{popa}, any planar algebra satisfying suitable positivity conditions arises in this way.  With inspiration from random matrix theory and Voiculescu's free probability, Guionnet, Jones and Shlyakhtenko \cite{gjs1} have recently given a diagrammatic construction of a tower of subfactors $M_0 \subset M_1 \subset \dotsb$, starting from a planar algebra $\mc P$.  In subsequent work \cite{gjs2}, they proved that if $\mc P$ is finite-depth then $M_k$ is an interpolated free group factor $L\mb F_{r_k}$, where $r_k$ is computed in terms of the index $\delta^2 = [M_1:M_0]$ and the \textit{global index} $I = [M_1 \boxtimes_{e_0} M_1^{op}:M_1 \otimes M_1^{op}]$.  In joint work with Jones and Shlyakhtenko \cite{cjs}, we have given a diagrammatic construction of the symmetric enveloping algebra associated to these planar algebra subfactors (see also \cite{shlicm}) .  As an application of our construction, we computed a certain free entropy dimension type quantity which provides some intuition for the formula for $r_k$.

In this paper, we give a diagrammatic construction of the Jones tower of the symmetric enveloping inclusion for planar algebra subfactors, in the case that $\mc P$ is finite-depth.  We then use this construction to derive a complete description of the planar algebra of the asymptotic inclusion in terms of the original planar algebra $\mc P$.  As an application, we compute the fusion rules for bimodules arising from the asymptotic inclusion, recovering some well-known results of Ocneanu and Evans-Kawahigashi \cite{evk}.  In particular, we show that the (reduced) fusion algebra of $M_1 \boxtimes_{e_0} M_1^{op}$ bimodules is described in terms of the affine category \cite{ghosh}, which is closely related to Ocneanu's tube algebra \cite{evk} and Jones' annular category \cite{annular}.  The relationship between the Drinfeld center of a fusion category and the annular/affine category is well-known in the TQFT community (see e.g. \cite{fnww}).  In the planar algebra setting, it was very recently shown by Das, Ghosh and Gupta \cite{dgg2} that the category of affine Hilbert representations of a finite-depth planar algebra $\mc P$ is equivalent to the Drinfeld center of the fusion category associated to $\mc P$.  

The paper is organized as follows.  In Section 1 we briefly recall the constructions from \cite{gjs1}, \cite{cjs}.  In Section 2 we study two canonical elements appearing in any finite-depth planar algebra, and establish a number of useful ``skein'' relations which these satisfy.  In Section 3 we construct the Jones tower for the symmetric enveloping inclusion and compute the higher relative commutants.  Section 4 contains our main result: a description of the planar algebra of the asymptotic inclusion.   In Section 5 we use this description to compute the fusion rules for the asymptotic inclusion.

\begin{acknowledgement}
I am grateful to Dimitri Shlyakhtenko for suggesting this problem, and for many useful discussions while completing this project.  I would also like to thank Dietmar Bisch and Vaughan Jones for several helpful conversations.
\end{acknowledgement}

\section{Background and preliminaries} \label{sec:background}

In this section we briefly recall the constructions from \cite{gjs1}, \cite{cjs}.

\medskip
\noindent \textbf{Planar algebra subfactors:} 
Let $\mc P = (P_k)_{k \geq 0}$ be a subfactor planar algebra.  For $n,k \geq 0$ let $P_{n,k}$ be a copy of $P_{n+k}$.  Elements of $P_{n,k}$ will be represented by diagrams
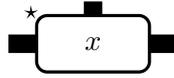
\begin{equation*}
 \begin{tikzpicture}[scale=.75]
 \draw[Box](0,0) rectangle (2,1);
  \draw[verythickline](0,.5) -- (-.5,.5); \draw[verythickline](2,.5) -- (2.5,.5);
\draw[verythickline](1,1) -- (1,1.25);
 \node at (1,.5) {$x$}; \node[marked] at (-.1,1.05) {};
 \end{tikzpicture}\end{equation*}
where the thick lines to the left and right represent $k$ strings, and the thick line at top represents $2n$ strings.  We will typically suppress the marked point $\star$, and take the convention that it occurs at the top-left corner which is adjacent to an unshaded region.

Define a product $\wedge_k:P_{n,k} \times P_{m,k} \to P_{n+m,k}$ by
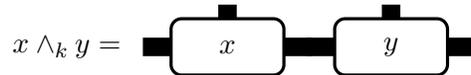
\begin{equation*}
\begin{tikzpicture}[thick,scale=.75]
\draw[Box](0,0) rectangle (2,1);
  \draw[verythickline](0,.5) -- (-.5,.5); \draw[verythickline](2,.5) -- (2.5,.5);
 \draw[verythickline](1,1) -- (1,1.25);
 \node at (1,.5) {$x$}; \node[left] at (-.7,.5) {$x \wedge_k y =$};
\begin{scope}[xshift=2.9cm]
\draw[Box](0,0) rectangle (2,1);
  \draw[verythickline](0,.5) -- (-.5,.5); \draw[verythickline](2,.5) -- (2.5,.5);
 \draw[verythickline](1,1) -- (1,1.25);
 \node at (1,.5) {$y$};
\end{scope}
\end{tikzpicture}
\end{equation*}
The involution $\dagger:P_{n,k} \to P_{n,k}$ is given by
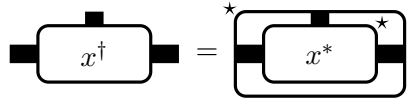
\begin{equation*}
\begin{tikzpicture}[thick,scale=.75]
\draw[Box](0,0) rectangle (2,1);
  \draw[verythickline](0,.5) -- (-.5,.5); \draw[verythickline](2,.5) -- (2.5,.5);
 \draw[verythickline](1,1) -- (1,1.25);
 \node at (1,.5) {$x^\dagger$}; \node at (3,.5) {$=$};
\begin{scope}[xshift=4cm]
\draw[Box](0,0) rectangle (2,1);
  \draw[verythickline](0,.5) -- (-.5,.5); \draw[verythickline](2,.5) -- (2.5,.5);
 \draw[verythickline](1,1) -- (1,1.25);
 \node at (1,.5) {$x^*$};
 \draw[Box] (-.5,-.25) rectangle (2.5,1.25); \node[marked] at (2.1,1.05) {}; \node[marked] at (-.6,1.3) {};
\end{scope}
\end{tikzpicture}
\end{equation*}

The \textit{Voiculescu trace} $\tau_k:P_{n,k} \to \C$ is defined by
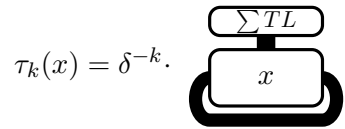
\begin{equation*}
\begin{tikzpicture}[scale=.75]
  \draw[Box] (0,0) rectangle (2,1); \node at (1,.5) {$x$};
  \draw[thickline] (0,.5) arc(90:270:.25cm and .375cm) -- (2,-.25) arc(-90:90:.25cm and .375cm);
  \draw[Box] (0,1.25) rectangle (2,1.75); \node[scale=.8] at (1,1.5) {$\sum TL$};
\draw[verythickline] (1,1) -- (1,1.25);
  \node[left] at (-.5,.75) {$\tau_k(x) = \delta^{-k} \cdot$};
 \end{tikzpicture}
\end{equation*}
where $\sum TL$ denotes the sum over all loopless Temperley-Lieb diagram with $2n$ boundary points.  Let $Gr_k(\mc P) = \bigoplus_{n \geq 0} P_{n,k}$, and observe that the formulas above give $Gr_k(\mc P)$ the structure of a graded $*$-algebra with trace $\tau_k$.  The unit of $Gr_k(\mc P)$ is the element of $P_{0,k}$ consisting of $k$ parallel lines. 

Let $\mathbf{e}_{k}$ denote the following element of $P_{0,k+2}$, $k \geq 0$:
\begin{equation*}
\begin{tikzpicture}
\begin{scope}[xscale=.75,yscale=.5]
 \draw[Box] (0,0) rectangle (1,1.75); \draw[verythickline] (0,1.15) -- (1,1.15);
 \draw[thick] (0,.75) arc(90:-90:.25cm); \draw[thick] (1,.75) arc(90:270:.25cm);
\node[left=2mm,scale=1] at (0,.75) {$\mathbf{e}_{k} = $};
\end{scope}
\end{tikzpicture}
\end{equation*}
Note that there are natural inclusions of $Gr_k(\mc P)$ into $Gr_{k+1}(\mc P)$ defined by
\begin{equation*}
 \begin{tikzpicture}[scale=.75]
 \draw[Box](0,0) rectangle (2,1);
  \draw[verythickline](0,.5) -- (-.5,.5); \draw[verythickline](2,.5) -- (2.5,.5);
\draw[verythickline](1,1) -- (1,1.25);
 \node at (1,.5) {$x$};
\begin{scope}[xshift=4.2cm]
 \draw[Box](0,0) rectangle (2,1);
  \draw[verythickline](0,.5) -- (-.5,.5); \draw[verythickline](2,.5) -- (2.5,.5);
\draw[verythickline](1,1) -- (1,1.25);
\draw[thick] (-.5,-.25) -- (2.5,-.25);
 \node at (1,.5) {$x$};
\node[left=1mm,scale=1.2] at (-.5,.5) {$\mapsto$};
\end{scope}
 \end{tikzpicture}
\end{equation*}
The main result of \cite{gjs1} is the following:

\begin{theorem}
For $k \geq 0$, the Voiculescu trace $\tau_k$ is a faithful tracial state on $Gr_k(\mc P)$, and its GNS completion is a $II_1$ factor $M_k$ as long as $\delta > 1$.  The inclusions $Gr_k(\mc P) \subset Gr_{k+1}(\mc P)$ extend to $M_k \subset M_{k+1}$, and $(M_{k+1},\mathbf e_{k})$ is the Jones tower of $M_0 \subset M_1$.  Moreover, the planar algebra of $M_0 \subset M_1$ is isomorphic to $\mc P$. \qed
\end{theorem}

\medskip
\noindent\textbf{Popa's symmetric enveloping algebra:}
For integers $k,s,t$ with $s+t + 2k = n$, let $V_{k}(s,t)$ be a copy of $P_{n}$.  Elements of $V_{k}(s,t)$ will be represented by diagrams of the form
\begin{equation*}
 \begin{tikzpicture}[scale=.75]
 \draw[Box](0,0) rectangle (2,1);
  \draw[verythickline](0,.5) -- (-.5,.5); \draw[verythickline](2,.5) -- (2.5,.5);
 \draw[verythickline](1,0) -- (1,-.25); \draw[verythickline](1,1) -- (1,1.25);
 \node at (1,.5) {$x$}; \node[marked, above left] at (0,1) {};
 \end{tikzpicture}
\end{equation*}
where there are $2s$ parallel strings at the top, $2t$ strings at the bottom and $2k$ at either side.   As above, we will use the convention that the marked point occurs at the upper left corner, which is adjacent to a unshaded region. 

Define a product $\wedge:V_{k}(s,t) \times V_{k}(s',t') \to V_{k}(s+s',t+t')$ by
\begin{equation*}
\begin{tikzpicture}[thick,scale=.75]
\draw[Box](0,0) rectangle (2,1);
  \draw[verythickline](0,.5) -- (-.5,.5); \draw[verythickline](2,.5) -- (2.5,.5);
 \draw[verythickline](1,0) -- (1,-.25); \draw[verythickline](1,1) -- (1,1.25);
 \node at (1,.5) {$x$}; \node[left] at (-.7,.5) {$x \wedge y =$};
\begin{scope}[xshift=2.9cm]
\draw[Box](0,0) rectangle (2,1);
  \draw[verythickline](0,.5) -- (-.5,.5); \draw[verythickline](2,.5) -- (2.5,.5);
 \draw[verythickline](1,0) -- (1,-.25); \draw[verythickline](1,1) -- (1,1.25);
 \node at (1,.5) {$y$};
\end{scope}
\end{tikzpicture}
\end{equation*}
The adjoint $\dagger:V_{k}(s,t) \to V_{k}(s,t)$ is defined by
\begin{equation*}
\begin{tikzpicture}[thick,scale=.75]
\draw[Box](0,0) rectangle (2,1);
  \draw[verythickline](0,.5) -- (-.5,.5); \draw[verythickline](2,.5) -- (2.5,.5);
 \draw[verythickline](1,0) -- (1,-.25); \draw[verythickline](1,1) -- (1,1.25);
 \node at (1,.5) {$x^\dagger$};
\node at (3,.5) {$=$};
\begin{scope}[xshift=4cm]
\draw[Box](0,0) rectangle (2,1);
  \draw[verythickline](0,.5) -- (-.5,.5); \draw[verythickline](2,.5) -- (2.5,.5);
 \draw[verythickline](1,0) -- (1,-.25); \draw[verythickline](1,1) -- (1,1.25);
 \node at (1,.5) {$x^*$}; \node[marked] at (2.1,1.05) {};
\draw[Box](-.5,-.25) rectangle (2.5,1.25); \node[marked] at (-.6,1.3) {};
\end{scope}
\end{tikzpicture} 
\end{equation*}
Define $\tau_k \boxtimes \tau_k:V_{k}(s,t) \to \C$ by
\begin{equation*}
\begin{tikzpicture}[scale=.75]
  \draw[Box] (0,0) rectangle (2,1); \node at (1,.5) {$x$};
  \draw[thickline] (0,.5) arc(90:270:.5cm and .75cm) -- (2,-1) arc(-90:90:.5cm and .75cm);
  \draw[Box] (0,-.25) rectangle (2,-.75); \node[scale=.8] at (1,-.5) {$\sum TL$};
  \draw[Box] (0,1.25) rectangle (2,1.75); \node[scale=.8] at (1,1.5) {$\sum TL$};
  \draw[verythickline](1,0) -- (1,-.25); \draw[verythickline] (1,1) -- (1,1.25);
  \node[left] at (-.7,.5) {$(\tau_k \boxtimes \tau_k)(x) = $};
 \end{tikzpicture}
\end{equation*}
Define
\begin{equation*}
 Gr_k \boxtimes Gr_k^{op} = \bigoplus_{s,t \geq 0} V_{k}(s,t).
\end{equation*}

Note that there is a natural anti-automorphism $y \mapsto y^{op}$ of $Gr_k(\mc P) \boxtimes Gr_k(\mc P)$, determined by
\begin{equation*}
\begin{tikzpicture}[thick,scale=.75]
\draw[Box](0,0) rectangle (2,1);
  \draw[verythickline](0,.5) -- (-.5,.5); \draw[verythickline](2,.5) -- (2.5,.5);
 \draw[verythickline](1,0) -- (1,-.25); \draw[verythickline](1,1) -- (1,1.25);
 \node at (1,.5) {$y^{op}$}; \node[marked] at (-.1,1.05) {};
\node at (3,.5) {$=$};
\begin{scope}[xshift=4cm]
\draw[Box](0,0) rectangle (2,1);
  \draw[verythickline](0,.5) -- (-.5,.5); \draw[verythickline](2,.5) -- (2.5,.5);
 \draw[verythickline](1,0) -- (1,-.25); \draw[verythickline](1,1) -- (1,1.25);
 \draw[Box] (-.5,-.25) rectangle (2.5,1.25); \node[marked] at (-.6,1.35) {};
 \node at (1,.5) {$y$}; \node[marked] at (2.1,-.05) {};
\end{scope}
\end{tikzpicture} 
\end{equation*}

Observe also that $Gr_k(\mc P) \boxtimes Gr_k(\mc P)$ contains a copy of $Gr_k(\mc P) \otimes Gr_k(\mc P)^{op}$ as follows:
\begin{equation*}
 \begin{tikzpicture}[scale=.75]
 \draw[Box] (0,0) rectangle (2,1); \node at (1,.5) {$x$};
 \draw[verythickline] (-.5,.5) -- (0,.5); \draw[verythickline] (2,.5) -- (2.5,.5);
 \draw[verythickline] (1,1) -- (1,1.25); \node[marked] at (-.1,1.05) {};
\begin{scope}[rotate around={180:(1,.5)}, yshift=1.5cm]
 \draw[Box] (0,0) rectangle (2,1); \node at (1,.5) {$y$}; \node[marked] at (-.1,1.05) {};
 \draw[verythickline] (-.5,.5) -- (0,.5); \draw[verythickline] (2,.5) -- (2.5,.5);
 \draw[verythickline] (1,1) -- (1,1.25);
\end{scope}
\draw[Box] (-.5,-1.75) rectangle (2.5,1.25); \node[marked] at (-.6,1.35) {};
\node[left=2mm,scale=1.1] at (-.5,-.25) {$x \otimes y^{op}  \mapsto$};
 \end{tikzpicture}
\end{equation*}
Let $\mathbf{f}_{k} \in Gr_k \boxtimes Gr_k$ be the projection
\begin{equation*}
 \begin{tikzpicture}[xscale=.5,yscale=.75]
  \draw[Box] (0,0) rectangle (2,1.5); \draw[thick] (0,.9) arc (90:-90:.3cm and .15cm);
  \draw[thick] (2,.9) arc(90:270:.3cm and .15cm); \draw[thickline] (0,1.2) -- node[scale=.7,count,rectangle, rounded corners] {$k-1$} (2,1.2);
  \draw[thickline] (0,.3) -- (2,.3);
  \node[left =1mm] at (0,.75) {$\mathbf f_{k} = \delta^{-1} \cdot$};
 \end{tikzpicture}
\end{equation*}
Note that $\mathbf f_k$ implements the Jones projections for both $M_{k-1} \subset M_k$ and $M_{k-1}^{op} \subset M_k^{op}$.  

We have proved the following result in \cite{cjs}:

\begin{theorem}
For $k \geq 0$, $\tau_k \boxtimes \tau_k$ is a faithful, tracial state on $Gr_k(\mc P) \boxtimes Gr_k(\mc P)$, and its GNS completion is a $II_1$ factor  $M_k \boxtimes M_k$ as long as $\delta > 1$.  The inclusion $Gr_k(\mc P) \otimes Gr_k(\mc P) \subset Gr_{k} \boxtimes Gr_{k}$ extends to $M_k \otimes M_k \subset M_k \boxtimes M_k$.  Moreover, for $k \geq 1$ we have that $M_k \otimes M_k^{op} \subset M_k \boxtimes M_k^{op}$ is isomorphic to Popa's symmetric enveloping inclusion $M_k \otimes M_k^{op} \subset M_k \boxtimes_{f_{k-1}} M_k$. \qed
\end{theorem}

\section{Skein relations}\label{sec:skein}

In this section we introduce two canonical elements $p$ and $q$ appearing in any finite-depth planar algebra, and describe a number of diagrammatic relations which they satisfy.

Throughout the paper $\mc P = (P_k)_{k \geq 0}$ will be a subfactor planar algebra.  Let $\Gamma$ denote the principal graph of $\mc P$ (see e.g. \cite{ghj}), and let $\Gamma_+$ denote the collection of even vertices.  Associated to each $v \in \Gamma_+$ is an irreducible $M_0-M_0$ bimodule $X_v$.  We let $*$ denote the distinguished vertex, corresponding to the bimodule $X_* = L^2(M_0)$.  

We will assume that $\Gamma_+$ is finite (i.e. $\mc P$ is finite-depth), and throughout the paper we fix $k$ such that $d(*,v) \leq 2k$ for all $v \in \Gamma_+$.  It follows that we have
\begin{equation*}
 {}_{M_0}L^2(M_k)_{M_0} \simeq \bigoplus_{v \in \Gamma_{+}} X_v \otimes \mc H_v,
\end{equation*}
where $\mc H_v$ are auxiliary finite-dimensional Hilbert spaces, whose dimensions we denote by $n_v$.

It follows that
\begin{equation*}
{}_{M_0 \otimes M_0^{op}}L^2(M_k \otimes M_{k}^{op})_{M_0 \otimes M_0^{op}} \simeq \bigoplus_{v,w \in \Gamma_{+}} X_v \otimes \overline{X_w} \otimes (\mc H_v \otimes \overline{\mc H_w}),
\end{equation*}
where $\overline{X_w}$ is the contragredient bimodule and $\overline{\mc H_w}$ is the conjugate Hilbert space.  Let $1_v \in \mc H_v \otimes \overline{\mc H_v}$ be the unit under the natural identification with $\Hom_{\C}(\mc H_v)$, and define $p_v$ to be the projection from $L^2(M_k \otimes M_k^{op})$ onto $X_v \otimes \overline{X_v} \otimes 1_v$.  Then define
\begin{equation*}
 p = \sum_{v \in \Gamma_{+}} p_v.
\end{equation*}

Recall that $P_{2k}$ can be identified with $\Hom_{M_0,M_0}(L^2(M_{k}))$, in particular the central projections of $P_{2k}$ are indexed by $v \in \Gamma_{+}$.  By an abuse of notation we will use $v$ to denote the central projection corresponding to the vertex $v$.  Let $(\mu_v)_{v \in \Gamma}$ be the Perron-Frobenius eigenvector of the adjacency matrix for $\Gamma$, normalized by $\mu(*) = 1$.  Then the trace of a minimal projection in the central component of $P_{2k}$ corresponding to $v \in \Gamma_+$ is $\tr_{P_{2k}}(v) = \delta^{-2k}\mu_v$.

Note that if $\{e_{ij}(v):1 \leq i,j \leq n_v\}$ are matrix units for the central component of $P_{2k}$ corresponding to $v \in \Gamma_+$, then we have
\begin{equation*}
 p_v = \frac{1}{n_v} \sum_{1 \leq i,j \leq n_v} e_{ij}(v) \otimes e_{ji}(v)^{op}.
\end{equation*}

Graphically, we will represent $p_v$ using a Sweedler type convention as follows:
\begin{equation*}
\begin{tikzpicture}[scale=.75]
\draw[Box] (0,0) rectangle (2,1); \node at (1,.5) {$p^{(2)}_v$}; \node[marked,below right] at (2,0) {};
\draw[Box] (0,1.5) rectangle (2,2.5); \node at (1,2) {$p^{(1)}_v$}; \node[marked,above left] at (0,2.5) {};
\draw[thickline] (-.5,.5) -- (0,.5); 
\draw[thickline] (2,.5) -- (2.5,.5);
\draw[thickline] (-.5,2) -- (0,2); 
\draw[thickline] (2,2) -- (2.5,2);
\node[left] at (-.6,1.25) {$p_v = $};
\end{tikzpicture}
\end{equation*}
The projection $p_v$ satisfies the following key skein relation:

\begin{lemma}\label{2cleaver}
Let $x, y \in P_{2k}$ and $v \in \Gamma_{+}$, then we have
\begin{equation*}
\begin{tikzpicture}[scale=.65]
\draw[Box] (0,0) rectangle (2,1); \node at (1,.5) {$p^{(2)}_v$}; \node[marked, below right] at (2,0) {};
\draw[Box] (0,1.5) rectangle (2,2.5); \node at (1,2) {$p^{(1)}_v$}; \node[marked, above left] at (0,2.5) {};
\draw[Box] (2.5,0) rectangle (4.5,1); \node at (3.5,.5) {$y$}; \node[marked, below right] at (4.5,0) {};
\draw[Box] (2.5,1.5) rectangle (4.5,2.5); \node at (3.5,2) {$x$}; \node[marked, above left] at (2.5,2.5) {};
\draw[thickline] (4.5,.5) -- (4.75,.5) arc(90:-90:.5cm) -- (-.25,-.5) arc(270:90:.5cm) -- (0,.5); 
\draw[thickline] (4.5,2) -- (4.75,2) arc(-90:90:.5cm) -- (-.25,3) arc (90:270:.5cm) -- (0,2);
\draw[thickline] (2,.5) -- (2.5,.5);
\draw[thickline] (2,2) -- (2.5,2);
\begin{scope}[xshift=6.5cm]
\draw[Box] (2.5,0) rectangle (4.5,1); \node at (3.5,.5) {$y$}; \node[marked, below right] at (4.5,0) {};
\draw[Box] (2.5,1.5) rectangle (4.5,2.5); \node at (3.5,2) {$x$}; \node[marked, above left] at (2.5,2.5) {};
\draw[thickline] (4.5,2) -- (4.75,2) arc(90:0:.5cm) -- node[fill=white, draw, thin, rectangle,rounded corners] {$v$} (5.25,1) arc(0:-90:.5cm) -- (4.5,.5);\node[marked,scale=.8] at (5.75,1.25) {};
\draw[thickline] (2.5,2) -- (2.25,2) arc(90:180:.5cm) -- (1.75,1) arc(180:270:.5cm) -- (2.5,.5);
\node[left,scale=1.25] at (1.5,1.25) {$= \frac{\mu_v}{n_v} \cdot$};
\end{scope}
\end{tikzpicture}
\end{equation*}
\end{lemma}

\begin{proof}
The left hand side is equal to
\begin{equation*}
 \frac{\delta^{4k}}{n_v}\sum_{1 \leq i,j \leq n_v}\tr(e_{ij}(v)x)\cdot \tr(ye_{ji}(v)) = \frac{\delta^{4k}}{n_v}\cdot \delta^{-2k}\mu_v\cdot\tr(xy\cdot v),
\end{equation*}
which is equal to the right hand side.
\end{proof}

We will now derive a number of `skein' relations for the projection $p$.  First we recall the following well known relation:
\begin{lemma}\label{*-proj}
For $x,y \in P_k$ we have
\begin{equation*}
\begin{tikzpicture}[scale=.65]
 \draw[Box] (0,.5) rectangle (1,1.5); \node at (.5,1) {$x$}; \node[marked,scale=.9,above left] at (0,1.5) {};
\draw[Box] (0,-.5) rectangle (1,-1.5); \node at (.5,-1) {$y$}; \node[marked,scale=.9,below right] at (1,-1.5) {};

\draw[thickline] (.5,.5) --node[cpr,scale=.9] (v) {$v$} (.5,-.5); \node[marked,scale=.75,right=.03] at (v.east) {};

\begin{scope}[xshift=4cm]
 \draw[Box] (0,.5) rectangle (1,1.5); \node at (.5,1) {$x$}; \node[marked,scale=.9,above left] at (0,1.5) {};
\draw[Box] (0,-.5) rectangle (1,-1.5); \node at (.5,-1) {$y$}; \node[marked,scale=.9,below right] at (1,-1.5) {};

\draw[thickline] (.5,.5) -- (.5,-.5);

\node[left] at (-.5,0) {$= \delta_{v,*}$}; 
\end{scope}

\end{tikzpicture}
\end{equation*}
\end{lemma} \qed

\begin{proposition}\label{p-skein} $p$ satisfies the following skein relations:
\begin{enumerate}
 \item Trace: $\tr_{P_{2k}}(p) = \delta^{-4k} I$, i.e. 
\begin{equation*}
\begin{tikzpicture}[scale=.6]
\draw[Box] (0,0) rectangle (2,1); \node at (1,.5) {$p^{(2)}$}; \node[marked,scale=.75,below right] at (2,0) {};
\draw[Box] (0,1.5) rectangle (2,2.5); \node at (1,2) {$p^{(1)}$}; \node[marked,scale=.75,above left] at (0,2.5) {};
\draw[thickline] (0,2) arc(270:90:.4cm and .5cm) -- (2,3) arc(90:-90:.4cm and .5cm);
\draw[thickline] (0,.5) arc(90:270:.4cm and .5cm) -- (2,-.5) arc(-90:90:.4cm and .5cm);

\node[right] at (3,1.25) {$ = I$};
\end{tikzpicture}
\end{equation*}
\item Rotational invariance:
\begin{equation*}
\begin{tikzpicture}[scale=.6]
\draw[Box] (0,0) rectangle (2,1); \node at (1,.5) {$p^{(2)}$}; \node[marked,scale=.8] at (2.05,-.05) {};
\draw[Box] (0,1.5) rectangle (2,2.5); \node at (1,2) {$p^{(1)}$}; \node[marked,scale=.8] at (-.05,2.55) {};
\draw[thickline] (-.5,.5) -- (0,.5); 
\draw[thickline] (2,.5) -- (2.5,.5);
\draw[thickline] (-.5,2) -- (0,2); 
\draw[thickline] (2,2) -- (2.5,2);
\draw[Box] (-.5,.-.25) rectangle (2.5,2.75); \node[marked] at (-.55,2.8) {};
\begin{scope}[xshift=4.5cm]
\draw[Box] (0,0) rectangle (2,1); \node at (1,.5) {$p^{(2)}$}; \node[marked,scale=.8] at (-.05,1.05) {};
\draw[Box] (0,1.5) rectangle (2,2.5); \node at (1,2) {$p^{(1)}$}; \node[marked,scale=.8] at (2.05,1.45) {};
\draw[Box] (-.5,.-.25) rectangle (2.5,2.75); \node[marked] at (-.55,2.8) {};
\draw[thickline] (-.5,.5) -- (0,.5); 
\draw[thickline] (2,.5) -- (2.5,.5);
\draw[thickline] (-.5,2) -- (0,2); 
\draw[thickline] (2,2) -- (2.5,2);
\node[left] at (-.75,1.25) {$ = $};
\end{scope}
\end{tikzpicture}
\end{equation*}

\item Capping:
\begin{equation*}
\begin{tikzpicture}[scale=.65]
\draw[Box] (0,.25) rectangle (1,1.25); \node at (.5,.75) {$p^{(1)}_1$}; \node[marked,above left,scale=.8] at (0,1.25) {};
\draw[Box] (0,-.25) rectangle (1,-1.25); \node at (.5,-.75) {$p^{(2)}_1$}; \node[marked, below right,scale=.8] at (1,-1.25) {};

\draw[thickline] (0,.75) -- (-.5,.75);
\draw[thickline] (0,-.75) -- (-.5,-.75);
\draw[medthick] (1,1) arc(90:-90:.25cm);
\draw[medthick] (1,-1) arc(-90:90:.25cm); 

\begin{scope}[xshift=4.25cm]
 \draw[Box] (0,-.25) rectangle (1,.25); \node at (.5,0) {$*$}; \node[marked,scale=.8] at (1.15,0) {};

\draw[thickline] (.5,.25) arc(0:90:.75cm and .5cm);
\draw[thickline] (.5,-.25) arc(0:-90:.75cm and .5cm);

\node[left] at (-.7,0) {$\displaystyle = \frac{\delta^k}{n_*}\; \cdot$}; 
\end{scope}
\end{tikzpicture}
\end{equation*}

\item Multiplication:  We have
\begin{equation*}
 \begin{tikzpicture}[scale=.65]
\draw[Box] (0,.25) rectangle (1,1.25); \node at (.5,.75) {$p^{(1)}_1$}; \node[marked,above left,scale=.8] at (0,1.25) {};
\draw[Box] (0,-.25) rectangle (1,-1.25); \node at (.5,-.75) {$p^{(2)}_1$}; \node[marked, below right,scale=.8] at (1,-1.25) {};

\draw[thickline] (0,.75) -- (-.5,.75);
\draw[thickline] (0,-.75) -- (-.5,-.75);
\draw[thickline] (1,.75) arc(90:-90:.5cm and .75cm);

\begin{scope}[xshift=3.75cm]
\draw[thickline] (-.5,.75) arc(94:-94:.75cm);
\node[left] at (-.75,0) {$\displaystyle = $}; 
\end{scope}
\end{tikzpicture}
\end{equation*}

\end{enumerate}
\end{proposition}

\begin{proof}
For (1), we have 
\begin{equation*}
 \tr_{P_{2k}}(p) = \sum_{v \in \Gamma_+} n_v^{-1} \sum_{1 \leq i,j \leq n_v} \tr_{P_{2k}(v)}(e_{ij}(v))^2 = \sum_{v \in \Gamma_+} \delta^{-4k}\mu_v^2 = \delta^{-4k}I.
\end{equation*}
(2) follows easily from Lemma \ref{2cleaver}, and the fact that
\begin{equation*}
\begin{tikzpicture}[scale=.5]
\draw[Box] (0,0) rectangle (2,1); \node at (1,.5) {$v$}; \node[marked,scale=.6,below right] at (2,0) {};
\draw[Box](-1,-.5) rectangle (3,1.5); \node[marked,scale=.6,above left] at (-1,1.5) {};
\draw[thickline] (-1,.5) -- (0,.5); \draw[thickline] (2,.5) -- (3,.5) {};
\begin{scope}[xshift=6cm]
 \draw[Box] (0,0) rectangle (2,1); \node at (1,.5) {$\overline v$}; \node[marked,scale=.6,above left] at (0,1) {};
\draw[Box](-1,-.5) rectangle (3,1.5); \node[marked,scale=.6,above left] at (-1,1.5) {};
\draw[thickline] (-1,.5) -- (0,.5); \draw[thickline] (2,.5) -- (3,.5) {};
\end{scope}
\end{tikzpicture}
\end{equation*}
where $\overline v$ is the vertex corresponding to the conjugate bimodule $\overline{X_v}$.

By Lemma \ref{2cleaver} we have
\begin{equation*}
\begin{tikzpicture}[scale=.75]
\draw[Box] (0,.25) rectangle (1,1.25); \node at (.5,.75) {$p^{(1)}_1$};
\draw[Box] (0,-.25) rectangle (1,-1.25); \node at (.5,-.75) {$p^{(2)}_1$};

\draw[medthick] (1,1) arc(90:-90:.25cm);
\draw[medthick] (1,-1) arc(-90:90:.25cm); 

\draw[Box] (1.5,.25) rectangle (2.5,1.25); \node at (2,.75) {$x$}; \node[marked,scale=.8] at (1.45,1.3) {};
\draw[Box] (1.5,-.25) rectangle (2.5,-1.25); \node at (2,-.75) {$y$}; \node[marked,scale=.8] at (2.55,-1.3) {};

\draw[thickline] (0,.75) arc(270:90:.5cm) -- (2.5,1.75) arc(90:-90:.5cm);
\draw[thickline] (0,-.75) arc(90:270:.5cm) -- (2.5,-1.75) arc(-90:90:.5cm);

\begin{scope}[xshift=6.75cm]
\draw[Box] (1.5,.25) rectangle (2.5,1.25); \node at (2,.75) {$x$}; \node[marked,scale=.8] at (1.45,1.3) {};
\draw[Box] (1.5,-.25) rectangle (2.5,-1.25); \node at (2,-.75) {$y$}; \node[marked,scale=.8] at (2.55,-1.3) {};
\draw[thickline] (2.5,.75) arc(90:0:.5cm and .75cm) node[cpr] (v) {$v$} arc(0:-90:.5cm and .75cm); \node[right = .01,marked, scale=.8] at (v.east) {};

\draw[Box] (0,-.25) rectangle (1,.25); \node at (.5,0) {$v$}; \node[marked,left=.05cm] at (0,0) {};
\draw[medthick] (.25,.25) arc(180:0:.25cm);
\draw[medthick] (.25,-.25) arc(180:360:.25cm);

\node[left] at (-.25,-.25) {$\displaystyle = \sum_{v \in \Gamma_+} \frac{\mu_v}{n_v} \; \cdot$};
\end{scope}
\end{tikzpicture}
\end{equation*}
for any $x, y \in P_k$.  (3) then follows from Lemma \ref{*-proj}.

 For (4), the left hand side is equal to
\begin{equation*}
 \sum_{v \in \Gamma_+} \frac{1}{n_v} \sum_{1 \leq i,j \leq n_v} e_{ij}(v) e_{ji}(v) = \sum_{v \in \Gamma_+} v = 1.
\end{equation*}

\end{proof}

We now introduce an element $q \in P_{3k}$, which will be central to our constructions.
\begin{proposition}\label{qv-def}
Fix $v,w,z \in \Gamma_+$.  Then there is a unique element $q_{v,w,z} \in P_{3k} \otimes P_{3k}^{op}$, which we represent in Sweedler notation as $q_{v,w,z} = q_{v,w,z}^{(1)} \otimes q_{v,w,z}^{(2)}$, with the property that for any $x,y \in P_{3k}$ we have
\begin{equation*}
\begin{tikzpicture}[scale=.65]
\draw[Box] (0,0) rectangle (2,1); \node at (1,.5) {$q_{v,w,z}^{(2)}$}; \node[marked] at (2.05,-.05) {};
\draw[Box] (0,1.5) rectangle (2,2.5); \node at (1,2) {$q_{v,w,z}^{(1)}$}; \node[marked] at (-.05,2.55) {};

\draw[Box] (3,0) rectangle (5,1); \node at (4,.5) {$y$}; \node[marked] at (5.05,-.05) {};
\draw[Box] (3,1.5) rectangle (5,2.5); \node at (4,2) {$x$}; \node[marked] at (2.95,2.55) {};

\draw[thickline] (2,.5) -- (3,.5); \draw[thickline] (2,2) -- (3,2);
\draw[thickline] (0,.5) arc(90:270:.5cm and .625cm) -- (5,-.75) arc(-90:90:.5cm and .625cm);
\draw[thickline] (0,2) arc(270:90:.5cm and .625cm) -- (5,3.25) arc(90:-90:.5cm and .625cm);
\draw[thickline] (1,0) arc(180:270:.375cm) -- (3.625,-.375) arc(-90:0:.375cm);
\draw[thickline] (1,2.5) arc(180:90:.375cm) -- (3.625,2.875) arc(90:0:.375cm);

\begin{scope}[xshift=14.5cm]
\draw[Box] (0,0) rectangle (2,1); \node at (1,.5) {$y$}; \node[marked] at (2.05,-.05) {};
\draw[Box] (0,1.5) rectangle (2,2.5); \node at (1,2) {$x$}; \node[marked] at (-.05,2.55) {};

\draw[thickline] (0,2) -- ++(-.25,0) arc(90:180:.375cm) -- node[cpr] (z) {$z$} (-.625,.875) arc(180:270:.375cm) -- ++(.25,0); \node[marked,scale=.75,right] at (z.east) {};
\draw[thickline] (1,2.5) arc(0:90:.5cm) -- (-1.25,3) arc(90:180:.5cm) -- node[cpr] (w) {$w$} (-1.75,0) arc(180:360:1.375cm and .5cm); \node[marked,scale=.75,right] at (w.east) {};
\draw[thickline] (2,2) -- ++(.25,0) arc(90:0:.25cm) --  node[cpr](v) {$v$} ++(0,-1) arc(0:-90:.25cm) -- ++(-.25,0); \node[marked,scale=.75,right] at (v.east) {};
\node[left] at (-2,1.25) {$\displaystyle = I^{-1/2}\cdot\bigg(\frac{\mu_v\mu_w\mu_z}{n_vn_wn_z}\biggr)^{1/2} \; \cdot $};
\end{scope}
\end{tikzpicture}
\end{equation*}

\end{proposition}

\begin{proof}
With the obvious unital embedding of $P_{3k} \otimes P_{3k}^{op}$ into $P_{6k}$, we have
\begin{equation*}
 \begin{tikzpicture}[scale=.75]
 \draw[Box] (0,0) rectangle (.5,.5); \node at (.25,.25) {$v$}; \node[marked,scale=.75, right] at (.55,.25) {};
\draw[Box] (1,0) rectangle (1.5,.5); \node at (1.25,.25) {$w$};  \node[marked,scale=.75, right] at (1.55,.25) {};
\draw[Box] (2,0) rectangle (2.5,.5); \node at (2.25,.25) {$z$};  \node[marked,scale=.75, right] at (2.55,.25) {};

\draw [medthick] (.25,.5) arc(0:90:.75cm and .5cm);
\draw[medthick] (.25,0) arc(0:-90:.75cm and .5cm);

\draw[medthick] (2.25,.5) arc(180:90:.75cm and .5cm);
\draw[medthick] (2.25,0) arc(180:270:.75cm and .5cm);

\draw[med] (1.175,.5) arc(0:90:1.675cm and .875cm);
\draw[med] (1.325,.5) arc(180:90:1.675cm and .875cm);

\draw[med] (1.175,0) arc(0:-90:1.675cm and .875cm);
\draw[med] (1.325,0) arc(180:270:1.675cm and .875cm);

\draw[Box] (-.5,-1.125) rectangle (3,1.625); \node[marked,above left,scale = .9] at (-.5,1.625) {};

\node[left] at (-1,.25) {$q_{v,w,z} = I^{-1/2} \biggl(\frac{\mu_v\mu_w\mu_z}{n_vn_wn_z}\biggr)^{1/2} E_{P_{3k} \otimes P_{3k}^{op}}\Biggl[$}; \node[right] at (3.5,.25) {$\Biggr]$};
 \end{tikzpicture}
\end{equation*}

\end{proof}

Define
\begin{equation*}
 q = \sum_{v,w,z \in \Gamma_+} q_{v,w,z}.
\end{equation*}

\begin{proposition}\label{q-skein} $q$ satisfies the following skein relations:
\begin{enumerate}
\item Compatibility with $p$:
\begin{equation*}
\begin{tikzpicture}[scale=.65]
\draw[Box] (0,0) rectangle (2,1); \node at (1,.5) {$q^{(2)}$}; \node[marked] at (2.05,-.05) {};
\draw[Box] (0,1.5) rectangle (2,2.5); \node at (1,2) {$q^{(1)}$}; \node[marked] at (-.05,2.55) {};

\draw[Box] (2.5,0) rectangle (4.5,1); \node at (3.5,.5) {$p^{(2)}$}; \node[marked] at (4.55,-.05) {};
\draw[Box] (2.5, 1.5) rectangle (4.5,2.5); \node at (3.5,2) {$p^{(1)}$}; \node[marked] at (2.45,2.55) {};

\draw[thickline] (-.5,.5) -- (0,.5); \draw[thickline] (-.5,2) -- (0,2); \draw[thickline] (4.5,.5) -- (5,.5);
\draw[thickline] (2,.5) -- (2.5,.5); \draw[thickline] (2,2) -- (2.5,2); \draw[thickline] (4.5,2) -- (5,2);
\draw[thickline] (1,0) -- (1,-.25); \draw[thickline] (1,2.5) -- (1,2.75);

\begin{scope}[xshift=7cm]
\draw[Box] (0,0) rectangle (2,1); \node at (1,.5) {$q^{(2)}$}; \node[marked] at (2.05,-.05) {};
\draw[Box] (0,1.5) rectangle (2,2.5); \node at (1,2) {$q^{(1)}$}; \node[marked] at (-.05,2.55) {};
\draw[thickline] (-.5,.5) -- (0,.5); \draw[thickline] (-.5,2) -- (0,2);
\draw[thickline] (2,.5) -- (2.5,.5); \draw[thickline] (2,2) -- (2.5,2); 
\draw[thickline] (1,0) -- (1,-.25); \draw[thickline] (1,2.5) -- (1,2.75);
\node[left] at (-.75,1.25) {$ =$};
\end{scope}

\begin{scope}[xshift= 12.25cm]
\draw[Box] (2.5,0) rectangle (4.5,1); \node at (3.5,.5) {$q^{(2)}$}; \node[marked] at (2.05,-.05) {};
\draw[Box] (2.5,1.5) rectangle (4.5,2.5); \node at (3.5,2) {$q^{(1)}$}; \node[marked] at (-.05,2.55) {};

\draw[Box] (0,0) rectangle (2,1); \node at (1,.5) {$p^{(2)}$}; \node[marked] at (4.55,-.05) {};
\draw[Box] (0, 1.5) rectangle (2,2.5); \node at (1,2) {$p^{(1)}$}; \node[marked] at (2.45,2.55) {};

\draw[thickline] (-.5,.5) -- (0,.5); \draw[thickline] (-.5,2) -- (0,2); \draw[thickline] (4.5,.5) -- (5,.5);
\draw[thickline] (2,.5) -- (2.5,.5); \draw[thickline] (2,2) -- (2.5,2); \draw[thickline] (4.5,2) -- (5,2);
\draw[thickline] (3.5,0) -- (3.5,-.25); \draw[thickline] (3.5,2.5) -- (3.5,2.75);

\begin{scope}[xshift=7cm]
\draw[Box] (0,0) rectangle (2,1); \node at (1,.5) {$q^{(2)}$}; \node[marked] at (2.05,-.05) {};
\draw[Box] (0,1.5) rectangle (2,2.5); \node at (1,2) {$q^{(1)}$}; \node[marked] at (-.05,2.55) {};
\draw[thickline] (-.5,.5) -- (0,.5); \draw[thickline] (-.5,2) -- (0,2);
\draw[thickline] (2,.5) -- (2.5,.5); \draw[thickline] (2,2) -- (2.5,2); 
\draw[thickline] (1,0) -- (1,-.25); \draw[thickline] (1,2.5) -- (1,2.75);
\node[left] at (-.75,1.25) {$ =$};
\end{scope}
\node[left] at (-1,1.25) { and };
\end{scope}
\end{tikzpicture}
\end{equation*}

 \item Rotation:
\begin{equation*}
\begin{tikzpicture}[scale=.65]
\draw[Box] (0,0) rectangle (2,1); \node at (1,.5) {$q^{(2)}$}; \node[marked] at (2.05,-.05) {};
\draw[Box] (0,1.5) rectangle (2,2.5); \node at (1,2) {$q^{(1)}$}; \node[marked] at (-.05,2.55) {};
\draw[thickline] (-.5,.5) -- (0,.5); \draw[thickline] (2,.5) -- (2.5,.5); \draw[thickline] (1,0) -- (1,-.25);
\draw[thickline] (-.5,2) -- (0,2); \draw[thickline] (2,2) -- (2.5,2); \draw[thickline] (1,2.5) -- (1,2.75);

\begin{scope}[xshift=4.75cm]
\draw[Box] (0,0) rectangle (2,1); \node at (1,.5) {$q^{(2)}$}; \node[marked] at (-.05,-.05) {};
\draw[Box] (0,1.5) rectangle (2,2.5); \node at (1,2) {$q^{(1)}$}; \node[marked] at (2.05,2.55) {};
\draw[thickline] (-.5,.5) -- (0,.5); \draw[thickline] (2,.5) -- (2.5,.5); \draw[thickline] (1,0) -- (1,-.25);
\draw[thickline] (-.5,2) -- (0,2); \draw[thickline] (2,2) -- (2.5,2); \draw[thickline] (1,2.5) -- (1,2.75);
\node[left] at (-1,1.25) {$=$};
\end{scope}
\end{tikzpicture}
\end{equation*}

\item Capping: If $x,y \in P_{k}$ then
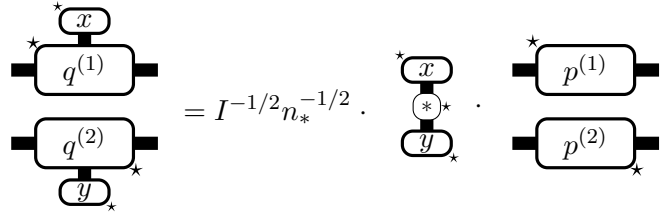
\begin{equation*}
\begin{tikzpicture}[scale=.65]
\draw[Box] (0,0) rectangle (2,1); \node at (1,.5) {$q^{(2)}$}; \node[marked] at (2.05,-.05) {};
\draw[Box] (0,1.5) rectangle (2,2.5); \node at (1,2) {$q^{(1)}$}; \node[marked] at (-.05,2.55) {};
\draw[Box] (.5,-.25) rectangle (1.5,-.75); \node at (1,-.5) {$y$}; \node[marked,scale=.75] at (1.55,-.8) {};
\draw[Box] (.5,2.75) rectangle (1.5,3.25); \node at (1,3) {$x$}; \node[marked,scale=.75] at (.45, 3.3) {};
\draw[thickline] (-.5,.5) -- (0,.5); \draw[thickline] (2,.5) -- (2.5,.5); 
\draw[thickline] (1,0) -- (1,-.25);
\draw[thickline] (-.5,2) -- (0,2); \draw[thickline] (2,2) -- (2.5,2); 
\draw[thickline] (1,2.5) -- (1,2.75);

\begin{scope}[xshift=7.5cm,yshift=.5cm]
\draw[Box] (0,-.25) rectangle (1,.25); \node at (.5,0) {$y$}; \node[marked,scale=.75] at (1.05,-.3) {};
\draw[Box] (0,1.25) rectangle (1,1.75); \node at (.5,1.5) {$x$}; \node[marked,scale=.75] at (-.05,1.8) {};
\node[left] at (-.5,.75) {$=  I^{-1/2}n_*^{-1/2}\; \cdot $};
\draw[thickline] (.5,.25) -- node[cpr,scale=.8] (s) {$*$} (.5,1.25); \node[marked,scale=.75,right=.03] at (s.east) {};
\end{scope}
\begin{scope}[xshift=10.25cm]
\draw[Box] (0,0) rectangle (2,1); \node at (1,.5) {$p^{(2)}$}; \node[marked,below right] at (2,0) {};
\draw[Box] (0,1.5) rectangle (2,2.5); \node at (1,2) {$p^{(1)}$}; \node[marked, above left] at (0,2.5) {};
\draw[thickline] (-.5,.5) -- (0,.5); 
\draw[thickline] (2,.5) -- (2.5,.5);
\draw[thickline] (-.5,2) -- (0,2); 
\draw[thickline] (2,2) -- (2.5,2);
\node[left] at (-.9,1.25) {$\cdot$};
\end{scope}
\end{tikzpicture}
\end{equation*}

\item Double capping:
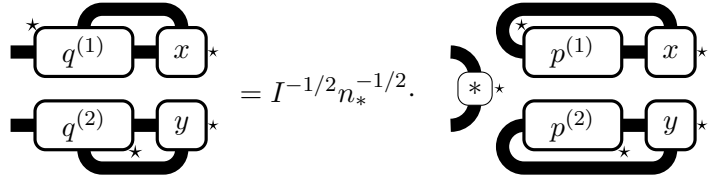
\begin{equation*}
 \begin{tikzpicture}[scale=.65]
\draw[Box] (0,0) rectangle (2,1); \node at (1,.5) {$q^{(2)}$}; \node[marked] at (2.05,-.05) {};
\draw[Box] (0,1.5) rectangle (2,2.5); \node at (1,2) {$q^{(1)}$}; \node[marked] at (-.05,2.55) {};

\draw[thickline] (-.5,.5) -- (0,.5); \draw[thickline] (2,.5) -- (2.5,.5); \draw[thickline] (1,0) arc(180:270:.375cm) -- (2.625,-.375) arc(-90:0:.375cm);
\draw[thickline] (-.5,2) -- (0,2); \draw[thickline] (2,2) -- (2.5,2); \draw[thickline] (1,2.5) arc(180:90:.375cm) -- (2.625,2.875) arc(90:0:.375cm);

\draw[Box] (2.5,1.5) rectangle (3.5,2.5); \node at (3,2) {$x$}; \node[marked,scale=.8,right=.05] at (3.5,2) {};
\draw[Box] (2.5,0) rectangle (3.5,1); \node at (3,.5) {$y$}; \node[marked,scale=.8,right=.05] at (3.5,.5) {};

\begin{scope}[xshift=10cm]
\draw[Box] (0,0) rectangle (2,1); \node at (1,.5) {$p^{(2)}$}; \node[marked,scale=.9] at (2.05,-.05) {};
\draw[Box] (0,1.5) rectangle (2,2.5); \node at (1,2) {$p^{(1)}$}; \node[marked,scale=.9] at (-.05,2.55) {};

 \draw[thickline] (2,.5) -- (2.5,.5); \draw[thickline] (0,.5) arc(90:270:.4375cm) -- (2.625,-.375) arc(-90:0:.375cm);
\draw[thickline] (2,2) -- (2.5,2); \draw[thickline] (0,2) arc(270:90:.4375cm) -- (2.625,2.875) arc(90:0:.375cm);

\draw[Box] (2.5,1.5) rectangle (3.5,2.5); \node at (3,2) {$x$}; \node[marked,scale=.8,right=.05] at (3.5,2) {};
\draw[Box] (2.5,0) rectangle (3.5,1); \node at (3,.5) {$y$}; \node[marked,scale=.8,right=.05] at (3.5,.5) {};

\draw[thickline] (-1.5,2) arc(90:0:.5cm) -- node[cpr] (s) {$*$} (-1,1) arc(0:-90:.5cm); \node[marked,scale=.75,right=.05] at (s.east) {};

\node[left] at (-2,1.25) {$= I^{-1/2}n_*^{-1/2} \cdot$};
\end{scope}

\end{tikzpicture}
\end{equation*}

\item Connecting:

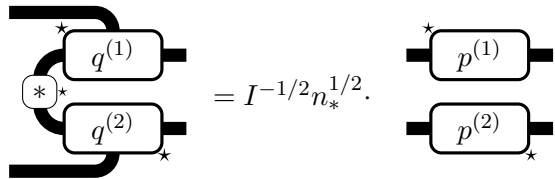
\begin{equation*}
 \begin{tikzpicture}[scale=.65]
\draw[Box] (0,0) rectangle (2,1); \node at (1,.5) {$q^{(2)}$}; \node[marked] at (2.05,-.05) {};
\draw[Box] (0,1.5) rectangle (2,2.5); \node at (1,2) {$q^{(1)}$}; \node[marked] at (-.05,2.55) {};

\draw[thickline]  (0,.5) arc(270:180:.5cm) -- node[cpr] (s) {$*$} (-.5,1.5) arc(180:90:.5cm); \node[marked,scale=.75,right=.03] at (s.east) {};

 \draw[thickline] (2,.5) -- (2.5,.5); \draw[thickline] (1,0) arc (0:-90:.375cm) -- ++(-1.75,0); 
\draw[thickline] (2,2) -- (2.5,2); \draw[thickline] (1,2.5) arc(0:90:.375cm) -- ++(-1.75,0);

\begin{scope}[xshift=7.5cm]
\draw[Box] (0,0) rectangle (2,1); \node at (1,.5) {$p^{(2)}$}; \node[marked,scale=.9] at (2.05,-.05) {};
\draw[Box] (0,1.5) rectangle (2,2.5); \node at (1,2) {$p^{(1)}$}; \node[marked,scale=.9] at (-.05,2.55) {};

 \draw[thickline] (2,.5) -- (2.5,.5); \draw[thickline] (0,.5) -- ++(-.5,0);
\draw[thickline] (2,2) -- (2.5,2); \draw[thickline] (0,2) -- ++(-.5,0);

\node[left] at (-1,1.25) {$=I^{-1/2}n_*^{1/2} \cdot$};
\end{scope}

\end{tikzpicture}
\end{equation*}
\end{enumerate}
\end{proposition}

\begin{proof}
(1) and (2) follow easily from Proposition \ref{qv-def}.  (3) follows from Proposition \ref{qv-def} and Lemma \ref{2cleaver}.  (4) follows from (2) and (3), and then (5) clearly follows from (4).
\end{proof}

The finite-depth assumption has not been essential so far, but the following consequence will be crucial in the remainder of the paper.

\begin{proposition}\label{depthspan}
Let $m \geq k$, then $P_{2m}$ is spanned by elements of the form
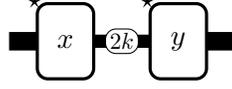
\begin{equation*}
\begin{tikzpicture}[xscale=.75, yscale=.5]
\draw[Box] (0,0) rectangle (1,2); \node at (.5,1) {$x$};
\draw[Box] (2,0) rectangle (3,2); \node at (2.5,1) {$y$};

\draw[verythickline] (0,1) -- (-.5,1); \draw[verythickline] (3,1) -- (3.5,1);
\draw[thickline] (1,1) --node[rcount,scale=.8] {$2k$} (2,1);

\node[marked] at (-.05,2.05) {}; \node[marked] at (1.95,2.05) {};
\end{tikzpicture}
\end{equation*}
for $x,y \in P_{m+k}$.  
\end{proposition}

\begin{proof}
Since the depth of $\mc P$ is less than or equal to $2k$, we have that $P_{2k+1}$ is equal to the basic construction of $P_{2k-1} \subset P_{2k}$.  So $P_{2k+1}$ is generated by $P_{2k}$ and the Jones projection $e_{2k-1}$.  Iterating, we see that $P_{2m}$ is generated by $P_{2k}$ and $e_{2k-1},\dotsc,e_{2m-2}$, from which the result easily follows.
\end{proof}

\begin{proposition} \label{q-connect}
Fix $w,z \in \Gamma_+$, then we have
\begin{equation*}
 \begin{tikzpicture}[scale=.65]
 \draw[Box] (0,0) rectangle (2,1); \node at (1,.5) {$q_{v,w,z}^{(2)}$}; \node[marked, below right] at (2,0) {};
\draw[Box] (0,1.5) rectangle (2,2.5); \node at (1,2) {$q_{v,w,z}^{(1)}$}; \node[marked,above left] at (0,2.5) {};
\draw[thickline] (1,0) -- (1,-.25);
\draw[thickline] (1,2.5) -- (1,2.75);
\draw[thickline] (0,2) arc(90:270:.5cm and .75cm);
\draw[thickline] (2,2) arc(-90:0:.5cm and .75cm);
\draw[thickline] (2,.5) arc(90:0:.5cm and .75cm);
\node[left] at (-.75,1.25) {$\displaystyle\sum_{v \in \Gamma_+} \biggl(\frac{\mu_v}{n_v}\biggr)^{1/2} \;\cdot$};
\begin{scope}[xshift=9.5cm]
\draw[thickline] (0,2.75) -- node[cpr] (z) {$w$} (0,-.25); \node[marked,right=.04cm,scale=.8] at (z.east) {};
\draw[thickline] (1.25,2.75) -- node[cpr] (w) {$z$} (1.25,-.25); \node[marked,right=.04cm,scale=.8] at (w.east) {};
\node[left] at (-.5,1.25) {$\displaystyle = I^{-1/2}\biggl(\frac{\mu_w\mu_z}{n_wn_z}\biggr)^{1/2}\;\cdot$};
\end{scope}

 \end{tikzpicture}
\end{equation*}

\end{proposition}

\begin{proof}
By  Proposition \ref{depthspan}, it suffices to show that
\begin{equation*}
\begin{tikzpicture}[scale=.65]
 \draw[Box] (0,0) rectangle (2,1); \node at (1,.5) {$q_{v,w,z}^{(2)}$}; \node[marked,below right] at (2,0) {};
\draw[Box] (0,1.5) rectangle (2,2.5); \node at (1,2) {$q_{v,w,z}^{(1)}$}; \node[marked,above left] at (0,2.5) {};
\draw[Box] (3,0) rectangle (5,1); \node at (4,.5) {$y$};\node[marked,right,scale=.8] at (5.05,.5) {};
\draw[Box] (3,1.5) rectangle (5,2.5); \node at (4,2) {$x$}; \node[marked,right,scale=.8] at (5.05,2) {};

\draw[thickline] (1,2.5) arc(180:90:.75cm) -- (4,3.25) arc(90:0:.75cm);
\draw[thickline] (2,2) arc(-90:0:.375cm) arc(180:90:.375cm) -- (3.25,2.75) arc(90:0:.25cm);
\draw[thickline] (4,1.5) -- (4,1);
\draw[thickline] (1,0) arc(180:270:.75cm) -- (4,-.75) arc(-90:0:.75cm);
\draw[thickline] (2,.5) arc(90:0:.375cm) arc(180:270:.375cm) -- (3.25,-.25) arc(-90:0:.25cm);

\draw[thickline] (0,2) arc(90:270:.5cm and .75cm);

\node[left] at (-.75,1.25) {$\displaystyle\sum_{v \in \Gamma_+} \biggl(\frac{\mu_v}{n_v}\biggr)^{1/2} \;\cdot$};

\begin{scope}[xshift=9.5cm,yshift = 1.25cm]
\draw[Box] (4,.25) rectangle (5,2.25); \node at (4.5,1.25) {$x$}; \node[marked, scale=.75, left] at (4,2.25) {};
\draw[Box] (4,-.25) rectangle (5,-2.25); \node at (4.5,-1.25) {$y$}; \node[marked,scale =.75, left] at (4,-2.25) {};

\draw[thickline] (4,1.75) -- (3,1.75) arc(90:180:.5cm) -- node[cpr] (w) {$w$} (2.5,-1.25) arc(180:270:.5cm) -- (4,-1.75); \node[marked,scale=.75,right] at (w.east) {};
\draw[thickline]  (4,.75) arc(90:180:.5cm) -- node[cpr] (z) {$z$} (3.5,-.25) arc(180:270:.5cm); \node[marked,scale=.75,right] at (z.east) {};

\draw[thickline] (5,1.25) arc(90:-0:.5cm) -- (5.5,-.75) arc(0:-90:.5cm);

\node[left] at (1.75,0) {$= I^{-1/2} \biggl(\frac{\mu_w\mu_z}{n_wn_z}\biggr)^{1/2}\;$};
\end{scope}

\end{tikzpicture}
\end{equation*}
for any $x,y \in P_{3k}$.

Now by Lemma \ref{2cleaver}, the left hand side of the above equation is equal to
\begin{equation*}
 \begin{tikzpicture}[scale=.65]
 \draw[Box] (0,.25) rectangle (1,1.25); \node at (.5,.75) {$p_v^{(1)}$}; \node[marked,scale=.75,above left] at (0,1.25) {};
\draw[Box] (1.5,.25) rectangle (3.5,1.25); \node at (2.5,.75) {$q_{v,w,z}^{(1)}$}; \node[marked,scale=.75, above left] at (1.5,1.25) {};
\draw[Box] (4,.25) rectangle (5,2.25); \node at (4.5,1.25) {$x$}; \node[marked, scale=.75, left] at (4,2.25) {};
\draw[Box] (0,-.25) rectangle (1,-1.25); \node at (.5,-.75) {$p_v^{(2)}$}; \node[marked, scale=.75, below right] at (1,-1.25) {};
\draw[Box] (1.5,-.25) rectangle (3.5,-1.25); \node at (2.5,-.75) {$q_{v,w,z}^{(2)}$}; \node[marked,scale=.75, below right] at (3.5,-1.25) {};
\draw[Box] (4,-.25) rectangle (5,-2.25); \node at (4.5,-1.25) {$y$}; \node[marked,scale =.75, left] at (4,-2.25) {};

\draw[thickline] (1,.75) -- (1.5,.75);
\draw[thickline] (3.5,.75) -- (4,.75);
\draw[thickline] (2.5,1.25) arc(180:90:.5cm) -- (4,1.75);
\draw[thickline] (0,.75) arc(270:90:.5cm and 1cm) -- (5,2.75) arc(90:-90:.375cm and .75cm);
\draw[thickline] (1,-.75) -- (1.5,.-.75);
\draw[thickline] (3.5,-.75) -- (4,-.75);
\draw[thickline] (2.5,-1.25) arc(180:270:.5cm) -- (4,-1.75);
\draw[thickline] (0,-.75) arc(90:270:.5cm and 1cm) -- (5,-2.75) arc(-90:90:.375cm and .75cm);

\node[left] at (-.75,0) {$\displaystyle \sum_{v \in \Gamma_+} \biggl(\frac{\mu_v}{n_v}\biggr)^{1/2} \cdot \frac{n_v}{\mu_v} \; \cdot$};
 \end{tikzpicture}
\end{equation*}
By Proposition \ref{qv-def} this is equal to
\begin{equation*}
  \begin{tikzpicture}[scale=.65]
 \draw[Box] (0,.25) rectangle (1,1.25); \node at (.5,.75) {$p_v^{(1)}$}; \node[marked,scale=.75,above left] at (0,1.25) {};
\draw[Box] (4,.25) rectangle (5,2.25); \node at (4.5,1.25) {$x$}; \node[marked, scale=.75, left] at (4,2.25) {};
\draw[Box] (0,-.25) rectangle (1,-1.25); \node at (.5,-.75) {$p_v^{(2)}$}; \node[marked, scale=.75, below right] at (1,-1.25) {};
\draw[Box] (4,-.25) rectangle (5,-2.25); \node at (4.5,-1.25) {$y$}; \node[marked,scale =.75, left] at (4,-2.25) {};

\draw[thickline] (4,1.75) -- (3,1.75) arc(90:180:.5cm) -- node[cpr] (w) {$w$} (2.5,-1.25) arc(180:270:.5cm) -- (4,-1.75); \node[marked,scale=.75,right] at (w.east) {};
\draw[thickline] (0,.75) arc(270:90:.5cm and 1cm) -- (5,2.75) arc(90:-90:.375cm and .75cm);

\draw[thickline] (1,-.75) arc(-90:90:.5cm and .75cm);

\draw[thickline]  (4,.75) arc(90:180:.5cm) -- node[cpr] (z) {$z$} (3.5,-.25) arc(180:270:.5cm); \node[marked,scale=.75,right] at (z.east) {};

\draw[thickline] (0,-.75) arc(90:270:.5cm and 1cm) -- (5,-2.75) arc(-90:90:.375cm and .75cm);

\node[left] at (-.75,0) {$\displaystyle \sum_{v \in \Gamma_+} \biggl(\frac{\mu_v}{n_v}\biggr)^{1/2} \cdot \frac{n_v}{\mu_v} \; \cdot I^{-1/2}\cdot \biggl(\frac{\mu_v\mu_w\mu_z}{n_vn_wn_z}\biggr)^{1/2}\;$};
 \end{tikzpicture}
\end{equation*}
The result then follows from (4) of Proposition \ref{p-skein}.
\end{proof}

\begin{corollary} \label{4cleaver}
Fix $v_1,v_2,z_1,z_2 \in \Gamma_+$.  Then if $x,y \in P_{4k}$ we have
\begin{equation*}
\begin{tikzpicture}[scale = .65]
\draw[Box] (0,0) rectangle (2,1); \node at (1,.5) {$q_{v_1,w_1,z}^{(2)}$}; \node[marked, below right, scale =.8] at (2,0) {};
\draw[Box] (0,1.5) rectangle (2,2.5); \node at (1,2) {$q_{v_1,w_1,z}^{(1)}$}; \node[marked, above left, scale=.8] at (0,2.5) {};

\draw[Box] (2.5,0) rectangle (4.5,1); \node at (3.5,.5) {$q_{z,w_2,v_2}^{(2)}$}; \node[marked, below right, scale=.8] at (4.5,0) {};
\draw[Box] (2.5,1.5) rectangle (4.5,2.5); \node at (3.5,2) {$q_{z,w_2,v_2}^{(1)}$}; \node[marked, above left, scale=.8] at (2.5,2.5) {};

\draw[thickline] (2,.5) -- (2.5,.5); \draw[thickline] (2,2) -- (2.5,2);

\draw[Box] (-.75,2.75) rectangle (5.25,3.5); \node at (2.25,3.125) {$x$}; \node[marked,right,scale=.8] at (5.3,3.125) {};
\draw[thickline] (0,2) arc(270:180:.375cm and .75cm); \draw[thickline] (4.5,2) arc(-90:0:.375cm and .75cm);
\draw[thickline] (1,2.5) -- (1,2.75); \draw[thickline] (3.5,2.5) -- (3.5,2.75);

\draw[Box] (-.75,-.25) rectangle (5.25,-1); \node at (2.25,-.625) {$y$}; \node[marked,right,scale=.8] at (5.3,-.625) {};
\draw[thickline] (0,.5) arc(90:180:.375cm and .75cm); \draw[thickline] (4.5,.5) arc(90:0:.375cm and .75cm);
\draw[thickline] (1,0) -- (1,.-.25); \draw[thickline] (3.5,0) -- (3.5,-.25);

\node[left] at (-1,1) {$\displaystyle \sum_{z \in \Gamma_+}$};

\begin{scope}[xshift=14.5cm]
\begin{scope}[xscale=1.25]
\draw[Box] (.25,2.25) rectangle (3.5,3); \node at (1.875,2.625) {$x$}; \node[marked,right,scale=.8] at (3.55,2.625) {};
\draw[Box] (.25,.25) rectangle (3.5,-.5); \node at (1.875, -.125) {$y$}; \node[marked,right,scale=.8] at (3.55,-.125) {};

\draw[thickline] (.75,.25) -- node[cpr,scale=.75] (v1) {$v_1$} (.75,2.25); \node[marked,above left, scale=.6] at (v1.north east) {};
\draw[thickline] (1.5,.25) -- node[cpr,scale=.75] (z1) {$w_1$} (1.5,2.25); \node[marked,above left, scale=.6] at (z1.north east) {};
\draw[thickline] (2.25,.25) -- node[cpr,scale=.75] (z2) {$w_2$} (2.25,2.25); \node[marked,above left, scale=.6] at (z2.north east) {};
\draw[thickline] (3,.25) -- node[cpr,scale=.75] (v2) {$v_2$} (3,2.25); \node[marked,above left, scale=.6] at (v2.north east) {};
\end{scope}
\node[left] at (-.5,1.25) {$\displaystyle = I^{-1} \cdot \biggl(\frac{\mu_{v_1}\mu_{v_2}\mu_{w_1}\mu_{w_2}}{n_{v_1}n_{v_2}n_{w_1}n_{w_2}}\biggr)^{1/2} \; \cdot$};
\end{scope}
\end{tikzpicture}
\end{equation*}

\end{corollary}

\begin{proof}
By Proposition \ref{qv-def}, the left hand side of the equation in the statement of the proposition is equal to
\begin{equation*}
\begin{tikzpicture}[scale = .65]
\draw[Box] (2.5,0) rectangle (4.5,1); \node at (3.5,.5) {$q_{z,w_2,v_2}^{(2)}$}; \node[marked,scale=.8,below right] at (4.5,0) {};
\draw[Box] (2.5,1.5) rectangle (4.5,2.5); \node at (3.5,2) {$q_{z,w_2,v_2}^{(1)}$}; \node[marked,scale=.8,above left] at (2.5,2.5) {};

\draw[thickline] (2.5,.5) arc(270:90:.5cm and .75cm);

\draw[Box] (-.75,2.75) rectangle (5.25,3.5); \node at (2.25,3.125) {$x$}; \node[marked,right,scale=.8] at (5.3,3.125) {};
\draw[Box] (-.75,-.25) rectangle (5.25,-1); \node at (2.25,-.625) {$y$}; \node[marked,right,scale=.8] at (5.3,-.625) {};

\draw[thickline] (0,2.75)  --node[cpr,scale=.8] (v1) {$v_1$} (0,-.25); \draw[thickline] (4.5,2) arc(-90:0:.375cm and .75cm); \node[marked,scale=.6,above left] at (v1.north east) {};
\draw[thickline] (1,2.75) --node[cpr,scale=.8] (z1) {$w_1$} (1,-.25); \draw[thickline] (3.5,2.5) -- (3.5,2.75);\node[marked,scale=.6,above left] at (z1.north east) {};
\draw[thickline] (4.5,.5) arc(90:0:.375cm and .75cm); \draw[thickline] (3.5,0) -- (3.5,-.25);

\node[left,scale=1] at (-.75,1.25) {$\displaystyle \sum_{z \in \Gamma_+} I^{-1/2} \biggl(\frac{\mu_{v_1}\mu_{w_1}\mu_z}{n_{v_1}n_{w_1}n_z}\biggr)^{1/2}\; \cdot$};
\end{tikzpicture}
\end{equation*}
and the result then follows from Proposition \ref{q-connect}.
\end{proof}

\begin{corollary}\label{2q-rotate}
Let $q_1= q_1^{(1)} \otimes q_1^{(2)}$ and $q_2 = q_2^{(1)} \otimes q_2^{(2)}$ be copies of $q$.  Then
\begin{equation*}
\begin{tikzpicture}[scale=.65]
\draw[Box] (0,.25) rectangle (2,1.25); \node at (1,.75) {$q_1^{(1)}$}; \node[marked,above left,scale=.8] at (0,1.25) {};
\draw[Box] (2.5,.25) rectangle (4.5,1.25); \node at (3.5,.75) {$q_2^{(1)}$}; \node[marked,above left,scale=.8] at (2.5,1.25) {};
\draw[Box] (0,-.25) rectangle (2,-1.25); \node at (1,-.75) {$q_1^{(2)}$}; \node[marked, below right,,scale=.8] at (2,-1.25) {};
\draw[Box] (2.5,-.25) rectangle (4.5,-1.25); \node at (3.5,-.75) {$q_2^{(2)}$}; \node[marked, below right,,scale=.8] at (4.5,-1.25) {};

\draw[thickline] (2,.75) -- (2.5,.75); \draw[thickline] (2,-.75) -- (2.5,-.75);
\draw[thickline] (0,.75) -- (-.5,.75); \draw[thickline] (0,-.75) -- (-.5,-.75);
\draw[thickline] (4.5,.75) -- (5,.75); \draw[thickline] (4.5,-.75) -- (5,-.75);
\draw[thickline] (1,1.25) arc(0:90:1cm) -- (-.5,2.25); \draw[thickline] (3.5,1.25) arc(180:90:1cm) -- (5,2.25);
\draw[thickline] (1,-1.25) arc(0:-90:1cm) -- (-.5,-2.25); \draw[thickline] (3.5,-1.25) arc(180:270:1cm) -- (5,-2.25);

\draw[Box] (-.5,-2.75) rectangle (5,2.75); \node[marked,scale=.9,above left] at (-.5,2.75) {};
\begin{scope}[xshift=7.25cm]
\draw[Box] (0,.25) rectangle (2,1.25); \node at (1,.75) {$q_1^{(1)}$}; \node[marked,above left,scale=.8] at (0,1.25) {};
\draw[Box] (0,1.75) rectangle (2,2.75); \node at (1,2.25) {$q_2^{(1)}$}; \node[marked,above left,scale=.8] at (0,2.75) {};
\draw[Box] (0,-.25) rectangle (2,-1.25); \node at (1,-.75) {$q_1^{(2)}$}; \node[marked, below right,,scale=.8] at (2,-1.25) {};
\draw[Box] (0,-1.75) rectangle (2,-2.75); \node at (1,-2.25) {$q_2^{(2)}$}; \node[marked, below right,,scale=.8] at (2,-2.75) {};

\draw[thickline] (0,.75) -- (-.5,.75); \draw[thickline] (0,-.75) -- (-.5,-.75);
\draw[thickline] (2,.75) -- (2.5,.75); \draw[thickline] (2,-.75) -- (2.5,-.75);
\draw[thickline] (1,1.25) -- (1,1.75); \draw[thickline] (1,-1.25) -- (1,-1.75);

\draw[thickline] (0,2.25) -- (-.5,2.25); \draw[thickline] (2,2.25) -- (2.5,2.25);
\draw[thickline] (0,-2.25) -- (-.5,-2.25); \draw[thickline] (2,-2.25) -- (2.5,-2.25);

\draw[Box] (-.5,-3) rectangle (2.5,3); \node[marked,scale=.9, above left] at (-.5,3) {};
\node[left] at (-1,0) {$=$};
\end{scope}

\end{tikzpicture}
\end{equation*}
\end{corollary}

\begin{proof}
It follows from Corollary \ref{4cleaver} and Proposition \ref{q-connect} that
\begin{equation*}
\begin{tikzpicture}[scale=.65]
\draw[Box] (0,.25) rectangle (2,1.25); \node at (1,.75) {$q_1^{(1)}$}; \node[marked,above left,scale=.8] at (0,1.25) {};
\draw[Box] (2.5,.25) rectangle (4.5,1.25); \node at (3.5,.75) {$q_2^{(1)}$}; \node[marked,above left,scale=.8] at (2.5,1.25) {};
\draw[Box] (0,-.25) rectangle (2,-1.25); \node at (1,-.75) {$q_1^{(2)}$}; \node[marked, below right,,scale=.8] at (2,-1.25) {};
\draw[Box] (2.5,-.25) rectangle (4.5,-1.25); \node at (3.5,-.75) {$q_2^{(2)}$}; \node[marked, below right,,scale=.8] at (4.5,-1.25) {};

\draw[thickline] (2,.75) -- (2.5,.75); \draw[thickline] (2,-.75) -- (2.5,-.75);
\draw[thickline] (0,.75) -- (-.5,.75); \draw[thickline] (0,-.75) -- (-.5,-.75);
\draw[thickline] (4.5,.75) -- (5,.75); \draw[thickline] (4.5,-.75) -- (5,-.75);
\draw[thickline] (1,1.25) arc(0:90:1cm) -- (-.5,2.25); \draw[thickline] (3.5,1.25) arc(180:90:1cm) -- (5,2.25);
\draw[thickline] (1,-1.25) arc(0:-90:1cm) -- (-.5,-2.25); \draw[thickline] (3.5,-1.25) arc(180:270:1cm) -- (5,-2.25);

\draw[Box] (-.5,-2.75) rectangle (5,2.75); \node[marked,scale=.9,above left] at (-.5,2.75) {};
\begin{scope}[xshift=7.25cm]
\draw[Box] (0,.25) rectangle (2,1.25); \node at (1,.75) {$q_1^{(1)}$}; \node[marked,below left,scale=.8] at (0,.25) {};
\draw[Box] (0,1.75) rectangle (2,2.75); \node at (1,2.25) {$q_2^{(1)}$}; \node[marked,below left,scale=.8] at (0,1.75) {};
\draw[Box] (0,-.25) rectangle (2,-1.25); \node at (1,-.75) {$q_1^{(2)}$}; \node[marked, above right,,scale=.8] at (2,-.25) {};
\draw[Box] (0,-1.75) rectangle (2,-2.75); \node at (1,-2.25) {$q_2^{(2)}$}; \node[marked, above right,,scale=.8] at (2,-1.75) {};

\draw[thickline] (0,.75) -- (-.5,.75); \draw[thickline] (0,-.75) -- (-.5,-.75);
\draw[thickline] (2,.75) -- (2.5,.75); \draw[thickline] (2,-.75) -- (2.5,-.75);
\draw[thickline] (1,1.25) -- (1,1.75); \draw[thickline] (1,-1.25) -- (1,-1.75);

\draw[thickline] (0,2.25) -- (-.5,2.25); \draw[thickline] (2,2.25) -- (2.5,2.25);
\draw[thickline] (0,-2.25) -- (-.5,-2.25); \draw[thickline] (2,-2.25) -- (2.5,-2.25);

\draw[Box] (-.5,-3) rectangle (2.5,3); \node[marked,scale=.9, above left] at (-.5,3) {};
\node[left] at (-1,0) {$=$};
\end{scope}
\end{tikzpicture}
\end{equation*}
The result then follows from (2) of Proposition \ref{q-skein}.
\end{proof}

\begin{corollary}\label{2q-p}
Let $q_1,q_2$ be copies of $q$.  Then we have
\begin{equation*}
\begin{tikzpicture}[scale=.65]
\draw[Box] (0,.25) rectangle (2,1.25); \node at (1,.75) {$q_1^{(1)}$};\node[marked,above left,scale=.8] at (0,1.25) {};
\draw[Box] (2.5,.25) rectangle (4.5,1.25); \node at (3.5,.75) {$q_2^{(1)}$};  \node[marked,above left,scale=.8] at (2.5,1.25) {};
\draw[Box] (0,-.25) rectangle (2,-1.25); \node at (1,-.75) {$q_1^{(2)}$}; \node[marked, below right,,scale=.8] at (2,-1.25) {};
\draw[Box] (2.5,-.25) rectangle (4.5,-1.25); \node at (3.5,-.75) {$q_2^{(2)}$};  \node[marked, below right,,scale=.8] at (4.5,-1.25) {};

\draw[thickline] (2,.75) -- (2.5,.75); \draw[thickline] (2,-.75) -- (2.5,-.75);
\draw[thickline] (0,.75) -- (-.5,.75); \draw[thickline] (0,-.75) -- (-.5,-.75);
\draw[thickline] (4.5,.75) -- (5,.75); \draw[thickline] (4.5,-.75) -- (5,-.75);
\draw[thickline] (1,1.25) arc(180:90:.5cm) -- (3,1.75) arc(90:0:.5cm); 
\draw[thickline] (1,-1.25) arc(180:270:.5cm) -- (3,-1.75) arc(-90:0:.5cm); 
\begin{scope}[xshift=6.75cm]
\draw[Box] (0,.25) rectangle (2,1.25); \node at (1,.75) {$p^{(1)}$}; \node[marked,scale=.8,above left] at (0,1.25) {};
\draw[Box] (0,-.25) rectangle (2,-1.25); \node at (1,-.75) {$p^{(2)}$}; \node[marked,scale=.8,below right] at (2,-1.25) {};
\draw[thickline] (0,.75) -- (-.5,.75); \draw[thickline] (2,.75) -- (2.5,.75);
\draw[thickline] (0,-.75) -- (-.5,-.75); \draw[thickline] (2,-.75) -- (2.5,-.75);
\node[left] at (-.75,0) {$=$};
\end{scope}
\end{tikzpicture}
\end{equation*}
\end{corollary}

\begin{proof}
By Corollary \ref{2q-rotate} and (3) of Proposition \ref{q-skein}, we have
\begin{equation*}
\begin{tikzpicture}[scale=.65]
\draw[Box] (0,.25) rectangle (2,1.25); \node at (1,.75) {$q_1^{(1)}$};\node[marked,above left,scale=.8] at (0,1.25) {};
\draw[Box] (2.5,.25) rectangle (4.5,1.25); \node at (3.5,.75) {$q_2^{(1)}$};  \node[marked,above left,scale=.8] at (2.5,1.25) {};
\draw[Box] (0,-.25) rectangle (2,-1.25); \node at (1,-.75) {$q_1^{(2)}$}; \node[marked, below right,,scale=.8] at (2,-1.25) {};
\draw[Box] (2.5,-.25) rectangle (4.5,-1.25); \node at (3.5,-.75) {$q_2^{(2)}$};  \node[marked, below right,,scale=.8] at (4.5,-1.25) {};

\draw[thickline] (2,.75) -- (2.5,.75); \draw[thickline] (2,-.75) -- (2.5,-.75);
\draw[thickline] (0,.75) -- (-.5,.75); \draw[thickline] (0,-.75) -- (-.5,-.75);
\draw[thickline] (4.5,.75) -- (5,.75); \draw[thickline] (4.5,-.75) -- (5,-.75);
\draw[thickline] (1,1.25) arc(180:90:.5cm) -- (3,1.75) arc(90:0:.5cm); 
\draw[thickline] (1,-1.25) arc(180:270:.5cm) -- (3,-1.75) arc(-90:0:.5cm); 
\begin{scope}[xshift=7.25cm]
\draw[Box] (0,.25) rectangle (2,1.25); \node at (1,.75) {$q_1^{(1)}$}; \node[marked,above left,scale=.8] at (0,1.25) {};
\draw[Box] (0,1.75) rectangle (2,2.75); \node at (1,2.25) {$q_2^{(1)}$}; \node[marked,above left,scale=.8] at (0,2.75) {};
\draw[Box] (0,-.25) rectangle (2,-1.25); \node at (1,-.75) {$q_1^{(2)}$}; \node[marked, below right,,scale=.8] at (2,-1.25) {};
\draw[Box] (0,-1.75) rectangle (2,-2.75); \node at (1,-2.25) {$q_2^{(2)}$}; \node[marked, below right,,scale=.8] at (2,-2.75) {};

\draw[thickline] (0,.75) -- (-.5,.75); \draw[thickline] (0,-.75) -- (-.5,-.75);
\draw[thickline] (2,.75) -- (2.5,.75); \draw[thickline] (2,-.75) -- (2.5,-.75);
\draw[thickline] (1,1.25) -- (1,1.75); \draw[thickline] (1,-1.25) -- (1,-1.75);

\draw[thickline] (0,2.25) arc(270:90:.5cm) -- (2,3.25) arc(90:-90:.5cm);
\draw[thickline] (0,-2.25) arc(90:270:.5cm) -- (2,-3.25) arc(-90:90:.5cm);

\node[left] at (-1,0) {$=$};
\end{scope}

\begin{scope}[xshift=15.25cm]
\draw[Box] (0,.75) rectangle (2,1.75); \node at (1,1.25) {$q_2^{(1)}$}; \node[marked,above left,scale=.8] at (0,1.75) {};
\draw[Box] (0,-.75) rectangle (2,-1.75); \node at (1,-1.25) {$q_2^{(2)}$}; \node[marked, below right,,scale=.8] at (2,-1.75) {};

\draw[thickline] (1,.75) -- node[cpr] (s) {$*$} (1,-.75); \node[marked,scale=.75,right=.05] at (s.east) {};

\draw[thickline] (0,1.25) arc(270:90:.5cm) -- (2,2.25) arc(90:-90:.5cm);
\draw[thickline] (0,-1.25) arc(90:270:.5cm) -- (2,-2.25) arc(-90:90:.5cm);

\node[left] at (-1,0) {$=I^{-1/2} n_*^{-1/2} \cdot$};

\begin{scope}[xshift=4.25cm]
\draw[Box] (0,.25) rectangle (2,1.25); \node at (1,.75) {$p^{(1)}$}; \node[marked,scale=.8,above left] at (0,1.25) {};
\draw[Box] (0,-.25) rectangle (2,-1.25); \node at (1,-.75) {$p^{(2)}$}; \node[marked,scale=.8,below right] at (2,-1.25) {};
\draw[thickline] (0,.75) -- (-.5,.75); \draw[thickline] (2,.75) -- (2.5,.75);
\draw[thickline] (0,-.75) -- (-.5,-.75); \draw[thickline] (2,-.75) -- (2.5,-.75);
\node[left] at (-.75,0) {$\cdot$};
\end{scope}

\end{scope}
\end{tikzpicture}
\end{equation*}
The result follows by applying (4) of Proposition \ref{q-skein} and (1) of Proposition \ref{p-skein}.
\end{proof}

\section{The Jones tower}

In this section we compute the Jones tower for the inclusion $M_0 \otimes M_0^{op} \subset M_0 \boxtimes M_0^{op}$.  Our construction is easily modified to obtain the Jones tower for the symmetric enveloping inclusion $M_1 \otimes M_1^{op} \subset M_1 \boxtimes M_1^{op}$.  However, we prefer to work with $M_0 \boxtimes M_0^{op}$ as it simplifies the diagrams.  As discussed in Remark \ref{m0vm1} in the following section, this is sufficient to compute the planar algebra of the asymptotic inclusion.

As above, we will assume that $\mc P$ is a finite-depth planar algebra and that $depth(\mc P) \leq 2k$. 

Fix $n \geq 0$.  Recall from Section \ref{sec:background} that we have a graphical representation of $P_{s,(2n+1)k} \otimes P_{t,(2n+1)k}$ as the span of elements of the form
\begin{equation*}
\begin{tikzpicture}[scale=.6]
\draw[Box] (0,0) rectangle (2,1); \node at (1,.5) {$x$};
 \draw[thickline] (-.5,.5) -- (0,.5); \draw[thickline] (2,.5) -- (2.5,.5);
 \draw[verythickline] (1,1) -- (1,1.25); \node[marked,above left, scale=.8] at (0,1) {};
\begin{scope}[rotate around={180:(1,.5)}, yshift=1.5cm]
 \draw[Box] (0,0) rectangle (2,1); \node at (1,.5) {$y$}; \node[marked, below right,scale=.8] at (0,1) {};
 \draw[thickline] (-.5,.5) -- (0,.5); \draw[thickline] (2,.5) -- (2.5,.5);
 \draw[verythickline] (1,1) -- (1,1.25);
\end{scope}
\draw[Box] (-.5,-1.75) rectangle (2.5,1.25); \node[marked] at (-.6,1.35) {};
\node[left=2mm,scale=1.1] at (-.5,-.25) {$x \otimes y =$};
\end{tikzpicture}
\end{equation*}
for $x \in P_{s,(2n+1)k}$ and $y \in P_{t,(2n+1)k}$.  

For $s,t \geq 0$ let $\mc V_{2n+1}(s,t)$ be the subspace of $P_{s,(2n+1)k} \otimes P_{t,(2n+1)k}$ spanned by
\begin{equation*}
\begin{tikzpicture}[xscale=.75, yscale = .6]
\draw[Box] (0,2) rectangle (2,3.5); \node at (1,2.75) {$x$};
\draw[Box] (.5,.5) rectangle (1.5,1.5); \node at (1,1) {$p^{(1)}$}; \node[marked,right=.05,scale=.8] at (1.5,1) {};
\draw[Box] (.5,-.5) rectangle (1.5,-1.5); \node at (1,-1) {$p^{(2)}$}; \node[marked,right=.05,scale=.8] at (1.5,-1) {};
\draw[Box] (0,-2) rectangle (2,-3.5); \node at (1,-2.75) {$y$};

\draw[medthick] (0,2.5) arc(90:270:.25cm and .375cm) arc(90:0:.75cm and .25cm);
\draw[medthick] (2,2.5) arc(90:-90:.25cm and .375cm) arc(90:180:.75cm and .25cm);
\draw[medthick] (.75,.5) arc(0:-90:.25cm) -- (-2,.25);
\draw[medthick] (1.25,.5) arc(180:270:.25cm) -- (4,.25);

\draw[medthick] (0,-2.5) arc(270:90:.25cm and .375cm) arc(-90:0:.75cm and .25cm);
\draw[medthick] (2,-2.5) arc(-90:90:.25cm and .375cm) arc(-90:-180:.75cm and .25cm);
\draw[medthick] (.75,-.5) arc(0:90:.25cm) -- (-2,-.25);
\draw[medthick] (1.25,-.5) arc(180:90:.25cm) -- (4,-.25);

\draw[thickline] (2,3) --node[Box,draw,fill=white](p21) {$p_n^{(1)}$} (4,3); \node[marked,scale=.8,above left] at (p21.north west) {};
\draw[thickline] (2,-3) --node[Box,draw,fill=white] (p22) {$p_n^{(2)}$} (4,-3); \node[marked,scale=.8,below] at (p22.south east) {};
\draw[thickline] (-2,3) --node[Box,draw,fill=white] (p11) {$p_n^{(1)}$} (0,3); \node[marked,scale=.8,above left] at (p11.north west) {};
\draw[thickline] (-2,-3) --node[Box,draw,fill=white] (p12) {$p_n^{(2)}$} (0,-3); \node[marked,scale=.8,below] at (p12.south east) {};
\draw[verythickline] (1,3.5) -- (1,4); \draw[verythickline] (1,-3.5) -- (1,-4);

\node[marked,above left] at (0,3.5) {}; \node[marked,below right] at (2,-3.5) {};
\end{tikzpicture}
\end{equation*}
for $x \in P_{s,(2n+1)k}$, $y \in P_{t,(2n+1)k}$, where $p_n = p_n^{(1)} \otimes p_n^{(2)}$ is defined recursively as follows: $p_0$ is the empty diagram, and for $n \geq 1$ we have
\begin{equation*}
\begin{tikzpicture}[scale=.65]
\draw[Box] (0,.5) rectangle (2,1.5); \node at (1,1) {$p_n^{(1)}$}; \node[marked,scale=.8,above left] at (0,1.5) {};
\draw[Box] (0,-.5) rectangle (2,-1.5); \node at (1,-1) {$p_n^{(2)}$}; \node[marked,scale=.8,below right] at (2,-1.5) {};
\foreach \x in {-.5,2} {\foreach \y in {-1,1} {
\draw[thickline] (\x,\y) -- ++(.5,0);
}}

\begin{scope}[xshift=5cm]
\draw[Box] (0,.25) rectangle (2,1.25); \node at (1,.75) {$p_{n-1}^{(1)}$}; \node[marked,scale=.8,above left] at (0,1.25) {};
\draw[Box] (0,-.25) rectangle (2,-1.25); \node at (1,-.75) {$p_{n-1}^{(2)}$}; \node[marked,scale=.8,below right] at (2,-1.25) {};

\foreach \x in {-.5,2} {\foreach \y in {-.75,.75} {
\draw[thickline] (\x,\y) -- ++(.5,0);
}}

\draw[Box] (0,1.5) rectangle (2,2.5); \node at (1,2) {$p^{(1)}$}; \node[marked, scale=.8, above left] at (0,2.5) {};
\draw[Box] (0,-1.5) rectangle (2,-2.5); \node at (1,-2) {$p^{(2)}$}; \node[marked,scale=.8,below right] at (2,-2.5) {};

\foreach \x in {-.5,2} {\foreach \y in {-2,2} {
\draw[medthick] (\x,\y) -- ++(.5,0);
}}

\node[left] at (-1,0) {$=$};
\end{scope}

\end{tikzpicture}
\end{equation*}
Note that $\mc V_{2n+1}(s,t)$ is closed under the involution on $P_{s,(2n+1)k} \otimes P_{t,(2n+1)k}$.

Fix $t_1,t_2,s_1,s_2 \geq 0$ and define
\begin{equation*}
 \begin{tikzpicture}[yscale=.6,xscale=.75]
\draw[Box] (0,1.25) rectangle (2,2.75); \node at (1,2) {$x_1$};
\draw[Box] (2.5,.75) rectangle (4.5,1.75); \node at (3.5,1.25) {$q^{(1)}$};
\draw[Box] (5,1.25) rectangle (7,2.75); \node at (6,2) {$x_2$};
\draw[thickline] (2,2.25) -- node[rcount] {$2nk$}(5,2.25);
\draw[thickline] (0,2.25) -- (-.5,2.25);
\draw[thickline] (7,2.25) -- (7.5,2.25);
\draw[medthick] (0,1.5) arc(90:270:.25cm) -- (2.5,1);
\draw[medthick] (7,1.5) arc(90:-90:.25cm) -- (4.5,1);
\draw[medthick] (2,1.5) -- (2.5,1.5);
\draw[medthick] (4.5,1.5) -- (5,1.5);

\draw[medthick] (3,.75) arc(0:-90:.25cm) -- (-.5,.5);
\draw[medthick] (4,.75) arc(180:270:.25cm) -- (7.5,.5);

\draw[verythickline] (1,2.75) -- (1,3);
\draw[verythickline] (6,2.75) -- (6,3); 

\node[marked,above left] at (0,2.75) {}; 
\node[marked, above left] at (5,2.75) {}; 

\begin{scope}[yshift=.5cm]
\draw[Box] (0,-1.25) rectangle (2,-2.75); \node at (1,-2) {$y_1$};
\draw[Box] (2.5,-.75) rectangle (4.5,-1.75); \node at (3.5,-1.25) {$q^{(2)}$};
\draw[Box] (5,-1.25) rectangle (7,-2.75); \node at (6,-2) {$y_2$};
\draw[medthick] (3,-.75) arc(0:90:.25cm) -- (-.5,-.5);
\draw[medthick] (4,-.75) arc(180:90:.25cm) -- (7.5,-.5);
\draw[verythickline] (6,-2.75) -- (6,-3);
\draw[verythickline] (1,-2.75) -- (1,-3);
\draw[medthick] (7,-1.5) arc(-90:90:.25cm) -- (4.5,-1);
\draw[medthick] (0,-1.5) arc(270:90:.25cm) -- (2.5,-1);
\draw[medthick] (4.5,-1.5) -- (5,-1.5);
\draw[medthick] (2,-1.5) -- (2.5,-1.5);
\draw[thickline] (2,-2.25) -- node[rcount] {$2nk$} (5,-2.25);
\draw[thickline] (0,-2.25) -- (-.5,-2.25);
\draw[thickline] (7,-2.25) -- (7.5,-2.25);
\node[marked, below right] at (2,-2.75) {};
\node[marked, below right] at (7,-2.75) {};
\end{scope}

\node[left] at (-1,0) {$(x_1 \otimes y_1) \star_q (x_2 \otimes y_2) = $};
 \end{tikzpicture}
\end{equation*}
for $x_i \otimes y_i \in P_{s_i,(2n+1)k}\otimes P_{t_i,(2n+1)k}$.  Extend this map linearly and then restrict to the subspace $\mc V_{2n+1}(s_1,t_1) \otimes \mc V_{2n+1}(s_2,t_2)$ to obtain a ``twisted'' multiplication
\begin{equation*}
\star_q: \mc V_{2n+1}(s_1,t_1) \otimes \mc V_{2n+1}(s_2,t_2) \to \mc V_{2n+1}(s_1+s_2,t_1+t_2).
\end{equation*}
We note that in the rest of the paper we will define linear maps on $\mc V_{2n+1}(s,t)$ by simply describing their action on tensors $x \otimes y$, then implicitly applying this extension/restriction procedure.

Let $\varphi_{2n+1}: \mc V_{2n+1}(s,t) \to \C$ be the linear functional determined by
\begin{equation*}
\begin{tikzpicture}[scale=.5]
\draw[Box] (0,-2) rectangle (2,-.5); \node at (1,-1.25) {$y$}; \draw[verythickline] (1,-2) -- (1,-2.25); \node[marked, below right] at (2,-2) {};
\draw[Box] (0,4.5) rectangle (2,6); \node at (1,5.25) {$x$};\draw [verythickline] (1,6) -- (1,6.25); \node[marked,above left] at (0,6) {};
\draw[Box] (0,6.25) rectangle (2,7); \node[scale=.8] at (1,6.625) {$\sum TL$};
\draw[Box] (0,-2.25) rectangle (2,-3); \node[scale=.8] at (1,-2.625) {$\sum TL$};

\draw[Box] (0,0) rectangle (2,1); \node at (1,.5) {$q^{(2)}$}; \node[marked,scale=.8,left=.05] at (0,.5) {};
\draw[Box] (0,3) rectangle (2,4); \node at (1,3.5) {$q^{(1)}$}; \node[marked,scale=.8,left=.05] at (0,3.5) {};

\draw[medthick] (.25,1) arc(180:0:.75cm and .75cm); \draw[medthick] (.6,1) arc(180:0:.4cm);
\draw[medthick] (.25,3) arc(180:360:.75cm and .75cm); \draw[medthick] (.6,3) arc(180:360:.4cm);

\draw[medthick] (.5,4) arc(0:90:.5cm and .25cm) arc(270:90:.375cm); \draw[medthick] (1.5,4) arc(180:90:.5cm and .25cm) arc (-90:90:.375cm);
\draw[medthick] (.5,0) arc(0:-90:.5cm and .25cm) arc(90:270:.375cm); \draw[medthick] (1.5,0) arc(180:270:.5cm and .25cm) arc(90:-90:.375cm);

\draw[thickline] (0,-1.5) arc(90:270:.75cm and 1cm) -- ++(2,0) arc(-90:90:.75cm and 1cm);
\draw[thickline] (0,5.5) arc(270:90:.75cm and 1cm) -- ++(2,0) arc(90:-90:.75cm and 1cm);
\node[left] at (-1,2) {$\varphi_{2n+1}(x \otimes y) = \displaystyle I^{-n}\; \cdot$};

\begin{scope}[xshift=10.5cm]
\draw[Box] (0,-.5) rectangle (2,1); \node at (1,.25) {$y$}; \draw[verythickline] (1,-.5) -- (1,-.75); \node[marked, below right] at (2,-.5) {};
\draw[Box] (0,3) rectangle (2,4.5); \node at (1,3.75) {$x$};\draw [verythickline] (1,4.5) -- (1,4.75); \node[marked,above left] at (0,4.5) {};
\draw[Box] (.5,1.5) rectangle (1.5,2.5); \node at (1,2) {$*$}; \node[marked,right=.05cm] at (1.5,2) {};

\draw[Box] (0,4.75) rectangle (2,5.5); \node[scale=.8] at (1,5.125) {$\sum TL$};
\draw[Box] (0,-.75) rectangle (2,-1.5); \node[scale=.8] at (1,-1.125) {$\sum TL$};

\draw[medthick] (2,.5) arc(-90:90:.375cm) -- (1.5,1.25) arc(270:180:.25cm);
\draw[medthick] (0,.5) arc(270:90:.375cm) -- (.5,1.25) arc(-90:0:.25cm);
\draw[medthick] (0,3.5) arc(90:270:.375cm) -- (.5,2.75) arc(90:0:.25cm);
\draw[medthick] (2,3.5) arc(90:-90:.375cm) -- (1.5,2.75) arc(90:180:.25cm);

\draw[thickline] (0,0) arc(90:270:.75cm and 1cm) -- ++(2,0) arc(-90:90:.75cm and 1cm);
\draw[thickline] (0,4) arc(270:90:.75cm and 1cm) -- ++(2,0) arc(90:-90:.75cm and 1cm);
\node[left] at (-1,2) {$= \displaystyle I^{-\frac{2n+1}{2}}n_*^{-1/2}\; \cdot$};
\end{scope}
\end{tikzpicture}
\end{equation*}
for $x \in P_{s,(2n+1)k}$ and $y \in P_{t,(2n+1)k}$, where the second equality follows from Proposition \ref{q-skein} (4) and Proposition \ref{p-skein} (1).

Define
\begin{equation*}
 \mc A_{2n+1} = \bigoplus_{s,t \geq 0} \mc V_{2n+1}(s,t).
\end{equation*}
Finally set
\begin{equation*}
\begin{tikzpicture}[xscale=.65, yscale=.55]
\draw[Box] (0,.5) rectangle (2,1.5); \node at (1,1) {$q^{(1)}$}; \node[marked,scale=.8,right=.05] at (2,1) {};
\draw[Box] (0,-.5) rectangle (2,-1.5); \node at (1,-1) {$q^{(2)}$}; \node[marked,scale=.8,right = .05] at (2,-1) {};

\draw[medthick] (.75,1.5) arc(180:0:.25cm);
\draw[medthick] (.4,1.5) arc(180:0:.6cm);
\draw[medthick] (.75,.5) arc(0:-90:.25cm) -- (-.5,.25);
\draw[medthick] (1.25,.5) arc(180:270:.25cm) -- (2.5,.25);

\draw[medthick] (.75,-1.5) arc(180:360:.25cm);
\draw[medthick] (.4,-1.5) arc(180:360:.6cm);
\draw[medthick] (.75,-.5) arc(0:90:.25cm) -- (-.5,-.25);
\draw[medthick] (1.25,-.5) arc(180:90:.25cm) -- (2.5,-.25);
\draw[thickline] (-.5,3) --node[Box,draw,fill=white] (n1) {$p_n^{(1)}$} (2.5,3); \node[marked, above left] at (n1.north west) {};
\draw[thickline] (-.5,-3) --node[Box,draw,fill=white] (n2) {$p_n^{(2)}$} (2.5,-3); \node[marked,below right] at (n2.south east) {};

\node[left] at (-.7,0) {$1_{\mc A_{2n+1}} = $};

\begin{scope}[xshift=7.5cm]
\draw[Box] (.5,-.25) rectangle (1.5,.25); \node at (1,0) {$*$}; \node[marked,right=.04cm, scale =.9] at (1.5,0) {};

\draw[medthick] (.75,.25) arc(0:90:.25cm) -- (-.5,.5); \draw[medthick] (1.25,.25) arc(180:90:.25cm) -- (2.5,.5);
\draw[medthick] (.75,-.25) arc(0:-90:.25cm) -- (-.5,-.5); \draw[medthick] (1.25,-.25) arc(180:270:.25cm) -- (2.5,-.5);

\draw[thickline] (-.5,1.5) --node[Box,draw,fill=white] (n3) {$p_n^{(1)}$} (2.5,1.5); \node[marked, above left] at (n3.north west) {};
\draw[thickline] (-.5,-1.5) --node[Box,draw,fill=white] (n4) {$p_n^{(2)}$} (2.5,-1.5); \node[marked,below right] at (n4.south east) {};

\node[left] at (-.75,0) {$ = I^{1/2} n_*^{-1/2} \; \cdot $};
\node[right] at (2,0) {$,$};
\end{scope}
\end{tikzpicture}
\end{equation*}

Now define $\Psi_n:\mc V_{2n+1}(s,t) \to V_{2nk}(s,t)$ to be the linear map determined by
\begin{equation*}
\begin{tikzpicture}[xscale=.6, yscale = .6]
\draw[Box] (0,-.5) rectangle (2,1); \node at (1,.25) {$y$}; \draw[verythickline] (1,-.5) -- (1,-.75); \node[marked, below right] at (2,-.5) {};
\draw[Box] (0,3) rectangle (2,4.5); \node at (1,3.75) {$x$};\draw [verythickline] (1,4.5) -- (1,4.75); \node[marked,above left] at (0,4.5) {};
\draw[Box] (.5,1.5) rectangle (1.5,2.5); \node at (1,2) {$v$}; \node[marked,right=.05cm] at (1.5,2) {};
\draw[medthick] (2,.5) arc(-90:90:.375cm) -- (1.5,1.25) arc(270:180:.25cm);
\draw[medthick] (0,.5) arc(270:90:.375cm) -- (.5,1.25) arc(-90:0:.25cm);
\draw[medthick] (0,3.5) arc(90:270:.375cm) -- (.5,2.75) arc(90:0:.25cm);
\draw[medthick] (2,3.5) arc(90:-90:.375cm) -- (1.5,2.75) arc(90:180:.25cm);
\draw[thickline] (0,0) -- (-.75,0); \draw[thickline] (2,0) -- (2.75,0);
\draw[thickline] (0,4) -- (-.75,4); \draw[thickline] (2,4) -- (2.75,4);
\node[left] at (-1,2) {$\Psi_n(x \otimes y) = \displaystyle I^{-1/2}\sum_{v \in \Gamma_{+}} \biggl(\frac{\mu_v}{n_v}\biggr)^{1/2}\; \cdot$};
\end{tikzpicture}
\end{equation*}

\begin{proposition}
$\Psi_n$ is an isomorphism of $(\mc A_{2n+1}, \star_q)$ onto $p_n(Gr_{2nk} \boxtimes Gr_{2nk}^{op})p_n$.  We have $\Psi_n(1_{\mc A_{2n+1}}) = p_n$, and 
\begin{equation*}
 \delta^{4nk}I^{-n}(\tau_{2nk} \boxtimes \tau_{2nk}) \circ \Psi_n = \varphi_{2n+1}.
\end{equation*}
\end{proposition}

\begin{proof}
The fact that $\Psi_0$ is injective was proved in \cite{cjs} (note that our $\Psi_0$ was denoted by $\Psi_k$ there, and had a different scaling factor).  The argument given there extends easily to $\Psi_n$ for $n > 0$.  It is clear that $\Psi_n$ respects the involutions on $\mc A_{2n+1}$ and $Gr_{2nk} \boxtimes Gr_{2nk}^{op}$.  That $\Psi_n(1_{\mc A_{2n+1}}) = p_n$ and that $\delta^{4nk}I^{-n}(\tau_{2nk} \boxtimes \tau_{2nk}) \circ \Psi_n = \varphi_{2n+1}$ are clear from the definitions and Lemma \ref{*-proj}.  Surjectivity follows from Lemma \ref{depthspan}.

It remains only to show that $\Psi_n$ is a homomorphism.  We have
\begin{equation*}
\begin{tikzpicture}[xscale=.6, yscale = .6]
\draw[Box] (0,1.25) rectangle (2,2.75); \node at (1,2) {$x_1$};
\draw[Box] (2.5,.75) rectangle (4.5,1.75); \node at (3.5,1.25) {$q^{(1)}$};
\draw[Box] (5,1.25) rectangle (7,2.75); \node at (6,2) {$x_2$};
\draw[thickline] (2,2.25) -- node[rcount] {$2nk$}(5,2.25);
\draw[thickline] (0,2.25) -- (-.5,2.25);
\draw[thickline] (7,2.25) -- (7.5,2.25);
\draw[medthick] (0,1.5) arc(90:270:.25cm) -- (2.5,1);
\draw[medthick] (7,1.5) arc(90:-90:.25cm) -- (4.5,1);
\draw[medthick] (2,1.5) -- (2.5,1.5);
\draw[medthick] (4.5,1.5) -- (5,1.5);

\draw[verythickline] (1,2.75) -- (1,3);
\draw[verythickline] (6,2.75) -- (6,3); 

\node[marked,above left] at (0,2.75) {}; 
\node[marked, above left] at (5,2.75) {}; 

\draw[thickline] (3.5,.75) --node[Box,draw,fill=white,scale=1.25] (v) {$v$} (3.5,-.75); \node[marked,right=.05cm] at (v.east) {};

\draw[Box] (0,-1.25) rectangle (2,-2.75); \node at (1,-2) {$y_1$};
\draw[Box] (2.5,-.75) rectangle (4.5,-1.75); \node at (3.5,-1.25) {$q^{(2)}$};
\draw[Box] (5,-1.25) rectangle (7,-2.75); \node at (6,-2) {$y_2$};
\draw[verythickline] (6,-2.75) -- (6,-3);
\draw[verythickline] (1,-2.75) -- (1,-3);
\draw[medthick] (7,-1.5) arc(-90:90:.25cm) -- (4.5,-1);
\draw[medthick] (0,-1.5) arc(270:90:.25cm) -- (2.5,-1);
\draw[medthick] (4.5,-1.5) -- (5,-1.5);
\draw[medthick] (2,-1.5) -- (2.5,-1.5);
\draw[thickline] (2,-2.25) -- node[rcount] {$2nk$} (5,-2.25);
\draw[thickline] (0,-2.25) -- (-.5,-2.25);
\draw[thickline] (7,-2.25) -- (7.5,-2.25);
\node[marked, below right] at (2,-2.75) {};
\node[marked, below right] at (7,-2.75) {};

\node[left] at (-1,0) {$\Psi_n((x_1 \otimes y_1) \star_q (x_2 \otimes y_2)) = \displaystyle I^{-1/2}\sum_{v \in \Gamma_+} \biggl(\frac{\mu_v}{n_v}\biggr)^{1/2}$};
\end{tikzpicture}
\end{equation*}

which by Proposition \ref{q-connect} is equal to
\begin{equation*}
\begin{tikzpicture}[xscale=.6, yscale = .6]
\draw[Box] (0,1.25) rectangle (2,2.75); \node at (1,2) {$x_1$};
\draw[Box] (4.5,1.25) rectangle (6.5,2.75); \node at (5.5,2) {$x_2$};
\draw[thickline] (2,2.25) -- node[rcount] {$2nk$}(4.5,2.25);
\draw[thickline] (0,2.25) -- (-.5,2.25);
\draw[thickline] (6.5,2.25) -- (7,2.25);

\begin{scope}[yshift=-1.75cm]
\draw[Box] (.5,1.5) rectangle (1.5,2.5); \node at (1,2) {$w$}; \node[marked,right=.05cm] at (1.5,2) {};
\draw[medthick] (2,.5) arc(-90:90:.375cm) -- (1.5,1.25) arc(270:180:.25cm);
\draw[medthick] (0,.5) arc(270:90:.375cm) -- (.5,1.25) arc(-90:0:.25cm);
\draw[medthick] (0,3.5) arc(90:270:.375cm) -- (.5,2.75) arc(90:0:.25cm);
\draw[medthick] (2,3.5) arc(90:-90:.375cm) -- (1.5,2.75) arc(90:180:.25cm);
\end{scope}

\begin{scope}[xshift=4.5cm, yshift=-1.75cm]
\draw[Box] (.5,1.5) rectangle (1.5,2.5); \node at (1,2) {$z$}; \node[marked,right=.05cm] at (1.5,2) {};
\draw[medthick] (2,.5) arc(-90:90:.375cm) -- (1.5,1.25) arc(270:180:.25cm);
\draw[medthick] (0,.5) arc(270:90:.375cm) -- (.5,1.25) arc(-90:0:.25cm);
\draw[medthick] (0,3.5) arc(90:270:.375cm) -- (.5,2.75) arc(90:0:.25cm);
\draw[medthick] (2,3.5) arc(90:-90:.375cm) -- (1.5,2.75) arc(90:180:.25cm);
\end{scope}

\draw[verythickline] (1,2.75) -- (1,3);
\draw[verythickline] (5.5,2.75) -- (5.5,3); 

\node[marked,above left] at (0,2.75) {}; 
\node[marked, above left] at (4.5,2.75) {}; 

\begin{scope}[yshift=.5cm]
\draw[Box] (0,-1.25) rectangle (2,-2.75); \node at (1,-2) {$y_1$};
\draw[Box] (4.5,-1.25) rectangle (6.5,-2.75); \node at (5.5,-2) {$y_2$};
\draw[verythickline] (5.5,-2.75) -- (5.5,-3);
\draw[verythickline] (1,-2.75) -- (1,-3);

\draw[thickline] (2,-2.25) -- node[rcount] {$2nk$} (4.5,-2.25);
\draw[thickline] (0,-2.25) -- (-.5,-2.25);
\draw[thickline] (6.5,-2.25) -- (7,-2.25);
\node[marked, below right] at (2,-2.75) {};
\node[marked, below right] at (6.5,-2.75) {};
\end{scope}

\node[left] at (-1,0) {$I^{-1}\displaystyle\sum_{w,z \in \Gamma_+} \left(\frac{\mu_w\mu_z}{n_wn_z}\right)^{1/2}$};
\end{tikzpicture}
\end{equation*}
This shows that $\Psi_n$ is a homomorphism, which completes the proof.
\end{proof}

The following theorem is now immediate.

\begin{theorem}\label{symmlift}
For $n \geq 0$, $(\mc A_{2n+1},\star_q)$ is a unital, associative $*$-algebra.  Moreover, $\varphi_{2n+1}$ is a faithful, tracial state on $\mc A_{2n+1}$, and $\Psi_n$ extends to an isomorphism of the GNS completion $\mc M_{2n+1}$ onto $p_n(M_{2nk} \boxtimes M_{2nk}^{op})p_n$.  In particular, $\mc M_1$ is isomorphic to $M_0 \boxtimes M_0^{op}$.  \qed
\end{theorem}

Now for $n \geq 0$, define $\mc A_{2n} = p_n(Gr_{2nk} \otimes Gr_{2nk}^{op})p_n$ and let $\varphi_{2n}$ be the renormalized trace 
\begin{equation*}
\varphi_{2n} = \delta^{4nk}I^{-n} \cdot (\tau_{2nk} \otimes \tau_{2nk})|_{\mc A_{2n}}.
\end{equation*}
Let $\mc M_{2n}$ denote the GNS completion of $\mc A_{2n}$, which is naturally isomorphic to $p_n(M_{2nk} \otimes M_{2nk}^{op})p_n$.

Define inclusions $i_{2n}:\mc A_{2n} \hookrightarrow \mc A_{2n+1}$ by
\begin{equation*}
\begin{tikzpicture}[scale=.5]
\draw[Box] (0,.5) rectangle (2,1.5); \node[scale=.8] at (1,1) {$q^{(1)}$};
\draw[Box] (0,-.5) rectangle (2,-1.5); \node[scale=.8] at (1,-1) {$q^{(2)}$};

\draw[medthick] (.75,1.5) arc(180:0:.25cm);
\draw[medthick] (.4,1.5) arc(180:0:.6cm);
\draw[medthick] (.75,.5) arc(0:-90:.25cm) -- (-.5,.25);
\draw[medthick] (1.25,.5) arc(180:270:.25cm) -- (2.5,.25);

\draw[medthick] (.75,-1.5) arc(180:360:.25cm);
\draw[medthick] (.4,-1.5) arc(180:360:.6cm);
\draw[medthick] (.75,-.5) arc(0:90:.25cm) -- (-.5,-.25);
\draw[medthick] (1.25,-.5) arc(180:90:.25cm) -- (2.5,-.25);
\draw[thickline] (-.5,3) --node[Box,draw,fill=white,minimum height=.65cm, minimum width=1cm] (n1) {$x$} (2.5,3); \node[marked, above left] at (n1.north west) {};
\draw[thickline] (-.5,-3) --node[Box,draw,fill=white,minimum height=.65cm, minimum width=1cm] (n2) {$y$} (2.5,-3); \node[marked,below right] at (n2.south east) {};
\draw[thickline] (n1.north) -- ++(0,.5); \draw[thickline] (n2.south) -- ++(0,-.5);

\node[marked,right=.03cm] at (2,1) {}; \node[marked,right=.03cm] at (2,-1) {};

\node[left] at (-.7,0) {$i_{2n}(x \otimes y) = $};

\begin{scope}[xshift=9cm]
\draw[Box] (.5,-.25) rectangle (1.5,.25); \node at (1,0) {$*$}; \node[marked,right=.04cm, scale =.9] at (1.5,0) {};

\draw[medthick] (.75,.25) arc(0:90:.25cm) -- (-.5,.5); \draw[medthick] (1.25,.25) arc(180:90:.25cm) -- (2.5,.5);
\draw[medthick] (.75,-.25) arc(0:-90:.25cm) -- (-.5,-.5); \draw[medthick] (1.25,-.25) arc(180:270:.25cm) -- (2.5,-.5);

\draw[thickline] (-.5,1.5) --node[Box,draw,fill=white,minimum height=.65cm, minimum width=1cm] (n3) {$x$} (2.5,1.5); \node[marked, above left] at (n3.north west) {};
\draw[thickline] (-.5,-1.5) --node[Box,draw,fill=white,minimum height=.65cm, minimum width=1cm] (n4) {$y$} (2.5,-1.5); \node[marked,below right] at (n4.south east) {};
\draw[thickline] (n3.north) -- ++(0,.5); \draw[thickline] (n4.south) -- ++(0,-.5);

\node[left] at (-.75,0) {$ = I^{1/2} n_*^{-1/2} \; \cdot $};
\node[right] at (2,0) {$,$};
\end{scope}
\end{tikzpicture}
\end{equation*}

\begin{proposition}\label{eveninclusion}
For $n \geq 0$, $i_{2n}$ extends to a unital, trace-preserving inclusion of $(\mc M_{2n},\varphi_{2n})$ into $(\mc M_{2n+1},\varphi_{2n+1})$.  Moreover, the following diagram commutes:
\begin{equation*}
\xymatrix{ \mc M_{2n} \ar[r]^{i_{2n}} \ar[d]_{\simeq} & \mc M_{2n+1} \ar[d]^{\Psi_n}\\
p_n\bigl(M_{2nk} \otimes M_{2nk}^{op}\bigr)p_{n} \ar[r] & p_{n}\bigl(M_{2nk} \boxtimes M_{2nk}^{op}\bigr)p_{n}}
\end{equation*}
\end{proposition}

\begin{proof}
It is clear that $\Psi_n \circ i_{2n}$ is the inclusion of $p_n(Gr_{2nk} \otimes Gr_{2nk}^{op})p_n$ into $p_n(Gr_{2nk} \boxtimes Gr_{2nk}^{op})p_n$, from which the result follows.
\end{proof}

We now compute the $\varphi_{2n+1}$-preserving conditional expectation from $\mc M_{2n+1}$ onto $\mc M_{2n}$.  Define $E_{\mc M_{2n}}:\mc V_{2n+1}(s,t) \to P_{s,2nk} \otimes P_{t,2nk}$ by
\begin{equation*}
\begin{tikzpicture}[scale=.5]
\draw[Box] (0,-2) rectangle (2,-.5); \node at (1,-1.25) {$y$}; \draw[verythickline] (1,-2) -- (1,-2.25); \node[marked, below right] at (2,-2) {};
\draw[Box] (0,4.5) rectangle (2,6); \node at (1,5.25) {$x$};\draw [verythickline] (1,6) -- (1,6.25); \node[marked,above left] at (0,6) {};

\draw[Box] (0,0) rectangle (2,1); \node at (1,.5) {$q^{(2)}$}; \node[marked,scale=.8,right=.05] at (2,.5) {};
\draw[Box] (0,3) rectangle (2,4); \node at (1,3.5) {$q^{(1)}$}; \node[marked,scale=.8,right=.05] at (2,3.5) {};

\draw[medthick] (.25,1) arc(180:0:.75cm and .75cm); \draw[medthick] (.6,1) arc(180:0:.4cm);
\draw[medthick] (.25,3) arc(180:360:.75cm and .75cm); \draw[medthick] (.6,3) arc(180:360:.4cm);

\draw[medthick] (.5,4) arc(0:90:.5cm and .25cm) arc(270:90:.375cm); \draw[medthick] (1.5,4) arc(180:90:.5cm and .25cm) arc (-90:90:.375cm);
\draw[medthick] (.5,0) arc(0:-90:.5cm and .25cm) arc(90:270:.375cm); \draw[medthick] (1.5,0) arc(180:270:.5cm and .25cm) arc(90:-90:.375cm);

\draw[thickline] (0,-1.5) -- (-.75,-1.5); \draw[thickline] (2,-1.5) -- (2.75,-1.5);
\draw[thickline] (0,5.5) -- (-.75,5.5); \draw[thickline] (2,5.5) -- (2.75,5.5);
\node[left] at (-1,2) {$E_{\mc M_{2n}}(x \otimes y) = \displaystyle I^{-1}\; \cdot$};

\begin{scope}[xshift=10cm]
\draw[Box] (0,-.5) rectangle (2,1); \node at (1,.25) {$y$}; \draw[verythickline] (1,-.5) -- (1,-.75); \node[marked, below right] at (2,-.5) {};
\draw[Box] (0,3) rectangle (2,4.5); \node at (1,3.75) {$x$};\draw [verythickline] (1,4.5) -- (1,4.75); \node[marked,above left] at (0,4.5) {};
\draw[Box] (.5,1.5) rectangle (1.5,2.5); \node at (1,2) {$*$}; \node[marked,right=.05cm] at (1.5,2) {};
\draw[medthick] (2,.5) arc(-90:90:.375cm) -- (1.5,1.25) arc(270:180:.25cm);
\draw[medthick] (0,.5) arc(270:90:.375cm) -- (.5,1.25) arc(-90:0:.25cm);
\draw[medthick] (0,3.5) arc(90:270:.375cm) -- (.5,2.75) arc(90:0:.25cm);
\draw[medthick] (2,3.5) arc(90:-90:.375cm) -- (1.5,2.75) arc(90:180:.25cm);
\draw[thickline] (0,0) -- (-.75,0); \draw[thickline] (2,0) -- (2.75,0);
\draw[thickline] (0,4) -- (-.75,4); \draw[thickline] (2,4) -- (2.75,4);
\node[left] at (-1,2) {$= \displaystyle I^{-1/2}n_*^{-1/2}\; \cdot$};
\end{scope}
\end{tikzpicture}
\end{equation*}

\begin{proposition}\label{evenjp}
$E_{\mc M_{2n}}$ extends to the unique $\varphi_{2n+1}$-preserving conditional expectation $\mc M_{2n+1} \to \mc M_{2n}$.
\end{proposition}

\begin{proof}
We have
\begin{equation*}
\begin{tikzpicture}[scale=.5]
\draw[Box] (0,-.5) rectangle (2,1); \node at (1,.25) {$y$}; \node[marked,scale=.8, below right] at (2,-.5) {};
\draw[Box] (0,3) rectangle (2,4.5); \node at (1,3.75) {$x$}; \node[marked,scale=.8,above left] at (0,4.5) {};
\draw[Box] (.5,1.5) rectangle (1.5,2.5); \node at (1,2) {$*$}; \node[marked,right=.05cm] at (1.5,2) {};
\draw[medthick] (2,.5) arc(-90:90:.375cm) -- (1.5,1.25) arc(270:180:.25cm);
\draw[medthick] (0,.5) arc(270:90:.375cm) -- (.5,1.25) arc(-90:0:.25cm);
\draw[medthick] (0,3.5) arc(90:270:.375cm) -- (.5,2.75) arc(90:0:.25cm);
\draw[medthick] (2,3.5) arc(90:-90:.375cm) -- (1.5,2.75) arc(90:180:.25cm);
\draw[thickline] (0,0) -- (-.75,0); \draw[thickline] (2,0) arc(90:-90:.75cm and 1.125cm) -- (-2.75,-2.25) arc(270:90:.75cm and 1.125cm);
\draw[thickline] (0,4) -- (-.75,4); \draw[thickline] (2,4) arc(-90:90:.75cm and 1.125cm) -- (-2.75,6.25) arc(90:270:.75cm and 1.125cm);

\draw[Box] (-2.75,3.5) rectangle (-.75,4.5); \node at (-1.75,4) {$a^*$};
\draw[Box] (-2.75,-.5) rectangle (-.75,.5); \node at (-1.75,0) {$b^*$};
\draw[thickline](-1.75,4.5) arc(180:90:.75cm) --node[draw,thick,rounded corners, fill=white,scale=.7] {$\sum TL$} (.25,5.25) arc(90:0:.75cm);
\draw[thickline](-1.75,-.5) arc(180:270:.75cm) --node[draw,thick,rounded corners, fill=white,scale=.7] {$\sum TL$} (.25,-1.25) arc(-90:0:.75cm);

\node[left] at (-4,2) {$\langle i_{2n}(a \otimes b), x \otimes y \rangle = I^{-\tfrac{2n+1}{2}}n_*^{-1/2}\; \cdot$};
\node[right] at (3.5,2) {$= \langle a \otimes b, E_{\mc M_{2n}}(x \otimes y)\rangle$};
\end{tikzpicture}
\end{equation*}
from which the result follows.
\end{proof}

Now define inclusions $i_{2n+1}:\mc A_{2n+1} \to \mc A_{2n+2}$ by
\begin{equation*}
\begin{tikzpicture}[scale=.5]
\draw[Box] (0,1.75) rectangle (2,3.25); \node at (1,2.5) {$x$};
\draw[Box] (0,.25) rectangle (2,1.25); \node at (1,.75) {$q^{(1)}$};
\draw[Box] (0,-.25) rectangle (2,-1.25); \node at (1,-.75) {$q^{(2)}$};
\draw[Box] (0,-1.75) rectangle (2,-3.25); \node at (1,-2.5) {$y$};

\draw[medthick] (0,2.25) arc(90:270:.25cm and .375cm) arc(90:0:.75cm and .25cm);
\draw[medthick] (2,2.25) arc(90:-90:.25cm and .375cm) arc(90:180:.75cm and .25cm);
\draw[medthick] (0,-2.25) arc(270:90:.25cm and .375cm) arc(-90:0:.75cm and .25cm);
\draw[medthick] (2,-2.25) arc(-90:90:.25cm and .375cm) arc(270:180:.75cm and .25cm);

\draw[verythickline] (1,3.25) -- (1,3.5); \draw[verythickline] (1,-3.25) -- (1,-3.5);
\draw[thickline] (0,.75) -- (-.5,.75); \draw[thickline] (2,.75) -- (2.5,.75);
\draw[thickline] (0,-.75) -- (-.5,-.75); \draw[thickline] (2,-.75) -- (2.5,-.75);

\draw[thickline] (0,2.75) -- (-.5,2.75); \draw[thickline] (0,-2.75) -- (-.5,-2.75);
\draw[thickline] (2,2.75) -- (2.5,2.75); \draw[thickline] (2.5,-2.75) -- (2,-2.75);

\node[marked,above left] at (0,3.25) {}; \node[marked, below right] at (2,-3.25) {};
\node[left] at (-1,0) {$i_{2n+1}(x \otimes y) = $};
\end{tikzpicture}
\end{equation*}

\begin{proposition} \label{oddinclusion}
$i_{2n+1}$ extends to a unital, trace preserving embedding of $\mc M_{2n+1}$ into $\mc M_{2n+2}$.  Moreover, the following diagram commutes:
\begin{equation*}
\xymatrix{ \mc M_{2n} \ar[r]^{i_{2n+1} \circ i_{2n}} \ar[d]_{\simeq} & \mc M_{2n+2} \ar[d]^{\simeq}\\
p_n\bigl(M_{2nk} \otimes M_{2nk}^{op}\bigr)p_{n} \ar[r] & p_{n+1}\bigl(M_{(2n+2)k} \otimes M_{(2n+2)k}^{op}\bigr)p_{n+1}
}
\end{equation*}
where the arrow on the bottom row is the obvious inclusion.
\end{proposition}

\begin{proof}
This follows from combining (3) and (4) of Proposition \ref{q-skein} and (1) of Proposition \ref{p-skein}.
\end{proof}

Now define $E_{\mc M_{2n+1}}: \mc A_{2n+2} \to \mc A_{2n+1}$ by
\begin{equation*}
\begin{tikzpicture}[scale=.5]
\draw[Box] (0,1.75) rectangle (2,3.25); \node at (1,2.5) {$x$};
\draw[Box] (0,.5) rectangle (2,1.5); \node at (1,1) {$q^{(1)}$};
\draw[Box] (0,-.5) rectangle (2,-1.5); \node at (1,-1) {$q^{(2)}$};
\draw[Box] (0,-1.75) rectangle (2,-3.25); \node at (1,-2.5) {$y$};

\draw[line width=4pt] (0,1) arc(270:90:.5cm and .625cm); \draw[line width=4pt] (2,1) arc(-90:90:.5cm and .625cm);
\draw[line width=4pt] (0,-1) arc(90:270:.5cm and .625cm); \draw[line width=4pt] (2,-1) arc(90:-90:.5cm and .625cm);

\draw[medthick] (.5,.5) arc(0:-90:.25cm) -- (-.75,.25); \draw[medthick] (1.5,.5) arc(180:270:.25cm) -- (2.75,.25);
\draw[medthick] (.5,-.5) arc(0:90:.25cm) -- (-.75,-.25); \draw[medthick] (1.5,-.5) arc(180:90:.25cm) -- (2.75,-.25);

\draw[verythickline] (1,3.25) -- (1,3.5); \draw[verythickline] (1,-3.25) -- (1,-3.5);

\draw[thickline] (0,2.75) -- (-.75,2.75); \draw[thickline] (0,-2.75) -- (-.75,-2.75);
\draw[thickline] (2,2.75) -- (2.75,2.75); \draw[thickline] (2.75,-2.75) -- (2,-2.75);

\node[marked,above left] at (0,3.25) {}; \node[marked, below right] at (2,-3.25) {};
\node[left] at (-1,0) {$E_{\mc M_{2n+1}}(x \otimes y) = $};
\end{tikzpicture}
\end{equation*}

\begin{proposition}\label{oddjp}
$E_{\mc M_{2n+1}}$ extends to the unique $\varphi_{2n+2}$-preserving conditional expectation of $\mc M_{2n+2}$ onto $\mc M_{2n+1}$.
\end{proposition}

\begin{proof}
Draw the picture and apply Proposition \ref{q-skein} (5) to see that
\begin{equation*}
\langle a \otimes b, E_{\mc M_{2n+1}}(x \otimes y) \rangle = \langle i_{2n+1}(a \otimes b), x \otimes y \rangle,
\end{equation*}
from which the result follows.
\end{proof}

Now define $\e_{2n} \in \mc A_{2n+2}$ by
\begin{equation*}
\begin{tikzpicture}[scale=.75]
\draw[Box] (-.25,.25) rectangle (.75,1.75); \node at (.25,1) {$q^{(1)}_1$};
\draw[Box] (-.25,-.25) rectangle (.75,-1.75); \node at (.25,-1) {$q^{(2)}_1$};
\draw[Box] (2.25,.25) rectangle (3.25,1.75); \node at (2.75,1) {$q^{(1)}_2$};
\draw[Box] (2.25,-.25) rectangle (3.25,-1.75); \node at (2.75,-1) {$q^{(2)}_2$};

\draw[thickline] (-.5,2.5) --node[Box,draw,fill=white] (n1) {$p_n^{(1)}$} (3.5,2.5); \node[marked,above left] at (n1.north west) {};
\draw[thickline] (-.5,-2.5) --node[Box,draw,fill=white] (n2) {$p_n^{(2)}$} (3.5,-2.5); \node[marked, below right] at (n2.south east) {};

\draw[thickline] (-.25,1) -- (-.5,1);
\draw[thickline] (-.25,-1) -- (-.5,-1);
\draw[thickline] (3.25,1) -- (3.5,1);
\draw[thickline] (3.25,-1) -- (3.5,-1);
\draw[medthick] (.75,1.5) arc(90:-90:.5cm); \draw[medthick] (2.25,1.5) arc(90:270:.5cm);
\draw[medthick] (.75,1.25) arc(90:-90:.25cm); \draw[medthick] (2.25,1.25) arc(90:270:.25cm);
\draw[medthick] (.75,-1.5) arc(-90:90:.5cm); \draw[medthick] (2.25,-1.5) arc(270:90:.5cm);
\draw[medthick] (.75,-1.25) arc(-90:90:.25cm); \draw[medthick] (2.25,-1.25) arc(270:90:.25cm);

\node[left] at (-.75,0) {$\displaystyle \e_{2n} = I^{-1} \; \cdot$};

\begin{scope}[xshift=7cm]
 \draw[Box] (0,-.25) rectangle (1,.25); \node at (.5,0) {$*$};
\draw[Box] (1.75,-.25) rectangle (2.75,.25); \node at (2.25,0) {$*$};
\draw[thickline] (-.5,.75) -- (0,.75) arc(90:0:.5cm);
\draw[thickline] (-.5,-.75) -- (0,-.75) arc(-90:0:.5cm);
\draw[thickline] (3.25,.75) -- (2.75,.75) arc(90:180:.5cm);
\draw[thickline] (3.25,-.75) -- (2.75,-.75) arc(270:180:.5cm);

\draw[thickline] (-.5,1.5) --node[Box,draw,fill=white] (n1) {$p_n^{(1)}$} (3.5,1.5); \node[marked,above left] at (n1.north west) {};
\draw[thickline] (-.5,-1.5) --node[Box,draw,fill=white] (n2) {$p_n^{(2)}$} (3.5,-1.5); \node[marked, below right] at (n2.south east) {};

\node[left] at (-.75,0) {$\displaystyle = n_*^{-1}\; \cdot$};
\end{scope}
\end{tikzpicture}
\end{equation*}
where the second equality follows from Proposition \ref{q-skein} (4). 

Also define $\e_{2n+1} \in \mc A_{2n+3}$ by
\begin{equation*}
 \begin{tikzpicture}[scale=.75]
\draw[thickline] (-.5,2.25) --node[Box,draw,fill=white] (n1) {$p_n^{(1)}$} (3.5,2.25); \node[marked,above left] at (n1.north west) {};
\draw[thickline] (-.5,-2.25) --node[Box,draw,fill=white] (n2) {$p_n^{(2)}$} (3.5,-2.25); \node[marked, below right] at (n2.south east) {};

\draw[thickline] (-.5,1) --node[Box,draw,fill=white] (q1) {$q^{(1)}$} (3.5,1); 
\draw[thickline] (-.5,-1) --node[Box,draw,fill=white] (q2) {$q^{(2)}$} (3.5,-1); 

\draw[medthick] (q1.south) ++(-.25,0) arc (0:-90:.25cm) -- ++(-1.5,0);
\draw[medthick] (q1.south) ++(.25,0) arc(180:270:.25cm) -- ++(1.5,0);

\draw[medthick] (q2.north) ++(-.25,0) arc(0:90:.25cm) -- ++(-1.5,0);
\draw[medthick] (q2.north) ++(.25,0) arc(180:90:.25cm) -- ++(1.5,0);

\node[left] at (-.75,0) {$\displaystyle \e_{2n+1} = $};
\end{tikzpicture}
\end{equation*}

\begin{theorem}\label{tower}
$(\mc M_{n}, I^{-1/2}\mathbf e_n)$ is the Jones tower for the subfactor $\mc M_0 \subset \mc M_1$.  
\end{theorem}

\begin{proof}
It follows from Propositions \ref{evenjp} and \ref{oddjp} that $\e_n$ is the Jones projection for $\mc M_n\subset \mc M_{n+1}$.  So it remains to show that $\mc M_{n+2}$ is the basic construction for $\mc M_n \subset \mc M_{n+1}$ for all $n \geq 0$.  

We first prove that $\mc M_{2n+2}$ is the basic construction of $\mc M_{2n} \subset \mc M_{2n+1}$.  Indeed, by a standard result in subfactor theory we have that the subalgebra $\langle \mc M_{2n+1},\e_{2n}\rangle$ of $\mc M_{2n+2}$ which is generated by $\mc M_{2n+1}$ and $\e_{2n}$ is isomorphic to the basic construction of $\mc M_{2n} \subset \mc M_{2n+1}$.  Note that by Proposition \ref{eveninclusion} we have
\begin{align*}
[\mc M_{2n+1}:\mc M_{2n}] &= [p_n(M_{2nk} \boxtimes M_{2nk}^{op})p_n:p_n(M_{2nk} \otimes M_{2nk}^{op})p_n] \\
&= [M_{2nk} \boxtimes M_{2nk}: M_{2nk} \otimes M_{2nk}^{op}] = I.
\end{align*}
It follows that $[\langle \mc M_{2n+1},\e_{2n}\rangle:\mc M_{2n}] = I^2$.  By Proposition \ref{oddinclusion} we have
\begin{align*}
[\mc M_{2n+2}:\mc M_{2n}] &= [p_{n+1}(M_{2(n+1)k} \otimes M_{2(n+1)k}^{op})p_{n+1}:p_n(M_{2nk} \otimes M_{2nk}^{op})p_{n}] \\
&= \biggl(\frac{(\tau_{2(n+1)k} \otimes \tau_{2(n+1)k})(p_{n+1})}{(\tau_{2nk} \otimes \tau_{2nk})(p_n)}\biggr)^2 \cdot [M_{2(n+1)k} \otimes M_{2(n+1)k}:M_{2nk} \otimes M_{2nk}]\\
&= \delta^{-8k}I^2 \cdot \delta^{8k} \\
&= [\langle \mc M_{2n+1},\e_{2n}\rangle:\mc M_{2n}],
\end{align*}
which implies that $\mc M_{2n+2} = \langle \mc M_{2n+1},\e_{2n}\rangle$ is the Jones basic construction as claimed.

It now follows that $\mc M_{2n+3}$ is the basic construction for $\mc M_{2n+1} \subset \mc M_{2n+2}$.  Indeed, we have $\langle M_{2n+2},\e_{2n+1} \rangle \subset \mc M_{2n+3}$.  But since we have shown that $\mc M_{2n+4}$ is the basic construction of $\mc M_{2n+2} \subset \mc M_{2n+3}$, we have
\begin{equation*}
[\mc M_{2n+3}:\mc M_{2n+2}] = [\mc M_{2n+4}:\mc M_{2n+2}]^{1/2} = I = [\langle \mc M_{2n+2},\e_{2n+1} \rangle:\mc M_{2n+2}],
\end{equation*}
so that $\mc M_{2n+3} = \langle \mc M_{2n+2},\e_{2n+1}\rangle$ as desired.
\end{proof}

It follows that the planar algebra $\mc P(\mc M_0 \subset \mc M_1)$ is isomorphic to $\mc P(M_0 \otimes M_0^{op} \subset M_0 \boxtimes M_0^{op})$.  We will now compute the higher relative commutants $\mc M_i ' \cap \mc M_n$ for $i =0,1$, we compute the action of planar tangles in the next section.

\begin{proposition}\label{commutants}
The higher relative commutants are as follows:
\begin{enumerate}
 \item For $n \geq 0$, $\mc M_0' \cap \mc M_{2n}$ is spanned by elements of the form:
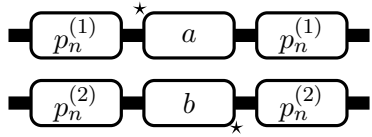
\begin{equation*}
\begin{tikzpicture}[scale=.6]
\draw[Box] (0,.25) rectangle (2,1.25); \node at (1,.75) {$p_n^{(1)}$};
\draw[Box] (0,-.25) rectangle (2,-1.25); \node at (1,-.75) {$p_n^{(2)}$};
\draw[Box] (2.5,.25) rectangle (4.5,1.25); \node at (3.5,.75) {$a$}; \node[marked,above left] at (2.5,1.25) {};
\draw[Box] (2.5,-.25) rectangle (4.5,-1.25); \node at (3.5,-.75) {$b$}; \node[marked,below right] at (4.5,-1.25) {};
\draw[Box] (5,.25) rectangle (7,1.25); \node at (6,.75) {$p_n^{(1)}$};
\draw[Box] (5,-.25) rectangle (7,-1.25); \node at (6,-.75) {$p_n^{(2)}$};

\draw[thickline] (0,.75) -- (-.5,.75); \draw[thickline] (2,.75) -- (2.5,.75);
\draw[thickline] (0,-.75) -- (-.5,-.75); \draw[thickline] (2,-.75) -- (2.5,-.75);
\draw[thickline] (4.5,.75) -- (5,.75); \draw[thickline] (7,.75) -- (7.5,.75);
\draw[thickline] (4.5,-.75) -- (5,-.75); \draw[thickline] (7,-.75) -- (7.5,-.75);
\end{tikzpicture}
\end{equation*}
\item For $n \geq 0$, $\mc M_0' \cap \mc M_{2n+1}$ is spanned by elements of the form:
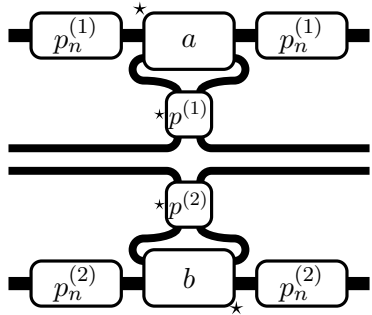
\begin{equation*}
\begin{tikzpicture}[scale=.6]
\begin{scope}[yshift=2cm]
\draw[Box] (0,.25) rectangle (2,1.25); \node at (1,.75) {$p_n^{(1)}$};
\draw[Box] (5,.25) rectangle (7,1.25); \node at (6,.75) {$p_n^{(1)}$};
\draw[Box] (2.5,0) rectangle (4.5,1.25); \node at (3.5,.625) {$a$}; \node[marked,above left] at (2.5,1.25) {};
\draw[thickline] (0,.75) -- (-.5,.75); \draw[thickline] (2,.75) -- (2.5,.75);
\draw[thickline] (4.5,.75) -- (5,.75); \draw[thickline] (7,.75) -- (7.5,.75);
\end{scope}

\draw[Box] (3,.5) rectangle (4,1.5); \node[scale=.9] at (3.5,1) {$p^{(1)}$}; \node[marked,left=.05,scale=.8] at (3,1) {};
\draw[Box] (3,-.5) rectangle (4,-1.5); \node[scale=.9] at (3.5,-1) {$p^{(2)}$}; \node[marked,left=.05,scale=.8] at (3,-1) {};

\draw[medthick] (2.5,2.25) arc(90:270:.25cm) arc(90:0:.75cm and .25cm); \draw[medthick] (4.5,2.25) arc(90:-90:.25cm) arc(90:180:.75cm and .25cm);
\draw[medthick] (2.5,-2.25) arc(270:90:.25cm) arc(-90:0:.75cm and .25cm); \draw[medthick] (4.5,-2.25) arc(-90:90:.25cm) arc(270:180:.75cm and .25cm);

\draw[medthick] (3.25,.5) arc(0:-90:.25cm) -- (-.5,.25); \draw[medthick] (3.75,.5) arc(180:270:.25cm) -- (7.5,.25);
\draw[medthick] (3.25,-.5) arc(0:90:.25cm) -- (-.5,-.25); \draw[medthick] (3.75,-.5) arc(180:90:.25cm) -- (7.5,-.25);

\begin{scope}[yshift=-2cm]
\draw[Box] (2.5,0) rectangle (4.5,-1.25); \node at (3.5,-.625) {$b$}; \node[marked,below right] at (4.5,-1.25) {};
\draw[Box] (0,-.25) rectangle (2,-1.25); \node at (1,-.75) {$p_n^{(2)}$};
\draw[Box] (5,-.25) rectangle (7,-1.25); \node at (6,-.75) {$p_n^{(2)}$};
\draw[thickline] (0,-.75) -- (-.5,-.75); \draw[thickline] (2,-.75) -- (2.5,-.75);
\draw[thickline] (4.5,-.75) -- (5,-.75); \draw[thickline] (7,-.75) -- (7.5,-.75);
\end{scope}
\end{tikzpicture}
\end{equation*}

\item For $n \geq 1$, $\mc M_1' \cap \mc M_{2n}$ is spanned by elements of the form:
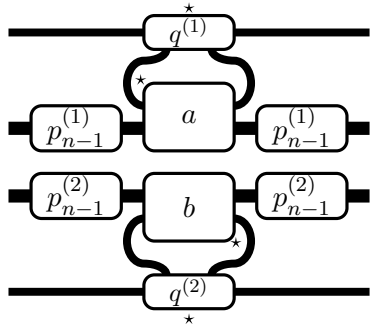
\begin{equation*}
\begin{tikzpicture}[scale=.6]
\draw[Box] (0,.25) rectangle (2,1.25); \node at (1,.75) {$p_{n-1}^{(1)}$};
\draw[Box] (0,-.25) rectangle (2,-1.25); \node at (1,-.75) {$p_{n-1}^{(2)}$};
\draw[Box] (2.5,.25) rectangle (4.5,1.75); \node at (3.5,1) {$a$}; \node[marked,above left,scale=.8] at (2.5,1.75) {};
\draw[Box] (2.5,-.25) rectangle (4.5,-1.75); \node at (3.5,-1) {$b$}; \node[marked,below right,scale=.8] at (4.5,-1.75) {};
\draw[Box] (2.5,2.5) rectangle (4.5,3.25); \node[scale=.9] at (3.5,2.875) {$q^{(1)}$}; \node[marked,scale=.8,above=.07] at (3.5,3.25) {};
\draw[Box] (2.5,-2.5) rectangle (4.5,-3.25); \node[scale=.9] at (3.5,-2.875) {$q^{(2)}$}; \node[marked,scale=.8,below=.03] at (3.5,-3.25) {};

\draw[medthick] (2.5,1.25) arc(270:90:.375cm and .5cm) arc(-90:0:.5cm and .25cm);
\draw[medthick] (4.5,1.25) arc(-90:90:.375cm and .5cm) arc(270:180:.5cm and .25cm);
\draw[medthick] (2.5,-1.25) arc(90:270:.375cm and .5cm) arc(90:0:.5cm and .25cm);
\draw[medthick] (4.5,-1.25) arc(90:-90:.375cm and .5cm) arc(90:180:.5cm and .25cm);

\draw[Box] (5,.25) rectangle (7,1.25); \node at (6,.75) {$p_{n-1}^{(1)}$};
\draw[Box] (5,-.25) rectangle (7,-1.25); \node at (6,-.75) {$p_{n-1}^{(2)}$};

\draw[medthick] (2.5,2.875) -- (-.5,2.875); \draw[medthick] (4.5,2.875) -- (7.5,2.875);
\draw[medthick] (2.5,-2.875) -- (-.5,-2.875); \draw[medthick] (4.5,-2.875) -- (7.5,-2.875);

\draw[thickline] (0,.75) -- (-.5,.75); \draw[thickline] (2,.75) -- (2.5,.75);
\draw[thickline] (0,-.75) -- (-.5,-.75); \draw[thickline] (2,-.75) -- (2.5,-.75);
\draw[thickline] (4.5,.75) -- (5,.75); \draw[thickline] (7,.75) -- (7.5,.75);
\draw[thickline] (4.5,-.75) -- (5,-.75); \draw[thickline] (7,-.75) -- (7.5,-.75);
\end{tikzpicture}
\end{equation*}

\item For $n \geq 1$, $\mc M_1' \cap \mc M_{2n+1}$ is spanned by elements of the form:
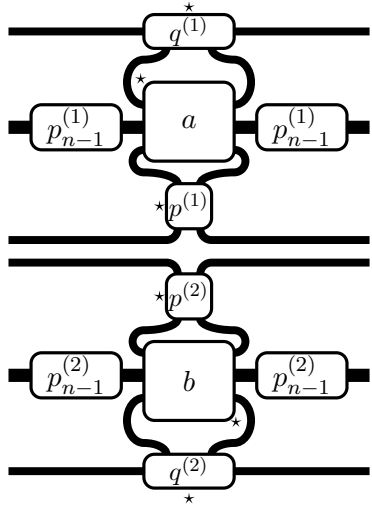
\begin{equation*}
\begin{tikzpicture}[scale=.6]
\begin{scope}[yshift=2cm]
\draw[Box] (0,.25) rectangle (2,1.25); \node at (1,.75) {$p_{n-1}^{(1)}$};

\draw[Box] (2.5,0) rectangle (4.5,1.75); \node at (3.5,.875) {$a$}; \node[marked,above left,scale=.8] at (2.5,1.75) {};
\draw[Box] (2.5,2.5) rectangle (4.5,3.25); \node[scale=.9] at (3.5,2.875) {$q^{(1)}$}; \node[marked,scale=.8,above=.07] at (3.5,3.25) {};

\draw[medthick] (2.5,1.25) arc(270:90:.375cm and .5cm) arc(-90:0:.5cm and .25cm);
\draw[medthick] (4.5,1.25) arc(-90:90:.375cm and .5cm) arc(270:180:.5cm and .25cm);

\draw[Box] (5,.25) rectangle (7,1.25); \node at (6,.75) {$p_{n-1}^{(1)}$};

\draw[medthick] (2.5,2.875) -- (-.5,2.875); \draw[medthick] (4.5,2.875) -- (7.5,2.875);

\draw[thickline] (0,.75) -- (-.5,.75); \draw[thickline] (2,.75) -- (2.5,.75);
\draw[thickline] (4.5,.75) -- (5,.75); \draw[thickline] (7,.75) -- (7.5,.75);
\end{scope}

\draw[Box] (3,.5) rectangle (4,1.5); \node[scale=.9] at (3.5,1) {$p^{(1)}$}; \node[marked,scale=.8,left=.05] at (3,1) {};
\draw[Box] (3,-.5) rectangle (4,-1.5); \node[scale=.9] at (3.5,-1) {$p^{(2)}$}; \node[marked,scale=.8,left=.05] at (3,-1) {};

\draw[medthick] (2.5,2.25) arc(90:270:.25cm) arc(90:0:.75cm and .25cm); \draw[medthick] (4.5,2.25) arc(90:-90:.25cm) arc(90:180:.75cm and .25cm);
\draw[medthick] (2.5,-2.25) arc(270:90:.25cm) arc(-90:0:.75cm and .25cm); \draw[medthick] (4.5,-2.25) arc(-90:90:.25cm) arc(270:180:.75cm and .25cm);

\draw[medthick] (3.25,.5) arc(0:-90:.25cm) -- (-.5,.25); \draw[medthick] (3.75,.5) arc(180:270:.25cm) -- (7.5,.25);
\draw[medthick] (3.25,-.5) arc(0:90:.25cm) -- (-.5,-.25); \draw[medthick] (3.75,-.5) arc(180:90:.25cm) -- (7.5,-.25);

\begin{scope}[yshift=-2cm]
\draw[Box] (0,-.25) rectangle (2,-1.25); \node at (1,-.75) {$p_{n-1}^{(2)}$};
\draw[Box] (2.5,0) rectangle (4.5,-1.75); \node at (3.5,-.875) {$b$}; \node[marked,below right,scale=.8] at (4.5,-1.75) {};
\draw[Box] (2.5,-2.5) rectangle (4.5,-3.25); \node[scale=.9] at (3.5,-2.875) {$q^{(2)}$}; \node[below=.03,scale=.8,marked] at (3.5,-3.25) {};

\draw[medthick] (2.5,-1.25) arc(90:270:.375cm and .5cm) arc(90:0:.5cm and .25cm);
\draw[medthick] (4.5,-1.25) arc(90:-90:.375cm and .5cm) arc(90:180:.5cm and .25cm);

\draw[Box] (5,-.25) rectangle (7,-1.25); \node at (6,-.75) {$p_{n-1}^{(2)}$};

\draw[medthick] (2.5,-2.875) -- (-.5,-2.875); \draw[medthick] (4.5,-2.875) -- (7.5,-2.875);

\draw[thickline] (0,-.75) -- (-.5,-.75); \draw[thickline] (2,-.75) -- (2.5,-.75);
\draw[thickline] (4.5,-.75) -- (5,-.75); \draw[thickline] (7,-.75) -- (7.5,-.75);
\end{scope}

\end{tikzpicture}
\end{equation*}
\end{enumerate}

\end{proposition}

\begin{proof}
Since $\mc M_0 \subset \mc M_{2n}$ is identified with the inclusion $M_0 \otimes M_0^{op} \subset p_n(M_{2nk} \otimes M_{2nk}^{op})p_n$, we have $\mc M_0' \cap \mc M_{2n} = p_n(P_{2nk} \otimes P_{2nk}^{op})p_n$.  This proves (1).

For (3), suppose that $x \in \mc M_1' \cap \mc M_{2n}$.  By (1), $x \in p_n(P_{2nk} \otimes P_{2nk}^{op})p_n$.  Since $x$ commutes with $\mc M_1$, we have
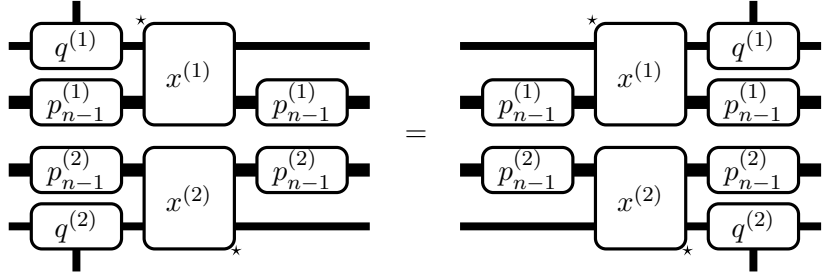
\begin{equation*}
\begin{tikzpicture}[scale=.6]
\draw[Box] (0,.25) rectangle (2,1.25); \node at (1,.75) {$p_{n-1}^{(1)}$};
\draw[Box] (0,1.5) rectangle (2,2.5); \node at (1,2) {$q^{(1)}$};
\draw[Box] (0,-1.5) rectangle (2,-2.5); \node at (1,-2) {$q^{(2)}$};
\draw[Box] (0,-.25) rectangle (2,-1.25); \node at (1,-.75) {$p_{n-1}^{(2)}$};
\draw[Box] (2.5,.25) rectangle (4.5,2.5); \node at (3.5,1.375) {$x^{(1)}$}; \node[marked,above left,scale=.8] at (2.5,2.5) {};
\draw[Box] (2.5,-.25) rectangle (4.5,-2.5); \node at (3.5,-1.375) {$x^{(2)}$}; \node[marked,below right,scale=.8] at (4.5,-2.5) {};

\draw[medthick] (2.5,2) -- (2,2); \draw[medthick] (4.5,2) -- (7.5,2);
\draw[medthick] (2.5,-2) -- (2,-2); \draw[medthick] (4.5,-2) -- (7.5,-2);

\draw[medthick] (0,2) -- (-.5,2); \draw[medthick] (1,2.5) -- (1,3);
\draw[medthick] (0,-2) -- (-.5,-2); \draw[medthick] (1,-2.5) -- (1,-3);

\draw[Box] (5,.25) rectangle (7,1.25); \node at (6,.75) {$p_{n-1}^{(1)}$};
\draw[Box] (5,-.25) rectangle (7,-1.25); \node at (6,-.75) {$p_{n-1}^{(2)}$};

\draw[thickline] (0,.75) -- (-.5,.75); \draw[thickline] (2,.75) -- (2.5,.75);
\draw[thickline] (0,-.75) -- (-.5,-.75); \draw[thickline] (2,-.75) -- (2.5,-.75);
\draw[thickline] (4.5,.75) -- (5,.75); \draw[thickline] (7,.75) -- (7.5,.75);
\draw[thickline] (4.5,-.75) -- (5,-.75); \draw[thickline] (7,-.75) -- (7.5,-.75);

\begin{scope}[xshift = 10cm]
\draw[Box] (0,.25) rectangle (2,1.25); \node at (1,.75) {$p_{n-1}^{(1)}$};
\draw[Box] (5,1.5) rectangle (7,2.5); \node at (6,2) {$q^{(1)}$};
\draw[Box] (5,-1.5) rectangle (7,-2.5); \node at (6,-2) {$q^{(2)}$};
\draw[Box] (0,-.25) rectangle (2,-1.25); \node at (1,-.75) {$p_{n-1}^{(2)}$};
\draw[Box] (2.5,.25) rectangle (4.5,2.5); \node at (3.5,1.375) {$x^{(1)}$}; \node[marked,above left,scale=.8] at (2.5,2.5) {};
\draw[Box] (2.5,-.25) rectangle (4.5,-2.5); \node at (3.5,-1.375) {$x^{(2)}$}; \node[marked,below right,scale=.8] at (4.5,-2.5) {};

\draw[medthick] (4.5,2) -- (5,2); \draw[medthick] (7,2) -- (7.5,2);
\draw[medthick] (4.5,-2) -- (5,-2); \draw[medthick] (7,-2) -- (7.5,-2);

\draw[medthick] (2.5,2) -- (-.5,2); \draw[medthick] (6,2.5) -- (6,3);
\draw[medthick] (2.5,-2) -- (-.5,-2); \draw[medthick] (6,-2.5) -- (6,-3);

\draw[Box] (5,.25) rectangle (7,1.25); \node at (6,.75) {$p_{n-1}^{(1)}$};
\draw[Box] (5,-.25) rectangle (7,-1.25); \node at (6,-.75) {$p_{n-1}^{(2)}$};

\draw[thickline] (0,.75) -- (-.5,.75); \draw[thickline] (2,.75) -- (2.5,.75);
\draw[thickline] (0,-.75) -- (-.5,-.75); \draw[thickline] (2,-.75) -- (2.5,-.75);
\draw[thickline] (4.5,.75) -- (5,.75); \draw[thickline] (7,.75) -- (7.5,.75);
\draw[thickline] (4.5,-.75) -- (5,-.75); \draw[thickline] (7,-.75) -- (7.5,-.75);

\node[left] at (-1,0) {$=$};
\end{scope}

\end{tikzpicture}
\end{equation*}
where we are using Sweedler notation $x = x^{(1)} \otimes x^{(2)}$.
It follows that
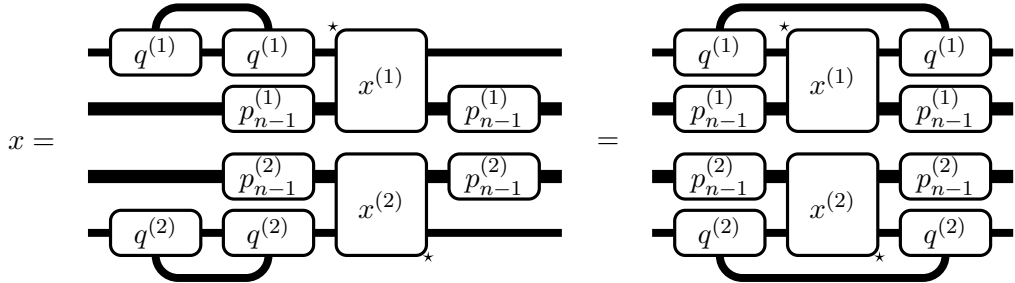
\begin{equation*}
\begin{tikzpicture}[scale=.6]
\draw[Box] (0,.25) rectangle (2,1.25); \node at (1,.75) {$p_{n-1}^{(1)}$};
\draw[Box] (0,1.5) rectangle (2,2.5); \node at (1,2) {$q^{(1)}$};
\draw[Box] (0,-1.5) rectangle (2,-2.5); \node at (1,-2) {$q^{(2)}$};
\draw[Box] (-2.5,1.5) rectangle (-.5,2.5); \node at (-1.5,2) {$q^{(1)}$};
\draw[Box] (-2.5,-1.5) rectangle (-.5,-2.5); \node at (-1.5,-2) {$q^{(2)}$};
\draw[Box] (0,-.25) rectangle (2,-1.25); \node at (1,-.75) {$p_{n-1}^{(2)}$};
\draw[Box] (2.5,.25) rectangle (4.5,2.5); \node at (3.5,1.375) {$x^{(1)}$}; \node[marked,above left,scale=.8] at (2.5,2.5) {};
\draw[Box] (2.5,-.25) rectangle (4.5,-2.5); \node at (3.5,-1.375) {$x^{(2)}$}; \node[marked,below right,scale=.8] at (4.5,-2.5) {};

\draw[medthick] (2.5,2) -- (2,2); \draw[medthick] (4.5,2) -- (7.5,2);
\draw[medthick] (2.5,-2) -- (2,-2); \draw[medthick] (4.5,-2) -- (7.5,-2);

\draw[medthick] (0,2) -- (-.5,2); \draw[medthick] (1,2.5) arc(0:90:.5cm) -- (-1,3) arc(90:180:.5cm);
\draw[medthick] (0,-2) -- (-.5,-2); \draw[medthick] (1,-2.5) arc(0:-90:.5cm) -- (-1,-3) arc(270:180:.5cm);
\draw[medthick] (-2.5,2) -- (-3,2); \draw[medthick] (-2.5,-2) -- (-3,-2);

\draw[Box] (5,.25) rectangle (7,1.25); \node at (6,.75) {$p_{n-1}^{(1)}$};
\draw[Box] (5,-.25) rectangle (7,-1.25); \node at (6,-.75) {$p_{n-1}^{(2)}$};

\draw[thickline] (0,.75) -- (-3,.75); \draw[thickline] (2,.75) -- (2.5,.75);
\draw[thickline] (0,-.75) -- (-3,-.75); \draw[thickline] (2,-.75) -- (2.5,-.75);
\draw[thickline] (4.5,.75) -- (5,.75); \draw[thickline] (7,.75) -- (7.5,.75);
\draw[thickline] (4.5,-.75) -- (5,-.75); \draw[thickline] (7,-.75) -- (7.5,-.75);

\node[left] at (-3.5,0) {$x = $};
\begin{scope}[xshift = 10cm]
\draw[Box] (0,.25) rectangle (2,1.25); \node at (1,.75) {$p_{n-1}^{(1)}$};
\draw[Box] (0,1.5) rectangle (2,2.5); \node at (1,2) {$q^{(1)}$};
\draw[Box] (0,-1.5) rectangle (2,-2.5); \node at (1,-2) {$q^{(2)}$};
\draw[Box] (5,1.5) rectangle (7,2.5); \node at (6,2) {$q^{(1)}$};
\draw[Box] (5,-1.5) rectangle (7,-2.5); \node at (6,-2) {$q^{(2)}$};
\draw[Box] (0,-.25) rectangle (2,-1.25); \node at (1,-.75) {$p_{n-1}^{(2)}$};
\draw[Box] (2.5,.25) rectangle (4.5,2.5); \node at (3.5,1.375) {$x^{(1)}$}; \node[marked,above left,scale=.8] at (2.5,2.5) {};
\draw[Box] (2.5,-.25) rectangle (4.5,-2.5); \node at (3.5,-1.375) {$x^{(2)}$}; \node[marked,below right,scale=.8] at (4.5,-2.5) {};

\draw[medthick] (4.5,2) -- (5,2); \draw[medthick] (7,2) -- (7.5,2);
\draw[medthick] (4.5,-2) -- (5,-2); \draw[medthick] (7,-2) -- (7.5,-2);

\draw[medthick] (2.5,2) -- (2,2); \draw[medthick] (6,2.5) arc(0:90:.5cm) -- (1.5,3) arc(90:180:.5cm);
\draw[medthick] (2.5,-2) -- (2,-2); \draw[medthick] (6,-2.5) arc(0:-90:.5cm) -- (1.5,-3) arc(270:180:.5cm);
\draw[medthick] (0,2) -- (-.5,2); \draw[medthick] (0,-2)-- (-.5,-2);

\draw[Box] (5,.25) rectangle (7,1.25); \node at (6,.75) {$p_{n-1}^{(1)}$};
\draw[Box] (5,-.25) rectangle (7,-1.25); \node at (6,-.75) {$p_{n-1}^{(2)}$};

\draw[thickline] (0,.75) -- (-.5,.75); \draw[thickline] (2,.75) -- (2.5,.75);
\draw[thickline] (0,-.75) -- (-.5,-.75); \draw[thickline] (2,-.75) -- (2.5,-.75);
\draw[thickline] (4.5,.75) -- (5,.75); \draw[thickline] (7,.75) -- (7.5,.75);
\draw[thickline] (4.5,-.75) -- (5,-.75); \draw[thickline] (7,-.75) -- (7.5,-.75);

\node[left] at (-1,0) {$=$};
\end{scope}

\end{tikzpicture}
\end{equation*}
where we have applied Corollary \ref{2q-p} for the first equality.  Applying Corollary \ref{2q-rotate} shows that $x$ is of the desired form, which proves (3).

Finally, since $\mc M_i' \cap \mc M_{2n+1} = E_{\mc M_{2n+1}}(\mc M_i' \cap \mc M_{2n+2})$ for $i = 0,1$, (2) and (4) then follow from (1) and (3) by Proposition \ref{oddjp}.
\end{proof}

As a corollary we can compute the conditional expectation $E_{\mc M_1'}(x)$ for $x$ in $\mc M_0' \cap \mc M_n$. The proof is a straightforward diagrammatic argument and is left to the reader.
\begin{corollary} \label{m1expect}
If $x = x^{(1)} \otimes x^{(2)} \in \mc M_0' \cap \mc M_{2n}$, then 
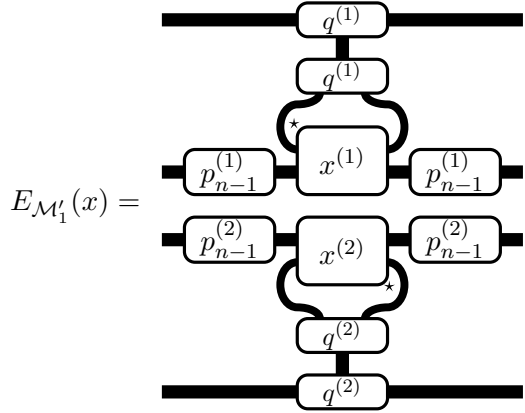
\begin{equation*}
\begin{tikzpicture}[scale=.6]
\draw[Box] (0,.25) rectangle (2,1.25); \node at (1,.75) {$p_{n-1}^{(1)}$};
\draw[Box] (0,-.25) rectangle (2,-1.25); \node at (1,-.75) {$p_{n-1}^{(2)}$};
\draw[Box] (2.5,.25) rectangle (4.5,1.75); \node at (3.5,1) {$x^{(1)}$}; \node[marked,above left,scale=.8] at (2.5,1.75) {};
\draw[Box] (2.5,-.25) rectangle (4.5,-1.75); \node at (3.5,-1) {$x^{(2)}$}; \node[marked,below right,scale=.8] at (4.5,-1.75) {};
\draw[Box] (2.5,2.5) rectangle (4.5,3.25); \node[scale=.9] at (3.5,2.875) {$q^{(1)}$};
\draw[Box] (2.5,3.75) rectangle (4.5,4.5); \node[scale=.9] at (3.5,4.125) {$q^{(1)}$};
\draw[Box] (2.5,-2.5) rectangle (4.5,-3.25); \node[scale=.9] at (3.5,-2.875) {$q^{(2)}$};
\draw[Box] (2.5,-3.75) rectangle (4.5,-4.5); \node[scale=.9] at (3.5,-4.125) {$q^{(2)}$};

\draw[medthick] (2.5,1.25) arc(270:90:.375cm and .5cm) arc(-90:0:.5cm and .25cm);
\draw[medthick] (4.5,1.25) arc(-90:90:.375cm and .5cm) arc(270:180:.5cm and .25cm);
\draw[medthick] (2.5,-1.25) arc(90:270:.375cm and .5cm) arc(90:0:.5cm and .25cm);
\draw[medthick] (4.5,-1.25) arc(90:-90:.375cm and .5cm) arc(90:180:.5cm and .25cm);

\draw[Box] (5,.25) rectangle (7,1.25); \node at (6,.75) {$p_{n-1}^{(1)}$};
\draw[Box] (5,-.25) rectangle (7,-1.25); \node at (6,-.75) {$p_{n-1}^{(2)}$};

\draw[thickline] (2.5,4.125) -- (-.5,4.125); \draw[thickline] (4.5,4.125) -- (7.5,4.125);
\draw[thickline] (2.5,-4.125) -- (-.5,-4.125); \draw[thickline] (4.5,-4.125) -- (7.5,-4.125);
\draw[thickline] (3.5,3.25) -- (3.5,3.75); \draw[thickline] (3.5,-3.25) -- (3.5,-3.75);

\draw[thickline] (0,.75) -- (-.5,.75); \draw[thickline] (2,.75) -- (2.5,.75);
\draw[thickline] (0,-.75) -- (-.5,-.75); \draw[thickline] (2,-.75) -- (2.5,-.75);
\draw[thickline] (4.5,.75) -- (5,.75); \draw[thickline] (7,.75) -- (7.5,.75);
\draw[thickline] (4.5,-.75) -- (5,-.75); \draw[thickline] (7,-.75) -- (7.5,-.75);
\node[left] at (-.75,.0) {$E_{\mc M_{1}'}(x) =$};

\end{tikzpicture}
\end{equation*}
Likewise, if $x = x^{(1)} \otimes x^{(2)} \in \mc M_0' \cap \mc M_{2n+1}$ then 
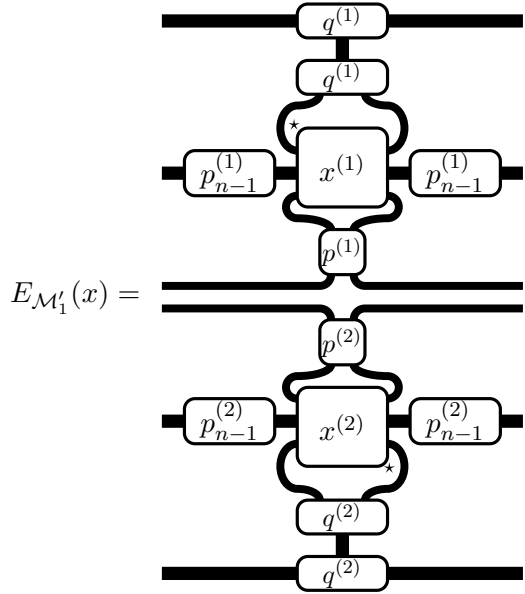
\begin{equation*}
\begin{tikzpicture}[scale=.6]
\begin{scope}[yshift=2cm]
\draw[Box] (0,.25) rectangle (2,1.25); \node at (1,.75) {$p_{n-1}^{(1)}$};

\draw[Box] (2.5,0) rectangle (4.5,1.75); \node at (3.5,.875) {$x^{(1)}$}; \node[marked,above left,scale=.8] at (2.5,1.75) {};
\draw[Box] (2.5,2.5) rectangle (4.5,3.25); \node[scale=.9] at (3.5,2.875) {$q^{(1)}$};
\draw[Box] (2.5,3.75) rectangle (4.5,4.5); \node[scale=.9] at (3.5,4.125) {$q^{(1)}$};
\draw[medthick] (2.5,1.25) arc(270:90:.375cm and .5cm) arc(-90:0:.5cm and .25cm);
\draw[medthick] (4.5,1.25) arc(-90:90:.375cm and .5cm) arc(270:180:.5cm and .25cm);

\draw[Box] (5,.25) rectangle (7,1.25); \node at (6,.75) {$p_{n-1}^{(1)}$};

\draw[thickline] (3.5,3.25) -- (3.5,3.75);
\draw[thickline] (2.5,4.125) -- (-.5,4.125); \draw[thickline] (4.5,4.125) -- (7.5,4.125);

\draw[thickline] (0,.75) -- (-.5,.75); \draw[thickline] (2,.75) -- (2.5,.75);
\draw[thickline] (4.5,.75) -- (5,.75); \draw[thickline] (7,.75) -- (7.5,.75);
\end{scope}

\draw[Box] (3,.5) rectangle (4,1.5); \node[scale=.9] at (3.5,1) {$p^{(1)}$};
\draw[Box] (3,-.5) rectangle (4,-1.5); \node[scale=.9] at (3.5,-1) {$p^{(2)}$};

\draw[medthick] (2.5,2.25) arc(90:270:.25cm) arc(90:0:.75cm and .25cm); \draw[medthick] (4.5,2.25) arc(90:-90:.25cm) arc(90:180:.75cm and .25cm);
\draw[medthick] (2.5,-2.25) arc(270:90:.25cm) arc(-90:0:.75cm and .25cm); \draw[medthick] (4.5,-2.25) arc(-90:90:.25cm) arc(270:180:.75cm and .25cm);

\draw[medthick] (3.25,.5) arc(0:-90:.25cm) -- (-.5,.25); \draw[medthick] (3.75,.5) arc(180:270:.25cm) -- (7.5,.25);
\draw[medthick] (3.25,-.5) arc(0:90:.25cm) -- (-.5,-.25); \draw[medthick] (3.75,-.5) arc(180:90:.25cm) -- (7.5,-.25);

\begin{scope}[yshift=-2cm]
\draw[Box] (0,-.25) rectangle (2,-1.25); \node at (1,-.75) {$p_{n-1}^{(2)}$};
\draw[Box] (2.5,0) rectangle (4.5,-1.75); \node at (3.5,-.875) {$x^{(2)}$}; \node[marked,below right,scale=.8] at (4.5,-1.75) {};
\draw[Box] (2.5,-2.5) rectangle (4.5,-3.25); \node[scale=.9] at (3.5,-2.875) {$q^{(2)}$};
\draw[Box] (2.5,-3.75) rectangle (4.5,-4.5); \node[scale=.9] at (3.5,-4.125) {$q^{(2)}$};

\draw[medthick] (2.5,-1.25) arc(90:270:.375cm and .5cm) arc(90:0:.5cm and .25cm);
\draw[medthick] (4.5,-1.25) arc(90:-90:.375cm and .5cm) arc(90:180:.5cm and .25cm);

\draw[Box] (5,-.25) rectangle (7,-1.25); \node at (6,-.75) {$p_{n-1}^{(2)}$};

\draw[thickline] (3.5,-3.25) -- (3.5,-3.75);
\draw[thickline] (2.5,-4.125) -- (-.5,-4.125); \draw[thickline] (4.5,-4.125) -- (7.5,-4.125);

\draw[thickline] (0,-.75) -- (-.5,-.75); \draw[thickline] (2,-.75) -- (2.5,-.75);
\draw[thickline] (4.5,-.75) -- (5,-.75); \draw[thickline] (7,-.75) -- (7.5,-.75);
\end{scope}
\node[left] at (-.75,.0) {$E_{\mc M_{1}'}(x) =$};

\end{tikzpicture}
\end{equation*}
\end{corollary}\qed

\section{The planar algebra}

In this section we construct the planar algebra $Sym(\mc P)$ of the subfactor $\mc M_0 \subset \mc M_1$.  By Theorem \ref{symmlift}, this is isomorphic to the planar algebra of $M_0 \otimes M_0^{op} \subset M_0 \boxtimes M_0^{op}$.  We explain how to recover the planar algebra of the asymptotic inclusion in Remark \ref{m0vm1} at the end of this section.

It will be convenient to express this in terms of the planar algebra of the inclusion $M_0 \otimes M_0^{op} \subset M_k \otimes M_k^{op}$, which we will now describe.

First we construct the planar algebra of $M_0^{op} \subset M_1^{op}$.  For $n \geq 1$ let $R_n$ be the following tangle:
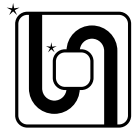
\begin{equation*}
\begin{tikzpicture}[scale = .5]
\draw[Box] (0,0) rectangle (1,1); \node[marked, above left,scale=.75] at (0,1) {};
\draw[Box] (-1,-1) rectangle (2,2); \node[marked, above left,scale=.9] at (-1,2) {};
\draw[thickline] (.5,1) arc(180:0:.5cm) -- (1.5,-1); \draw[thickline] (.5,0) arc(360:180:.5cm) -- (-.5,2);
\end{tikzpicture}
\end{equation*}
where there are $n$ strings on the top and bottom of the input box, and the region adjacent to the left side of the outer box is unshaded.  Let $R_n^*$ be the same tangle, but with the region adjacent to the left side of the inner box unshaded.  Define $P^{rev}_{n}$ to be $P_{n,+}$ if $n$ is even and $P_{n,-}$ if $n$ is odd.  Note that $Z_{R_n}$ maps $P_n^{rev}$ onto $P_n$.  Let $T$ be a tangle with $2n$ marked points on the boundary of its outer disc, and with input discs $D_1,\dotsc,D_m$ such that $D_i$ has $2n_i$ marked points on its boundary.  We then define 
\begin{equation*}
 Z_T^{rev}: P^{rev}_{n_1} \otimes \dotsb \otimes P^{rev}_{n_m} \to P^{rev}_n
\end{equation*}
by
\begin{equation*}
 Z_T^{rev}(x_1 \otimes \dotsb \otimes x_m) = Z_{R_n^*}(Z_T(Z_{R_{n_1}}(x_1) \otimes \dotsb \otimes Z_{R_{n_m}}(x_m))).
\end{equation*}
It is not hard to see that $\mc P^{rev} = (P^{rev}_n)_{n \geq 0}$ is the planar algebra for $M_0^{op} \subset M_1^{op}$, see \cite{palg}.

Now the planar algebra of $M_0 \otimes M_0^{op} \subset M_1 \otimes M_1^{op}$ is the tensor product $\mc P \otimes \mc P^{rev}$.  We refer to \cite{palg} for details, but note that in particular we have $(\mc P \otimes \mc P^{rev})_n = P_n \otimes P^{rev}_n$.  

Finally, the planar algebra of $M_0 \otimes M_0^{op} \subset M_k \otimes M_k^{op}$ is the \textit{$k$-cabling} $\mc C_k(\mc P \otimes \mc P^{rev})$ of the planar algebra of $M_0 \otimes M_0^{op} \subset M_1 \otimes M_1^{op}$.  Again we will refer to \cite{palg} for the definition, but note that $(\mc C_k(\mc P \otimes \mc P^{rev}))_n = P_{nk} \otimes P^{rev}_{nk}$.  In particular, the elements $p$ and $q$ of Section \ref{sec:skein} are now elements of the $2$- and $3$-box spaces of this planar algebra, respectively.  Let us now interpret some of the skein relations from that section in terms of this planar algebra.

\begin{proposition}\label{cabled_skein}
The following skein relations hold in $\mc C_k(\mc P \otimes \mc P^{rev})$:
\begin{enumerate}
 \item 
\begin{equation*}
 \begin{tikzpicture}[very thick,scale=.9]
 \draw[Box,thin] (-.25,-.25) rectangle (.25,.25); \node[scale=.9] at (0,0) {$p$}; \node[marked,above left, scale=.8] at (-.25,.25) {};
\draw (-.5,-.1) -- (-.25,-.1); \draw(-.5,.1) -- (-.25,.1);
\draw (.25,.1) -- (.5,.1); \draw (.25,-.1) -- (.5,-.1);
\draw[Box] (-.5,-.5) rectangle (.5,.5); \node[marked,above left] at (-.5,.5) {};
\begin{scope}[xshift=1.75cm]
 \draw[Box,thin] (-.25,-.25) rectangle (.25,.25); \node[scale=.9] at (0,0) {$p$}; \node[marked,below right, scale=.8] at (.25,-.25) {};
\draw (-.5,-.1) -- (-.25,-.1); \draw(-.5,.1) -- (-.25,.1);
\draw (.25,.1) -- (.5,.1); \draw (.25,-.1) -- (.5,-.1); 
\draw[Box] (-.5,-.5) rectangle (.5,.5); \node[marked,above left] at (-.5,.5) {};

\node[left] at (-.55,0) {$=$};
\end{scope}
\begin{scope}[xshift=6cm]
 \draw[Box,thin] (-.25,-.25) rectangle (.25,.25); \node[scale=.9] at (0,0) {$q$}; \node[marked,above left, scale=.8] at (-.25,.25) {};
\draw (-.5,-.1) -- (-.25,-.1); \draw(-.5,.1) -- (-.25,.1);
\draw (.25,.1) -- (.5,.1); \draw (.25,-.1) -- (.5,-.1);
\draw (-.1,.25) -- (-.1,.5); \draw (.1,.25) -- (.1,.5);
\draw[Box] (-.5,-.5) rectangle (.5,.5); \node[marked,above left] at (-.5,.5) {};
\begin{scope}[xshift=1.75cm]
 \draw[Box,thin] (-.25,-.25) rectangle (.25,.25); \node[scale=.9] at (0,0) {$q$}; \node[marked,above right, scale=.8] at (.25,.25) {};
\draw (-.5,-.1) -- (-.25,-.1); \draw(-.5,.1) -- (-.25,.1);
\draw (.25,.1) -- (.5,.1); \draw (.25,-.1) -- (.5,-.1); 
\draw (-.1,.25) -- (-.1,.5); \draw (.1,.25) -- (.1,.5);

\draw[Box] (-.5,-.5) rectangle (.5,.5); \node[marked,above left] at (-.5,.5) {};
\node[left] at (-.55,0) {$=$};
\end{scope}
\node at (-2.25,0) {and};
\end{scope}
 \end{tikzpicture}
\end{equation*}
\item
\begin{equation*}
\begin{tikzpicture}[very thick]
\draw[Box,thin] (-.25,-.25) rectangle (.25,.25); \node[scale=.9] at (0,0) {$p$};
\draw[Box,thin] (.5,-.25) rectangle (1,.25); \node[scale=.9] at (.75,0) {$q$};
\draw (-.5,-.1) -- (-.25,-.1); \draw(-.5,.1) -- (-.25,.1);
\draw (.25,.1) -- (.5,.1); \draw (.25,-.1) -- (.5,-.1);
\draw (1,.1) -- (1.25,.1); \draw (1,-.1) -- (1.25,-.1);
\draw (.65,.25) -- (.65,.5); \draw (.85,.25) -- (.85,.5);
\draw[Box] (-.5,-.4) rectangle (1.25,.5); \node[marked,scale=.9,above left] at (-.5,.5) {};
\begin{scope}[xshift=2.15cm]
\draw[Box,thin] (.5,-.25) rectangle (1,.25); \node[scale=.9] at (.75,0) {$q$};
\draw (.25,.1) -- (.5,.1); \draw (.25,-.1) -- (.5,-.1);
\draw (1,.1) -- (1.25,.1); \draw (1,-.1) -- (1.25,-.1);
\draw (.65,.25) -- (.65,.5); \draw (.85,.25) -- (.85,.5);
\draw[Box] (.25,-.4) rectangle (1.25,.5);  \node[marked,scale=.9,above left] at (.25,.5) {};
\node[left] at (0,0) {$=$};
\end{scope}
\begin{scope}[xshift=5cm]
\draw[Box,thin] (-.25,-.25) rectangle (.25,.25); \node[scale=.9] at (0,0) {$q$};
\draw[Box,thin] (.5,-.25) rectangle (1,.25); \node[scale=.9] at (.75,0) {$p$};
\draw (-.5,-.1) -- (-.25,-.1); \draw(-.5,.1) -- (-.25,.1);
\draw (.25,.1) -- (.5,.1); \draw (.25,-.1) -- (.5,-.1);
\draw (1,.1) -- (1.25,.1); \draw (1,-.1) -- (1.25,-.1);
\draw (-.1,.25) -- (-.1,.5); \draw (.1,.25) -- (.1,.5);
\draw[Box] (-.5,-.4) rectangle (1.25,.5);\node[marked,scale=.9,above left] at (-.5,.5) {};
\node[left] at (-.75,0) {$=$};
\end{scope}
\end{tikzpicture}
\end{equation*}

\item
\begin{equation*}
\begin{tikzpicture}[very thick]
\draw[Box,thin] (0,0) rectangle (.5,.75); \node at (.25,.375) {$q$}; 
\draw[Box,thin] (1,0) rectangle (1.5,.75); \node at (1.25,.375) {$q$};
\draw (.5,.6) -- (1,.6); \draw (.5,.45) -- (1,.45);
\draw (.5,.3) -- (1,.3); \draw (.5,.15) -- (1,.15);
\draw (0,.25) -- (-.25,.25); \draw (0,.5) -- (-.25,.5);
\draw (1.5,.25) -- (1.75,.25); \draw (1.5,.5) -- (1.75,.5);
\draw[Box] (-.25,-.15) rectangle (1.75,.9); \node[marked,scale=.9,above left] at (-.25,.9) {};
\begin{scope}[xshift=3.35cm,yshift=.375cm]
 \draw[Box,thin] (-.25,-.25) rectangle (.25,.25); \node[scale=.9] at (0,0) {$p$};
\draw (-.5,-.1) -- (-.25,-.1); \draw(-.5,.1) -- (-.25,.1);
\draw (.25,.1) -- (.5,.1); \draw (.25,-.1) -- (.5,-.1);
\draw[Box] (-.5,-.4) rectangle (.5,.4); \node[marked,scale=.9,above left] at (-.5,.4) {};
\node[left] at (-.75,0) {$=$};
\end{scope}
\end{tikzpicture}
\end{equation*}
\item
\begin{equation*}
\begin{tikzpicture}[very thick]
\draw[Box,thin] (-.25,-.25) rectangle (.25,.25); \node[scale=.9] at (0,0) {$q$};
\draw[Box,thin] (.5,-.25) rectangle (1,.25); \node[scale=.9] at (.75,0) {$q$};
\draw (-.5,-.1) -- (-.25,-.1); \draw(-.5,.1) -- (-.25,.1);
\draw (.25,.1) -- (.5,.1); \draw (.25,-.1) -- (.5,-.1);
\draw (1,.1) -- (1.25,.1); \draw (1,-.1) -- (1.25,-.1);
\draw (-.1,.25) -- (-.1,.5); \draw (.1,.25) -- (.1,.5);
\draw (.65,.25) -- (.65,.5); \draw (.85,.25) -- (.85,.5);
\draw[Box] (-.5,-.4) rectangle (1.25,.5); \node[marked,scale=.9,above left] at (-.5,.5) {};
\begin{scope}[xshift=3.125cm,yshift=-.375cm]
 \draw[Box,thin] (-.25,-.25) rectangle (.25,.25); \node[scale=.9] at (0,0) {$q$};
\draw[Box,thin] (-.375,.5) rectangle (.375,1); \node[scale=.9] at (0,.75) {$q$};
\draw (-.5,-.1) -- (-.25,-.1); \draw(-.5,.1) -- (-.25,.1);
\draw (.25,.1) -- (.5,.1); \draw (.25,-.1) -- (.5,-.1);
\draw (-.1,.25) -- (-.1,.5); \draw (.1,.25) -- (.1,.5);
\draw (-.225,1) -- (-.225,1.25); \draw (-.075,1) -- (-.075,1.25);
\draw (.075,1) -- (.075,1.25); \draw (.225,1) -- (.225,1.25);
\draw[Box] (-.5,-.4) rectangle (.5,1.25); \node[marked,scale=.9,above left] at (-.5,1.25) {};
\node[left] at (-.875,.375) {$=$};
\end{scope}
\end{tikzpicture}
\end{equation*}

\item For any $x \in (\mc C_k(\mc P \otimes \mc P^{rev}))_{2}$ we have
\begin{equation*}
\begin{tikzpicture}[very thick]
\draw[Box,thin] (0,0) rectangle (.5,.75); \node at (.25,.375) {$q$}; 
\draw[Box,thin] (.75,0) rectangle (1.25,.75); \node at (1,.375) {$q$};
\draw[Box,thin] (1.5,0) rectangle (2,.75); \node at (1.75,.375) {$x$}; \node[marked,scale=.7,below left] at (1.5,0) {};
\foreach \x in {-.25,1.25} {
\foreach \y in {.125,.25,.5,.625} {
\draw (\x,\y) -- (\x+.25,\y);
}}
 \draw (.5,.45) -- (.75,.45); \draw (.5,.3) -- (.75,.3); 

\draw[Box] (-.25,-.15) rectangle (2.15,.9); \node[marked,scale=.9,above left] at (-.25,.9) {};

\begin{scope}[xshift=3.75cm,yshift=.375cm]

\begin{scope}[xshift=-.125cm]
\draw  (-.0625,0) -- ++(0,.25)  arc(0:90:.125cm) -- (-.375,.375); \draw (.0625,0) -- ++ (0,.25) arc(0:90:.25cm) -- (-.375,.5);
\draw (-.0625,0) -- ++ (0,-.25) arc(0:-90:.125cm) -- (-.375,-.375); \draw (.0625,0) -- ++ (0,-.25) arc(0:-90:.25cm) -- (-.375,-.5);
\node[cpr,draw,fill=white,scale=.8] at (0,0) {$p$};
\end{scope}

\draw[Box,thin] (.75,-.625) rectangle (1.25,.625); \node at (1,0) {$x$}; \node[marked,scale=.7,below left] at (.75,-.625) {};
\begin{scope}[xshift=.375cm,xscale=-1]
\draw  (-.0625,0) -- ++(0,.25)  arc(0:90:.125cm) -- (-.375,.375); \draw (.0625,0) -- ++ (0,.25) arc(0:90:.25cm) -- (-.375,.5);
\draw (-.0625,0) -- ++ (0,-.25) arc(0:-90:.125cm) -- (-.375,-.375); \draw (.0625,0) -- ++ (0,-.25) arc(0:-90:.25cm) -- (-.375,-.5);
\node[cpr,draw,fill=white,scale=.8] at (0,0) {$p$};
\end{scope}

\draw[Box] (-.5,-.75) rectangle (1.375,.75); \node[marked,scale=.9,above left] at (-.5,.75) {};
\node[left] at (-.75,0) {$=$};
\end{scope}
\end{tikzpicture}
\end{equation*}

\end{enumerate}

\end{proposition}
Because of (1) we will usually not label the marked interval for $p$ and $q$, but assume that it occurs at one of the corners.  Note that (5) follows from combining (4) and (5) of Proposition \ref{q-skein}.  The following important relation follows from (3) and (4) above.

\begin{lemma}\label{daisychain}
For $n \geq 3$ we have
\begin{equation*}
 \begin{tikzpicture}[scale=.75, very thick]
 \foreach \x in {0,.75,1.5,3,3.75}{
  \draw[Box,thin] (\x,0) rectangle (\x+.5,.5); \node at (\x+.25,.25) {$q$};
\draw (\x+.2,0) -- (\x +.2,-.25); \draw (\x +.3,0) -- (\x + .3,-.25);
}
\foreach \x in {.5,1.25,2,2.75,3.5}{ \foreach \y in {.2,.3}{
\draw (\x,\y) -- (\x+.25,\y);
}}
\draw (0,.3) arc(270:90:.2cm) -- (4.25,.7) arc(90:-90:.2cm);
\draw (0,.2) arc(270:90:.3cm) -- (4.25,.8) arc(90:-90:.3cm);
\node[scale=.6] at (2.5,.25) {$\dotsb$};

\draw[Box] (-.5,-.25) rectangle (4.75,1); \node[marked,above left=.01cm] at (-.5,1) {};
\begin{scope}[xshift=6.25cm]
 \foreach \x in {.75,1.5,3,3.75}{
  \draw[Box,thin] (\x,0) rectangle (\x+.5,.5); \node at (\x+.25,.25) {$q$};
\draw (\x+.2,0) -- (\x +.2,-.25); \draw (\x +.3,0) -- (\x + .3,-.25);
}
\foreach \x in {1.25,2,2.75,3.5}{ \foreach \y in {.2,.3}{
\draw (\x,\y) -- (\x+.25,\y);
}}
\draw (.75,.3) arc(90:180:.3cm) -- (.45,-.25); \draw (.75,.2) arc(90:180:.2cm) -- (.55,-.25);

\draw (.2,-.25) -- (.2,.5) arc(180:90:.2cm) -- (4.25,.7) arc(90:-90:.2cm);
\draw (.1,-.25) -- (.1,.5) arc(180:90:.3cm) -- (4.25,.8) arc(90:-90:.3cm);
\node[scale=.6] at (2.5,.25) {$\dotsb$};

\draw[Box] (-.25,-.25) rectangle (4.75,1); \node[marked,above left=.01cm] at (-.25,1) {};
\node[left] at (-.5,.375) {$=$};
\end{scope}
\end{tikzpicture}
\end{equation*}
where there are $n$ $q$'s appearing on the left hand side of the equation, and $n-2$ on the right hand side.
\end{lemma}

\begin{proof}
By applying relation (4) to the second and third $q$ from left, we see that the left hand side is equal to
\begin{equation*}
 \begin{tikzpicture}[scale=.75, very thick]
 \foreach \x in {0,2,3.5,4.25}{
  \draw[Box,thin] (\x,0) rectangle (\x+.5,.5); \node at (\x+.25,.25) {$q$};
\draw (\x+.2,0) -- (\x +.2,-1); \draw (\x +.3,0) -- (\x + .3,-1);
}
\foreach \x in {.5,2.5,3.25,4}{ \foreach \y in {.2,.3}{
\draw (\x,\y) -- (\x+.25,\y);
}}
\draw[Box,thin] (1,0) rectangle (1.5,.5); \node at (1.25,.25) {$q$};
\draw[Box,thin] (1,-.75) rectangle (1.5,-.25); \node at (1.25,-.5) {$q$};
\draw (.5,.2) -- (1,.2); \draw (.5,.3) -- (1,.3);
\draw(1.5,.3) -- (2,.3); \draw (1.5,.2) -- (2,.2);
\draw(1.2,0) -- (1.2,-.25); \draw (1.3,0) -- (1.3,-.25);
\draw(1,-.45) arc(90:180:.3cm) -- (.7,-1); \draw(1,-.55) arc(90:180:.2cm) -- (.8,-1);
\draw(1.5,-.45) arc(90:0:.3cm) -- (1.8,-1); \draw (1.5,-.55) arc(90:0:.2cm) -- (1.7,-1); 
\draw (0,.3) arc(270:90:.2cm) -- (4.75,.7) arc(90:-90:.2cm);
\draw (0,.2) arc(270:90:.3cm) -- (4.75,.8) arc(90:-90:.3cm);
\node[scale=.6] at (3,.25) {$\dotsb$};

\draw[Box] (-.5,-1) rectangle (5.25,1); \node[marked,above left=.01cm] at (-.5,1) {};
\end{tikzpicture}
\end{equation*}
Now if $n = 3$ then the statement follows from relation (3) above, otherwise it follows from induction on $n$.
\end{proof}

Define $Sym(\mc P)_{n} \subset \mc C_k(\mc P \otimes \mc P^{rev})_{n}$ to be the range of the partially labelled tangle
\begin{equation*}
\begin{tikzpicture}[yscale =.6, xscale = .75,very thick]
\draw[Box] (0,0) rectangle (3,1); \node[marked,scale=.9] at (-.05,1.05) {};
\draw[Box] (-.5,-.5) rectangle (3.5,2.5); \node[marked] at (-.55,2.55) {};

\draw (.25,1) -- (.25,2.5); \draw(.5,1) -- (.5,2.5); \node[cpr]  at (.375,1.75) {$p$};
\draw (1,1) -- (1,2.5); \draw(1.25,1) -- (1.25,2.5); \node[cpr] at (1.125,1.75) {$p$};
\draw (2.5,1) -- (2.5,2.5); \draw(2.75,1) -- (2.75,2.5); \node[cpr] at (2.625,1.75) {$p$};
\node at (1.875,1.75) {$\dotsb$};
\end{tikzpicture}
\end{equation*}
By Proposition \ref{commutants}, we have an identification of $Sym(\mc P)_n$ with $\mc M_0' \cap \mc M_{n}$ (as vector spaces).

For $n \geq 1$ define partially labelled tangles $F_n$ and $C_n$ by
\begin{equation*}
 \begin{tikzpicture}[yscale =.6, xscale = .75,very thick]
\draw[Box] (-.125,0) rectangle (3.625,1); \node[marked,scale=.9] at (-.175,1.05) {};
\draw[Box] (-.5,-.5) rectangle (4,2.5); \node[marked] at (-.55,2.55) {};

\draw (.25,1) -- (.25,1.5); \draw(.5,1) -- (.5,1.5); 
\draw (.1,1.5) -- (.1,2.5); \draw (.25,1.5) -- (.25,2.5);
\draw (.5,1.5) -- (.5,2.5); \draw (.65,1.5) -- (.65,2.5);
\node[cpr]  at (.375,1.75) {$\;q\;$};
\draw (1.25,1) -- (1.25,1.5); \draw(1.5,1) -- (1.5,1.5); 
\draw (1.1,1.5) -- (1.1,2.5); \draw (1.25,1.5) -- (1.25,2.5);
\draw (1.5,1.5) -- (1.5,2.5); \draw (1.65,1.5) -- (1.65,2.5);
\node[cpr] at (1.375,1.75) {$\;q\;$};

\node at (2.25,1.75) {$\dotsb$};

\draw (3,1) -- (3,1.5); \draw(3.25,1) -- (3.25,1.5); 
\draw (2.85,1.5) -- (2.85,2.5); \draw (3,1.5) -- (3,2.5);
\draw (3.25,1.5) -- (3.25,2.5); \draw (3.4,1.5) -- (3.4,2.5);
\node[cpr] at (3.125,1.75) {$\;q\;$};
\node[left] at (-.75,1.25) {$F_n = $};
\begin{scope}[xshift=7.5cm]
\draw[Box] (-.125,0) rectangle (3.625,1); \node[marked,scale=.9] at (-.175,1.05) {};
\draw[Box] (-.5,-.5) rectangle (4,2.5); \node[marked] at (-.55,2.55) {};

\draw (.25,1.5) -- (.25,2.5); \draw(.5,1.5) -- (.5,2.5); 
\draw (.1,1) -- (.1,1.5); \draw (.25,1) -- (.25,1.5);
\draw (.5,1) -- (.5,1.5); \draw (.65,1) -- (.65,1.5);
\node[cpr]  at (.375,1.75) {$\;q\;$};
\draw (1.25,1.5) -- (1.25,2.5); \draw(1.5,1.5) -- (1.5,2.5); 
\draw (1.1,1) -- (1.1,1.5); \draw (1.25,1) -- (1.25,1.5);
\draw (1.5,1) -- (1.5,1.5); \draw (1.65,1) -- (1.65,1.5);
\node[cpr] at (1.375,1.75) {$\;q\;$};

\node at (2.25,1.75) {$\dotsb$};

\draw (3,1.5) -- (3,2.5); \draw(3.25,1.5) -- (3.25,2.5); 
\draw (2.85,1) -- (2.85,1.5); \draw (3,1) -- (3,1.5);
\draw (3.25,1) -- (3.25,1.5); \draw (3.4,1) -- (3.4,1.5);
\node[cpr] at (3.125,1.75) {$\;q\;$}; 
\node[left] at (-.75,1.25) {$C_n = $};
\end{scope}
 \end{tikzpicture}
\end{equation*}
where the number of $q$'s in each tangle is $n$.  By relation (2) above, the linear maps associated to $F_n$ and $C_n$ restrict to maps $Z_{F_n}:Sym(\mc P)_n \to Sym(\mc P)_{2n}$ and $Z_{C_n}:Sym(\mc P)_{2n} \to Sym(\mc P)_n$.  Note that $Z_{C_n} \circ Z_{F_n}$ restricts to the identity on $Sym(\mc P)_n$ by (3) of Proposition \ref{cabled_skein}.

Now let $T$ be a planar tangle with $2n$ marked points on the boundary of its outer disc, and with input discs $D_1,\dotsc,D_m$ such that $D_i$ has $2n_i$ marked points on its boundary.  We define a ``spin factor'' (cf. \cite{jon3}) $\sigma(T)$ as follows.  First isotope $T$ so that the output and input rectangles each have as many marked points on the top as on the bottom (with no marked points on the sides), and such that the left side is the distinguished interval.  $\sigma(T)$ is then defined to be the product over all local extrema of all the strings of $I^{\pm 1/4}$, where the sign is taken to be positive if the convex region adjacent to the extreme point is unshaded, and negative if it is shaded.  If $n = 0$, we also multiply by $I^{-1}$ if the region adjacent to the boundary is shaded.

Let $\widetilde{T}$ be the tangle obtained by doubling the strings of $T$.  We then define a linear map
\begin{equation*}
 Z_T^{Sym}: Sym(\mc P)_{n_1} \otimes \dotsb \otimes Sym(\mc P)_{n_m} \to Sym(\mc P)_n
\end{equation*}
by
\begin{equation*}
Z_T^{Sym}(x_1 \otimes \dotsb \otimes x_m) = \sigma(T) \cdot Z_{C_n}(Z_{\widetilde T}(Z_{F_{n_1}}(x_1) \otimes \dotsb \otimes Z_{F_{n_m}}(x_m))).
\end{equation*}

In other words, $Z_T^{Sym}$ is equal to $\sigma(T) \cdot Z_{f(T)}$, where $f(T)$ is the partially labelled tangle obtained from $T$ by doubling the number of strings and composing with the tangles $C_n$ and $F_{n_i}$ in the appropriate manner. 

Let us now compute this action explicitly for several important classes of tangles.
\begin{example}\label{keytangles} In $Sym(\mc P)$, we have the following:
\begin{enumerate}
\item Identity:  If $n$ is even we have
\begin{equation*}
\begin{tikzpicture}[scale=.75,very thick]
\draw[thickline] (0,.75) --node[rcount] {$n$} (1,.75);
\draw[Box] (0,.25) rectangle (1,1.25); \node[marked,above left = .02cm] at (0,1.25) {};
\begin{scope}[xshift = 2.5cm,yshift = -1.25cm]
\draw[Box,thin] (0,2.25) rectangle (.5,3); \node at (.25,2.625) {$q$}; 
\draw[Box,thin] (0,1) rectangle (.5,1.75); \node at (.25,1.375) {$q$};
\draw[Box,thin] (1.5,2.25) rectangle (2,3); \node at (1.75,2.625) {$q$};
\draw[Box,thin] (1.5,1) rectangle (2,1.75); \node at (1.75,1.375) {$q$};
\draw (.5,2.85) -- (1.5,2.85); \draw (.5,2.7) -- (1.5,2.7);
\draw (.5,2.55) -- (1.5,2.55); \draw (.5,2.4) -- (1.5,2.4);
\draw (0,2.75) -- (-.25,2.75); \draw (0,2.5) -- (-.25,2.5);
\draw (2,2.75) -- (2.25,2.75); \draw (2,2.5) -- (2.25,2.5);
\draw (.5,1.6) -- (1.5,1.6); \draw (.5,1.45) -- (1.5,1.45);
\draw (.5,1.3) -- (1.5,1.3); \draw (.5,1.15) -- (1.5,1.15);
\draw (0,1.5) -- (-.25,1.5); \draw (2,1.5) -- (2.25,1.5);
\draw (0,1.25) -- (-.25,1.25); \draw (2,1.25) -- (2.25,1.25);

\node[rotate=90,scale = .9] at (1,2) {$\dotsb$};
\node[left] at (-.5,2) {$\mapsto$};
\end{scope}
\begin{scope}[xshift = 6.5cm,yshift = -1.25cm]
\draw[Box,thin] (0,2.375) rectangle (.5,2.875); \node[scale=.8] at (.25,2.625) {$p$}; 
\draw[Box,thin] (0,1.125) rectangle (.5,1.625); \node[scale=.8] at (.25,1.375) {$p$};

\foreach \x in {-.5,.5} {\foreach \y in {2.675,2.575,1.425,1.325} {
\draw (\x,\y) -- (\x+.5,\y);
}}

\node[rotate=90,scale = .8] at (.25,2) {$\dotsb$};
\node[left] at (-.75,2) {$=$};
\node[right] at (1.25,2) {$= 1_{\mc M_{n}}$};
\end{scope}

\end{tikzpicture}
\end{equation*}
Likewise if $n$ is odd we have
\begin{equation*}
\begin{tikzpicture}[scale=.75,very thick]
\draw[thickline] (0,.75) --node[rcount] {$n$} (1,.75);
\draw[Box] (0,.25) rectangle (1,1.25); \node[marked,above left = .02cm] at (0,1.25) {};
\begin{scope}[xshift = 2.25cm,yshift = -.75cm]
\draw[Box,thin] (.75,2.375) rectangle (1.25,2.875); \node[scale=.8] at (1,2.625) {$p$}; 
\draw[Box,thin] (.75,1.125) rectangle (1.25,1.625); \node[scale=.8] at (1,1.375) {$p$};
\draw[Box,thin] (.5,.125) rectangle (1.5,.625); \node at (1,.375) {$q$};

\draw (.75,.125) arc(0:-90:.125cm) -- (0,0); \draw (1.25,.125) arc(180:270:.125cm) -- (2,0);
\draw (.7,.625) arc(180:0:.3cm); \draw(.8,.625) arc(180:0:.2cm);

\foreach \x in {0,1.25} {\foreach \y in {2.675,2.575,1.425,1.325} {
\draw (\x,\y) -- (\x+.75,\y);
}}

\node[rotate=90,scale = .8] at (1,2) {$\dotsb$};
\node[left] at (-.25,1.5) {$\mapsto$};
\node[right] at (2.25,1.5) {$= 1_{\mc M_{n}}$};
\end{scope}

\end{tikzpicture}
\end{equation*}

\item Multiplication:  Let $x,y \in Sym(\mc P)_n$.  If $n$ is even we have
\begin{equation*}
 \begin{tikzpicture}[scale=.75,very thick]
\draw[Box] (0,-.5) rectangle (2,.5); \node at (1,0) {$x$}; \node[marked,above left] at (0,.5) {};
\draw[Box] (3,-.5) rectangle (5,.5); \node at (4,0) {$y$}; \node[marked, above left] at (3,.5) {};
\draw[thickline] (2,0) --node[rcount] {$n$} (3,0);
\draw[thickline] (0,0) -- (-.5,0); \draw[thickline] (5,0) -- (5.5,0);

\begin{scope}[xshift=8cm,yshift=-2.25cm]
\draw[Box] (.5,1.25) rectangle (1.5,3.25); \node at (1,2.25) {$x$}; \node[marked,above left] at (.5,3.25) {};
\draw[Box] (2.5,1.25) rectangle (3.5,3.25); \node at (3,2.25) {$y$}; \node[marked, above left] at (2.5,3.25) {};
\foreach \x in {-.25,1.75,3.75}{\foreach \y in {1.5,2.5} {
\draw[Box,thin] (\x,\y) rectangle (\x+.5,\y+.5); 
\node[scale=.8] at (\x+.25,\y+.25) {$p$}; 
\draw (\x,\y+.2) -- (\x-.25,\y+.2); \draw(\x-.25,\y + .4) -- (\x,\y+.4); \draw(\x+.5,\y+.2) -- (\x+.75,\y+.2); \draw(\x+.5,\y+.4) -- (\x+.75,\y+.4);
}}
\node[rotate=90,scale=.6] at (0,2.25) {$\dotsb$}; \node[rotate=90,scale=.6] at (2,2.25) {$\dotsb$}; \node[rotate=90,scale=.6] at (4,2.25) {$\dotsb$};
\node[left] at (-1,2.25) {$\mapsto$};
\node[right] at (4.65,2.25) {$= x \wedge y$};
\end{scope}
 \end{tikzpicture}
\end{equation*}

Likewise, if $n$ is odd we have
\begin{equation*}
 \begin{tikzpicture}[scale=.75,very thick]
\draw[Box] (0,-.5) rectangle (2,.5); \node at (1,0) {$x$}; \node[marked,above left] at (0,.5) {};
\draw[Box] (3,-.5) rectangle (5,.5); \node at (4,0) {$y$}; \node[marked, above left] at (3,.5) {};
\draw[thickline] (2,0) --node[rcount] {$n$} (3,0);
\draw[thickline] (0,0) -- (-.5,0); \draw[thickline] (5,0) -- (5.5,0);

\begin{scope}[xshift=8cm,yshift=-1.5cm]
\draw[Box] (.5,1) rectangle (1.5,3.25); \node at (1,2.25) {$x$}; \node[marked,above left] at (.5,3.25) {};
\draw[Box] (2.5,1) rectangle (3.5,3.25); \node at (3,2.25) {$y$}; \node[marked, above left] at (2.5,3.25) {};
\draw[Box] (1.5,0) rectangle (2.5,.5); \node at (2,.25) {$q$};
\draw (.5,1.25) arc(90:270:.2cm) -- (1.5,.85) arc(90:0:.2cm) -- (1.7,.5);
\draw (1.5,1.25) arc(90:0:.35cm) -- (1.85,.5);
\draw (2.5,1.25) arc(90:180:.35cm) -- (2.15,.5);
\draw (3.5,1.25) arc(90:-90:.2cm) -- (2.5,.85) arc(90:180:.2cm) -- (2.3,.5);
\draw (1.75,0) arc(0:-90:.25cm) -- (-.5,-.25);
\draw(2.25,0) arc(180:270:.25cm) -- (4.5,-.25);

\foreach \x in {-.25,1.75,3.75}{\foreach \y in {1.5,2.5} {
\draw[Box,thin] (\x,\y) rectangle (\x+.5,\y+.5); 
\node[scale=.8] at (\x+.25,\y+.25) {$p$}; 
\draw (\x,\y+.2) -- (\x-.25,\y+.2); \draw(\x-.25,\y + .3) -- (\x,\y+.3); \draw(\x+.5,\y+.2) -- (\x+.75,\y+.2); \draw(\x+.5,\y+.3) -- (\x+.75,\y+.3);
}}
\node[rotate=90,scale=.6] at (0,2.25) {$\dotsb$}; \node[rotate=90,scale=.6] at (2,2.25) {$\dotsb$}; \node[rotate=90,scale=.6] at (4,2.25) {$\dotsb$};
\node[left] at (-1,1.5) {$\mapsto$};
\node[right] at (4.65,1.5) {$= x \star_q y$};
\end{scope}
 \end{tikzpicture}
\end{equation*}
Note that we have applied Lemma \ref{daisychain} here.

\item Inclusions:  Let $x \in Sym(\mc P)_n$.  If $n$ is even we have
\begin{equation*}
 \begin{tikzpicture}[scale=.75,very thick]
\draw[Box] (0,-.5) rectangle (2,.5); \node at (1,0) {$x$}; \node[marked,above left] at (0,.5) {};
\draw[thickline] (2,0) -- (2.5,0);
\draw[thickline] (0,0) -- (-.5,0);
\draw (-.5,-.75) -- (2.5,-.75);
\draw[Box] (-.5,-1) rectangle (2.5,.85); \node[marked,above left] at (-.5,.85) {};
\begin{scope}[xshift = 4.35cm,yshift = -1.7cm]
\draw[Box] (.5,1.25) rectangle (1.5,3.25); \node at (1,2.25) {$x$}; \node[marked,above left] at (.5,3.25) {};
\foreach \x in {-.25,1.75}{\foreach \y in {1.5,2.5} {
\draw[Box,thin] (\x,\y) rectangle (\x+.5,\y+.5); 
\node[scale=.8] at (\x+.25,\y+.25) {$p$}; 
\draw (\x,\y+.2) -- (\x-.25,\y+.2); \draw(\x-.25,\y + .3) -- (\x,\y+.3); \draw(\x+.5,\y+.2) -- (\x+.75,\y+.2); \draw(\x+.5,\y+.3) -- (\x+.75,\y+.3);
}}
\node[rotate=90,scale=.6] at (0,2.25) {$\dotsb$}; \node[rotate=90,scale=.6] at (2,2.25) {$\dotsb$};
\draw[Box,thin] (.5,.125) rectangle (1.5,.625); \node at (1,.375) {$q$};

\draw (.75,.125) arc(0:-90:.125cm) -- (-.5,0); \draw (1.25,.125) arc(180:270:.125cm) -- (2.5,0);
\draw (.7,.625) arc(180:0:.3cm); \draw(.8,.625) arc(180:0:.2cm);

\node[left] at (-.75,1.625) {$\mapsto$};
\node[right] at (2.65,1.625) {$= i_{n}(x) $};

\end{scope}  
 \end{tikzpicture}
\end{equation*}
If $n$ is odd we have
\begin{equation*}
 \begin{tikzpicture}[scale=.75,very thick]
\draw[Box] (0,-.5) rectangle (2,.5); \node at (1,0) {$x$}; \node[marked,above left] at (0,.5) {};
\draw[thickline] (2,0) -- (2.5,0);
\draw[thickline] (0,0) -- (-.5,0);
\draw (-.5,-.75) -- (2.5,-.75);
\draw[Box] (-.5,-1) rectangle (2.5,.85); \node[marked,above left] at (-.5,.85) {};
\begin{scope}[xshift = 4.35cm,yshift = -1.8cm]
\draw[Box] (.5,1) rectangle (1.5,3.25); \node at (1,2.125) {$x$}; \node[marked,above left] at (.5,3.25) {};
\foreach \x in {-.25,1.75}{\foreach \y in {1.5,2.5} {
\draw[Box,thin] (\x,\y) rectangle (\x+.5,\y+.5); 
\node[scale=.8] at (\x+.25,\y+.25) {$p$};
\draw (\x,\y+.2) -- (\x-.25,\y+.2); \draw(\x-.25,\y + .3) -- (\x,\y+.3); \draw(\x+.5,\y+.2) -- (\x+.75,\y+.2); \draw(\x+.5,\y+.3) -- (\x+.75,\y+.3);
}}
\node[rotate=90,scale=.6] at (0,2.25) {$\dotsb$}; \node[rotate=90,scale=.6] at (2,2.25) {$\dotsb$};

\draw[Box,thin] (.5,.125) rectangle (1.5,.625); \node at (1,.375) {$q$};

\draw (.5,1.25) arc(90:270:.2cm) arc (90:0:.375cm and .225cm); \draw(1.5,1.25) arc(90:-90:.2cm) arc(90:180:.375cm and .225cm);
\foreach \x in {-.5,1.5}{\foreach \y in {.325,.425}{\draw(\x,\y) -- (\x+1,\y);  }}

\node[left] at (-.75,1.6) {$\mapsto$};
\node[right] at (2.65,1.6) {$= i_{n}(x) $};
\end{scope}  

 \end{tikzpicture}
\end{equation*}
Again we have applied  Lemma 4.2 in this case.
 \item Jones projections:  If $n$ is even we have
\begin{equation*}
\begin{tikzpicture}[scale=.75,very thick]
\draw[thickline] (0,.75) --node[rcount] {$n$} (1,.75);
\draw[thick] (0,.4) arc(90:-90:.2cm); \draw[thick] (1,.4) arc(90:270:.2cm);
\draw[Box] (0,-.25) rectangle (1,1.2); \node[marked,above left = .02cm] at (0,1.2) {};
\begin{scope}[xshift = 4.25cm,yshift = -1cm]
\draw[Box,thin] (0,0) rectangle (.5,.75); \node at (.25,.375) {$q$};
\draw[Box,thin] (1.5,0) rectangle (2,.75); \node at (1.75,.375) {$q$};
\draw[Box,thin] (.75,2.375) rectangle (1.25,2.875); \node[scale=.8] at (1,2.625) {$p$}; 
\draw[Box,thin] (.75,1.125) rectangle (1.25,1.625); \node[scale=.8] at (1,1.375) {$p$};

\foreach \x in {-.25,1.25} {\foreach \y in {2.675,2.575,1.425,1.325} {
\draw (\x,\y) -- (\x+1,\y);
}}

\draw (.5,.5) arc(90:-90:.125cm); \draw (.5,.6) arc(90:-90:.225cm);
\draw (1.5,.5) arc(90:270:.125cm); \draw (1.5,.6) arc(90:270:.225cm);
\draw (0,.5) -- (-.25,.5); \draw (0,.25) -- (-.25,.25);
\draw (2,.5) -- (2.25,.5); \draw (2,.25) -- (2.25,.25);

\node[rotate=90,scale = .9] at (1,2) {$\dotsc$};
\node[left] at (-.5,1.5) {$= I^{-1/2} \; \cdot$};
\node[right] at (2.5,1.5) {$ = I^{1/2}\cdot \e_n$};
\end{scope}
\end{tikzpicture}
\end{equation*}
Likewise, if $n$ is odd we have
\begin{equation*}
\begin{tikzpicture}[scale=.75,very thick]
\draw[thickline] (0,.75) --node[rcount] {$n$} (1,.75);
\draw[thick] (0,.4) arc(90:-90:.2cm); \draw[thick] (1,.4) arc(90:270:.2cm);
\draw[Box] (0,-.25) rectangle (1,1.2); \node[marked,above left = .02cm] at (0,1.2) {};
\begin{scope}[xshift = 4.25cm,yshift = -1cm]
\draw[Box,thin] (.5,.125) rectangle (1.5,.625); \node at (1,.375) {$q$};

\draw[Box,thin] (.75,2.375) rectangle (1.25,2.875); \node[scale=.8] at (1,2.625) {$p$}; 
\draw[Box,thin] (.75,1.125) rectangle (1.25,1.625); \node[scale=.8] at (1,1.375) {$p$};

\foreach \x in {-.25,1.25} {\foreach \y in {2.675,2.575,1.425,1.325} {
\draw (\x,\y) -- (\x+1,\y);
}}

\foreach \x in {-.25,1.5}{\foreach \y in {.325,.425}{\draw(\x,\y) -- (\x+.75,\y);  }}

\draw (.875,.125) arc(0:-90:.125cm) -- (-.25,0);
\draw (1.125,.125) arc(180:270:.125cm) -- (2.25,0);

\node[rotate=90,scale = .9] at (1,2) {$\dotsc$};
\node[left] at (-.5,1.5) {$= I^{1/2} \; \cdot$};
\node[right] at (2.5,1.5) {$ = I^{1/2}\cdot \e_n$};
\end{scope}
\end{tikzpicture}
\end{equation*}
where we have applied Lemma 4.2.  
\item Conditional expectation onto $\mc M_1'$:  Let $x \in Sym(\mc P)_n$.  If $n$ is even we have
\begin{equation*}
 \begin{tikzpicture}[scale=.75,very thick]
\draw[Box] (.25,-.75) rectangle (1.75,.25); \node at (1,-.25) {$x$}; \node[marked,above left,scale=.75] at (.3,.2) {};
\draw[thickline] (1.75,-.375) -- (2.5,-.375);
\draw[thickline] (.25,-.375) -- (-.5,-.375);

\draw (.25,0) arc(270:90:.25cm) -- (1.75,.5) arc(90:-90:.25cm);
\draw (-.5,.75) -- (2.5,.75);
\draw[Box] (-.5,-1) rectangle (2.5,1); \node[marked,above left] at (-.5,1) {};
\begin{scope}[xshift = 5.75cm,yshift = -3.2cm]
\draw[Box] (.5,1.25) rectangle (1.5,3.75); \node at (1,2.5) {$x$}; \node[marked,above left,scale=.75] at (.5,3.75) {};
\foreach \x in {-.25,1.75}{\foreach \y in {1.5,2.5} {
\draw[Box,thin] (\x,\y) rectangle (\x+.5,\y+.5); 
\node[scale=.8] at (\x+.25,\y+.25) {$p$}; 
\draw (\x,\y+.2) -- (\x-.25,\y+.2); \draw(\x-.25,\y + .3) -- (\x,\y+.3); \draw(\x+.5,\y+.2) -- (\x+.75,\y+.2); \draw(\x+.5,\y+.3) -- (\x+.75,\y+.3);
}}
\node[rotate=90,scale=.6] at (0,2.25) {$\dotsb$}; \node[rotate=90,scale=.6] at (2,2.25) {$\dotsb$};

\draw[Box,thin] (.75,4) rectangle (1.25,4.5); \node[scale=.8] at (1,4.25) {$q$};
\draw[Box,thin] (.75,4.75) rectangle (1.25,5.25); \node[scale=.8] at (1,5) {$q$};

\draw (.5,3.4) arc(270:90:.375cm) -- (.75,4.15); \draw(.5,3.2) arc(270:90:.575cm) -- (.75,4.35);
\draw (1.5,3.4) arc(-90:90:.375cm) -- (1.25,4.15); \draw(1.5,3.2) arc(-90:90:.575cm) -- (1.25,4.35);

\draw (.9,4.5) -- (.9,4.75); \draw (1.1,4.5) -- (1.1,4.75);
\draw (.75,4.9) -- (-.5,4.9); \draw (.75,5.1) -- (-.5,5.1);
\draw (1.25,4.9) -- (2.5,4.9); \draw (1.25,5.1) -- (2.5,5.1);
\node[left] at (-.75,3.125) {$\mapsto I^{1/2} \; \cdot $};
\node[right] at (3.15,3) {$= E_{\mc M_1'}(x)$};
\end{scope}  
 \end{tikzpicture}
\end{equation*}
where we have applied Lemma \ref{daisychain}.  Likewise, if $n$ is odd we have
\begin{equation*}
 \begin{tikzpicture}[scale=.75,very thick]
\draw[Box] (.25,-.75) rectangle (1.75,.25); \node at (1,-.25) {$x$}; \node[marked,above left,scale=.75] at (.3,.2) {};
\draw[thickline] (1.75,-.375) -- (2.5,-.375);
\draw[thickline] (.25,-.375) -- (-.5,-.375);

\draw (.25,0) arc(270:90:.25cm) -- (1.75,.5) arc(90:-90:.25cm);
\draw (-.5,.75) -- (2.5,.75);
\draw[Box] (-.5,-1) rectangle (2.5,1); \node[marked,above left] at (-.5,1) {};
\begin{scope}[xshift = 5.75cm,yshift = -2.75cm]
\draw[Box] (.5,1) rectangle (1.5,3.75); \node at (1,2.375) {$x$}; \node[marked,above left,scale=.75] at (.5,3.75) {};
\foreach \x in {-.25,1.75}{\foreach \y in {1.5,2.5} {
\draw[Box,thin] (\x,\y) rectangle (\x+.5,\y+.5); 
\node[scale=.8] at (\x+.25,\y+.25) {$p$}; 
\draw (\x,\y+.2) -- (\x-.25,\y+.2); \draw(\x-.25,\y + .3) -- (\x,\y+.3); \draw(\x+.5,\y+.2) -- (\x+.75,\y+.2); \draw(\x+.5,\y+.3) -- (\x+.75,\y+.3);
}}
\node[rotate=90,scale=.6] at (0,2.25) {$\dotsb$}; \node[rotate=90,scale=.6] at (2,2.25) {$\dotsb$};

\draw[Box,thin] (.75,4) rectangle (1.25,4.5); \node[scale=.8] at (1,4.25) {$q$};
\draw[Box,thin] (.75,4.75) rectangle (1.25,5.25); \node[scale=.8] at (1,5) {$q$};

\draw (.5,3.4) arc(270:90:.375cm) -- (.75,4.15); \draw(.5,3.2) arc(270:90:.575cm) -- (.75,4.35);
\draw (1.5,3.4) arc(-90:90:.375cm) -- (1.25,4.15); \draw(1.5,3.2) arc(-90:90:.575cm) -- (1.25,4.35);

\draw (.9,4.5) -- (.9,4.75); \draw (1.1,4.5) -- (1.1,4.75);
\draw (.75,4.9) -- (-.5,4.9); \draw (.75,5.1) -- (-.5,5.1);
\draw (1.25,4.9) -- (2.5,4.9); \draw (1.25,5.1) -- (2.5,5.1);

\draw[Box,thin] (.75, .25) rectangle (1.25,.75); \node[scale=.8] at (1,.5) {$p$};

\draw (.5,1.35) arc(90:270:.25cm) -- (.8,.85) arc(90:0:.1cm);
\draw (1.5,1.35) arc(90:-90:.25cm) -- (1.2,.85) arc(90:180:.1cm);
\draw (.9,.25) arc(0:-90:.25cm) -- (-.5,0);
\draw (1.1,.25) arc(180:270:.25cm) -- (2.5,0);
\node[left] at (-.75,2.675) {$\mapsto I^{1/2} \; \cdot$};
\node[right] at (3.15,2.675) {$= E_{\mc M_1'}(x)$};

\end{scope}  
 \end{tikzpicture}
\end{equation*}
In both cases we are applying Corollary \ref{m1expect}.

\item Trace: Let $x \in Sym(\mc P)_n$.  If $n$ is even we have
\begin{equation*}
\begin{tikzpicture}[scale=.75,very thick]
\draw[Box] (0,0) rectangle (1,1); \node at (.5,.5) {$x$}; \node[marked,above left,scale=.75] at (0,1) {};
\draw[thickline] (1,.5) arc(90:-90:.5cm) --node[rcount] {$n$} (0,-.5) arc(270:90:.5cm);
\begin{scope}[xshift=4cm,yshift=-1.6cm]
\draw[Box] (.5,1.25) rectangle (1.5,3.25); \node at (1,2.25) {$x$}; \node[marked,above left] at (.5,3.25) {};
\foreach \x in {-.25,1.75}{\foreach \y in {1.5,2.5} {
\draw[Box,thin] (\x,\y) rectangle (\x+.5,\y+.5); 
\node[scale=.8] at (\x+.25,\y+.25) {$p$}; 
}}

\foreach \x in {.5,1.75}{\foreach \y in {1.5,2.5} {
\draw (\x,\y+.2) -- (\x-.25,\y+.2); \draw(\x-.25,\y + .3) -- (\x,\y+.3);
\node[rotate=90,scale=.6] at (0,2.25) {$\dotsb$}; \node[rotate=90,scale=.6] at (2,2.25) {$\dotsb$};}}

\draw (2.25,1.7) arc(90:-90:.25cm and .3cm) -- (-.25,1.1) arc(270:90:.25cm and .3cm);
\draw (2.25,1.8) arc(90:-90:.35cm and .4cm) -- (-.25,1) arc(270:90:.35cm and .4cm);

\draw (2.25,2.7) arc(90:-90:.7cm and 1.2cm) -- (-.25,.3) arc(270:90:.7cm and 1.2cm);
\node[rotate=90,scale=.75] at (1,.6) {$\dotsb$};
\draw (2.25,2.8) arc(90:-90:.8cm and 1.3cm) -- (-.25,.2) arc(270:90:.8cm and 1.3cm);

\node[left] at (-1.25,1.75) {$\mapsto$};
\node[right] at (3.35,1.75) {$= \varphi_n(x) $};
\end{scope}

\end{tikzpicture}
\end{equation*}
Likewise if $n$ is odd we have
\begin{equation*}
\begin{tikzpicture}[scale=.75,very thick]
\draw[Box] (0,0) rectangle (1,1); \node at (.5,.5) {$x$}; \node[marked,above left,scale=.75] at (0,1) {};
\draw[thickline] (1,.5) arc(90:-90:.5cm) --node[rcount] {$n$} (0,-.5) arc(270:90:.5cm);
\begin{scope}[xshift=6cm,yshift=-.8cm]
\draw[Box] (.5,1) rectangle (1.5,3.25); \node at (1,2.125) {$x$}; \node[marked,above left] at (.5,3.25) {};
\foreach \x in {-.25,1.75}{\foreach \y in {1.5,2.5} {
\draw[Box,thin] (\x,\y) rectangle (\x+.5,\y+.5); 
\node[scale=.8] at (\x+.25,\y+.25) {$p$}; 
}}

\foreach \x in {.5,1.75}{\foreach \y in {1.5,2.5} {
\draw (\x,\y+.2) -- (\x-.25,\y+.2); \draw(\x-.25,\y + .3) -- (\x,\y+.3);
\node[rotate=90,scale=.6] at (0,2.25) {$\dotsb$}; \node[rotate=90,scale=.6] at (2,2.25) {$\dotsb$};}}

\draw (2.25,1.7) arc(90:-90:.35cm and 1.05cm) -- (-.25,-.4) arc(270:90:.35cm and 1.05cm);
\draw (2.25,1.8) arc(90:-90:.45cm and 1.15cm) -- (-.25,-.5) arc(270:90:.45cm and 1.15cm);

\node[rotate=90,scale=.75] at (1,-.9) {$\dotsb$};

\draw (2.25,2.7) arc(90:-90:.8cm and 1.95cm) -- (-.25,-1.2) arc(270:90:.8cm and 1.95cm);
\draw (2.25,2.8) arc(90:-90:.9cm and 2.05cm) -- (-.25,-1.3) arc(270:90:.9cm and 2.05cm);

\node[rotate=90,scale=.6] at (0,2.25) {$\dotsb$}; \node[rotate=90,scale=.6] at (2,2.25) {$\dotsb$};

\draw[Box,thin] (.5,.125) rectangle (1.5,.625); \node at (1,.375) {$q$};

\draw (.5,1.25) arc(90:270:.2cm) arc (90:0:.375cm and .225cm); \draw(1.5,1.25) arc(90:-90:.2cm) arc(90:180:.375cm and .225cm);
\draw (.7,.125) arc(180:360:.3cm); \draw (.85,.125) arc(180:360:.15cm);

\node[left] at (-1.35,1) {$\mapsto I^{-1/2} \; \cdot$};
\node[right] at (3.35,1) {$= \varphi_n(x) $};
\end{scope}

\end{tikzpicture}
\end{equation*}
\end{enumerate}

\end{example}

We will now show that $Sym(\mc P)$ is indeed the planar algebra of $M_0 \otimes M_0^{op} \subset M_0 \boxtimes M_0^{op}$.  First we need the following technical lemma, which will imply the compatibility of gluing for $Sym(\mc P)$.  

\begin{lemma}\label{gluing}
Let $T$ be a planar tangle which is partially labelled by $q$'s.  Beginning from the marked interval and moving clockwise, number the outgoing strings of each $q$ by $1,\dotsc,6$.  Assume that the following condition is satisfied: for each $q$ appearing in $T$, there is another one such that strings $1$ and $2$ of the first are connected to strings $4$ and $3$ of the second, respectively.  Suppose that $T$ contains a subtangle of the following form:
\begin{equation*}
 \begin{tikzpicture}[thick,yscale=.75]
\draw (.25,1) -- (.25,1.5) ; \draw(.5,1) -- (.5,1.5); 
\draw (.1,1.5) -- (.1,2.5) node[above=.05cm,scale=.6] {$1$}; \draw (.25,1.5) -- (.25,2.5) node[above=.05cm,scale=.6] {$2$};
\draw (.5,1.5) -- (.5,2.5) node[above=.05cm,scale=.6] {$3$}; \draw (.65,1.5) -- (.65,2.5) node[above=.05cm,scale=.6] {$4$};
\node[cpr,scale=1.2] (q1) at (.375,1.75) {$\;q\;$}; \node[marked,above left, scale=.7] at (q1.north west) {};
\begin{scope}[yscale=-1,yshift=-2.375cm]
 \draw (.25,1) -- (.25,1.5) ; \draw(.5,1) -- (.5,1.5) ;
\draw (.1,1.5) -- (.1,2.5) node[below=.05cm,scale=.6] {$4$}; \draw (.25,1.5) -- (.25,2.5) node[below=.05cm,scale=.6] {$3$};
\draw (.5,1.5) -- (.5,2.5) node[below=.05cm,scale=.6] {$2$}; \draw (.65,1.5) -- (.65,2.5) node[below=.05cm,scale=.6] {$1$};
\node[cpr,scale=1.2] (q2) at (.375,1.75) {$\;q\;$}; \node[marked, below right, scale=.7] at (q2.south east) {};
\end{scope}
\draw[Box,dashed] (-.25,-.125) rectangle (1,2.5); \node[marked,scale=.8,above left] at (-.25,2.5) {};
\end{tikzpicture}
\end{equation*}
Let $T'$ be the tangle obtained from $T$ by removing this subtangle and replacing it with the following:
\begin{equation*}
 \begin{tikzpicture}[thick,yscale=.75]
\draw (0,0) -- (0,2); \draw (.15,0) -- (.15,2);
\draw (.6,0) -- (.6,2); \draw (.75,0) -- (.75,2);

\node[cpr,draw,fill=white] at (.075,1) {$p$};
\node[cpr,draw,fill=white] at (.675,1) {$p$};
\draw[Box,dashed] (-.25,0) rectangle (1,2); \node[marked,scale=.8,above left] at (-.25,2) {};
\end{tikzpicture}
\end{equation*}
Then in the planar algebra $\mc C_k(\mc P \otimes \mc P^{rev})$, we have $Z_T = Z_{T'}$.
\end{lemma}

\begin{proof}
We may assume without loss of generality that $T$ is connected.  Indeed, suppose there were some connected component consisting of internal discs and strings.  If the subtangle in question is contained in this, then we may restrict to this connected component.  Otherwise we may remove this component, which for fixed inputs will contribute the same multiplicative factor to both $Z_T$ and $Z_{T'}$.

Label the upper and lower $q$'s appearing in the diagram by $q_t$ and $q_b$, respectively.  Now consider following strings $1$ and $2$ of the $q_b$: by assumption we will arrive at strings $4$ and $3$ of another $q$.  Following strings $1$ and $2$ of this new $q$ we would arrive at strings $4$ and $3$ of another $q$, and so on.  Eventually we must arrive back at strings $4$ and $3$ of $q_b$.  There are two cases to consider, depending on whether this path meets $q_t$ or not.  First suppose that it does not.  Then the winding number of this path around $q_t$ is either 0 or 1. We may assume without loss of generality that it is equal to 0.  Indeed, if it is equal to 1, then consider the path obtained by starting at strings $1$ and $2$ of the $q_t$ and following the same procedure.  Then this path does not meet $q_b$, and it is easy to see that the winding number around $q_b$ must be equal to $0$ by planarity.  We may then rotate the picture by 180 degrees, reversing the roles of $q_t$ and $q_b$.

We may now isotope the diagram to obtain the following picture:
\begin{equation*}
 \begin{tikzpicture}[scale=.75, very thick]
 \foreach \x in {0,.75,1.5,3,3.75}{
  \draw[Box,thin] (\x,0) rectangle (\x+.5,.5); \node at (\x+.25,.25) {$q$};
\draw (\x+.2,0) -- (\x +.2,-.25); \draw (\x +.3,0) -- (\x + .3,-.25);
la}
\foreach \x in {.5,1.25,2,2.75,3.5}{ \foreach \y in {.2,.3}{
\draw (\x,\y) -- (\x+.25,\y);
}}
\node[scale=.6] at (2.5,.25) {$\dotsb$};

\draw (0,.3) arc(270:90:.45cm) -- (4.25,1.2) arc(90:-90:.45cm);
\draw (0,.2) arc(270:90:.55cm) -- (4.25,1.3) arc(90:-90:.55cm);

\begin{scope}[yshift=.25cm]
\draw[Box,thin,fill=white] (1.875,.75) rectangle (2.375,1.25); \node[scale=.8] (qb) at (2.125,1) {$q_b$}; \node[marked,scale=.6,below=-.07] at (qb.south) {};
\draw[Box,thin,fill=white] (1.75,1.5) rectangle (2.5,2); \node[scale=.8] (qt) at (2.125,1.75) {$q_t$}; \node[marked,scale=.6,left=.08] at (qt.west) {};
\draw (2.075,1.25) -- ++(0,.25); \draw (2.175,1.25) -- ++(0,.25);
\foreach \x in {.125,.25,.5,.625} {\draw (1.75 + \x,2) -- (1.75 + \x,2.25);}
\end{scope}
\draw[Box,dashed] (-.75,-.25) rectangle (5,2.5); \node[marked,above left=.01cm] at (-.75,2.5) {};
\end{tikzpicture}
\end{equation*}
Note that we are using the assumption that $T$ is connected here, otherwise there could be something appearing in the region adjacent to the marked interval of $q_b$.  By applying relation (4) of Proposition \ref{cabled_skein}, we may modify $T$ as follows (without effecting $Z_T$):
\begin{equation*}
 \begin{tikzpicture}[scale=.75, very thick]
 \foreach \x in {0,.75,1.5,3,3.75}{
  \draw[Box,thin] (\x,0) rectangle (\x+.5,.5); \node at (\x+.25,.25) {$q$};
\draw (\x+.2,0) -- (\x +.2,-.25); \draw (\x +.3,0) -- (\x + .3,-.25);
}
\foreach \x in {.5,1.25,2,2.75,3.5}{ \foreach \y in {.2,.3}{
\draw (\x,\y) -- (\x+.25,\y);
}}
\draw (0,.3) arc(270:90:.325cm) -- (4.25,.95) arc(90:-90:.325cm);
\draw (0,.2) arc(270:90:.425cm) -- (4.25,1.05) arc(90:-90:.425cm);

\begin{scope}[xshift=.375cm]
\draw (2.075,1.25) -- ++(0,.25); \draw (2.175,1.25) -- ++(0,.25);
\draw[Box,thin,fill=white] (1.875,.75) rectangle (2.375,1.25); \node at (2.125,1) {$q$};
\end{scope}
\begin{scope}[xshift=-.375cm]
\draw (2.075,1.25) -- ++(0,.25); \draw (2.175,1.25) -- ++(0,.25);
\draw[Box,thin,fill=white] (1.875,.75) rectangle (2.375,1.25); \node at (2.125,1) {$q$};
\end{scope}
\node[scale=.6] at (2.5,.25) {$\dotsb$};

\draw[Box,dashed] (-.75,-.25) rectangle (5,1.5); \node[marked,above left=.01cm] at (-.75,1.5) {};
\end{tikzpicture}
\end{equation*}
The result then follows from an application of Lemma \ref{daisychain} (or Proposition \ref{cabled_skein} (3) if there are only 2 $q$'s appearing). 

Now consider the second case, where the path constructed above contains $q_t$ as well as $q_b$.  Consider the part of the path which begins at string $1$ of $q_b$ and ends at string $4$ of $q_t$.  Suppose that we extend this downward to connect back to $q_b$ to obtain a closed path.  Then the winding number of this path around any point in region adjacent to the marked interval of $q_b$ is either $0$ or $1$.  As above, by reversing the roles of $q_b$ and $q_t$ if necessary, we may assume that it is equal to $1$.  

After isotoping, we may then arrange the diagram as follows:
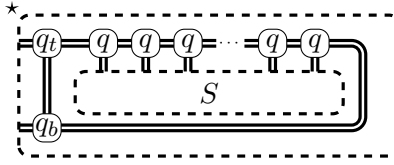
\begin{equation*}
 \begin{tikzpicture}[scale=.75, very thick]
 \foreach \x in {0,.75,1.5,3,3.75}{
  \draw[Box,thin] (\x,0) rectangle (\x+.5,.5); \node at (\x+.25,.25) {$q$};
\draw (\x+.2,0) -- (\x +.2,-.25); \draw (\x +.3,0) -- (\x + .3,-.25);
}
\foreach \x in {.5,1.25,2,2.75,3.5}{ \foreach \y in {.2,.3}{
\draw (\x,\y) -- (\x+.25,\y);
}}

\draw[Box,dashed] (-.25,-1) rectangle (4.5,-.25); \node at (2.125,-.625) {$S$};

\draw[Box,thin] (-1,0) rectangle (-.5,.5); \node at (-.75,.25) {$q_t$};

\draw[Box,thin] (-1,-1.5) rectangle (-.5,-1); \node at (-.75,-1.25) {$q_b$};
\draw (-.5,-1.2) -- ++(5.15,0)  arc(-90:0:.15cm) -- ++(0,1.1) arc (0:90:.15cm) -- ++(-.4,0);
\draw (-.5,-1.3) -- ++(5.15,0)  arc(-90:0:.25cm) -- ++(0,1.1) arc(0:90:.25cm) -- ++(-.4,0);

\draw (-.7,-1) -- ++(0,1); \draw (-.8,-1) -- ++(0,1);
\foreach \y in {-1.2,-1.3,.2,.3} {\draw (-1,\y) -- ++(-.25,0);}

\draw (-.5,.2) -- ++(.5,0); \draw (-.5,.3) -- ++(.5,0);

\node[scale=.6] at (2.5,.25) {$\dotsb$};

\draw[Box,dashed] (-1.25,-1.75) rectangle (5.5,.75); \node[marked,above left=.01cm] at (-1.25,.75) {};
\end{tikzpicture}
\end{equation*}
where $S$ is a subtangle of $T$.   Apply relation (4) of Proposition \ref{cabled_skein} to obtain the following:
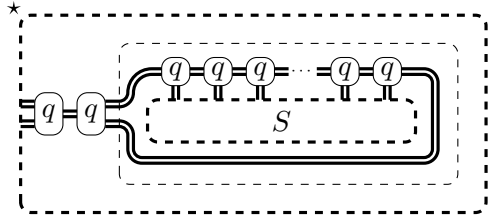
\begin{equation*}
 \begin{tikzpicture}[scale=.75, very thick]
 \foreach \x in {0,.75,1.5,3,3.75}{
  \draw[Box,thin] (\x,0) rectangle (\x+.5,.5); \node at (\x+.25,.25) {$q$};
\draw (\x+.2,0) -- (\x +.2,-.25); \draw (\x +.3,0) -- (\x + .3,-.25);
}
\foreach \x in {.5,1.25,2,2.75,3.5}{ \foreach \y in {.2,.3}{
\draw (\x,\y) -- (\x+.25,\y);
}}

\draw[Box,dashed] (-.25,-1) rectangle (4.5,-.25); \node at (2.125,-.625) {$S$};

\draw[Box,thin] (-1.5,-.875) rectangle (-1,-.125); \node at (-1.25,-.5) {$q$};

\draw (-1,-.125-.15)-- ++(.25,0) arc(-90:0:.15cm)  arc(180:90:.425cm) -- ++(.175,0);
\draw(-1,-.125-.25) -- ++(.25,0) arc(-90:0:.25cm)  arc(180:90:.325cm) -- ++(.175,0);

\draw(-1,-.875+.15) -- ++(.25,0) arc(90:0:.15cm) -- ++(0,-.25) arc(180:270:.25cm) -- ++(5,0) arc(-90:0:.25cm) -- ++(0,1.175) arc(0:90:.25cm) -- ++(-.4,0);
\draw(-1,-.875+.25) -- ++(.25,0) arc(90:0:.25cm) --++(0,-.25) arc(180:270:.15cm) -- ++(5,0) arc(-90:0:.15cm) -- ++(0,1.175) arc(0:90:.15cm) -- ++(-.4,0);

\draw[Box,thin] (-2.25,-.875) rectangle (-1.75,-.125); \node at (-2,-.5) {$q$};
\draw (-1.75,-.875+.325) -- ++(.25,0); \draw (-1.75,-.875+.425) -- ++(.25,0);

\foreach \y in {.15,.25,.5,.6} {\draw (-2.25,-.875+\y) --++(-.25,0);}

\node[scale=.6] at (2.5,.25) {$\dotsb$};

\draw[Box,dashed] (-2.5,-2.25) rectangle (5.75,1.25); \node[marked,above left=.01cm] at (-2.5,1.25) {};
\draw[Box,dashed,thin] (-.75,-1.75) rectangle (5.25,.75);
\end{tikzpicture}
\end{equation*}
The result then follows from (5) of Proposition \ref{cabled_skein}.
\end{proof}

\begin{remark}
Note that the tangle $T'$ obtained from $T$ as above still satisfies the hypotheses of the Lemma (modulo an application of Proposition \ref{cabled_skein} (2) if necessary).  So by iterating this procedure, we may make this replacement for every occurrence of the subdiagram which appeared in the statement of the lemma.

Note that if $T$ is fully labelled by $q$'s and satisfies the hypotheses of the Lemma, then by iterating this procedure and applying Proposition \ref{cabled_skein} we can compute the partition function $Z_T$.  In particular, we see that in the (unshaded) planar algebra generated by $q$ in the unshaded 2-cabling of $\mc C_k(\mc P \otimes \mc P^{rev})$ we have the relation
\begin{equation*}
 \begin{tikzpicture}[yscale=.33,xscale=.5]
  \draw[Box] (0,0) rectangle (1,2); \node at (.5,1) {$q$}; \node[marked,above left] at (0,2) {};
\draw[Box] (2,0) rectangle (3,2); \node at (2.5,1) {$q$}; \node[marked,above right] at (3,2) {};

\foreach \x in {-.5,3} {\foreach \y in {.75,1.25} {\draw[thick] (\x,\y) --++(.5,0); }}
\draw[thick] (1,1) --++(1,0);

\begin{scope}[xshift=5.75cm]
\draw[thick] (0,1.25) --++(2,0);
\draw[thick] (0,.75) --++(2,0);
\node[left] at (-.5,1) {$=$};
\end{scope}
 \end{tikzpicture}
\end{equation*}
which together with the rotational invariance of $q$ and the loop parameter $I$ determines the partition function.
\end{remark}

\begin{theorem}
$Sym(\mc P)$ is a spherical C$^*$-planar algebra, and the identification of $Sym(\mc P)_n$ with $\mc M_0' \cap \mc M_n$ is an isomorphism between $Sym(\mc P)$ and the planar algebra of the inclusion $M_0 \otimes M_0^{op} \subset M_0 \boxtimes M_0^{op}$.
\end{theorem}

\begin{proof}
Since $\mc C_k(\mc P \otimes \mc P^{rev})$ is a planar algebra, it is clear that the action $Z^{Sym}_T$ depends only on the isotopy class of $T$.  We must show that the action is compatible with the gluing of tangles.  So suppose that we have two tangles $T$ and $S$, such that the $i$-th input disc of $T$ and the output disc of $S$ both have the same number of marked points $2n$.  Now we have $Z^{Sym}_{T \circ_i S} = \sigma(T \circ_i S) Z_{f(T \circ_i S)}$ while $Z^{Sym}_{T} \circ_i Z^{Sym}_{S} = \sigma(T) \sigma(S) Z_{f(T)} \circ Z_{f(S)} = \sigma(T)\sigma(S) Z_{f(T) \circ_i f(S)}$.  It is clear that $\sigma(T \circ_i S) = \sigma(T) \sigma(S)$, so it remains to show that $Z_{f(T) \circ_i f(S)} = Z_{f(T \circ_i S)}$.  

Now in $f(T) \circ_i f(S)$ there are $n$ copies of the diagram from the statement of Lemma \ref{gluing} which are arranged around the former boundary of the $i$-th input disc of $\widetilde T$.  Since $T \circ_i S$ is a shaded tangle, it is not hard to see that $f(T) \circ_i f(S)$ satisfies the conditions of that lemma, and so we may make the indicated replacements.  After applying (2) of Proposition \ref{cabled_skein} if needed, the resulting tangle is then clearly isotopic to $f(T \circ_i S)$, and hence we have $Z_{f(T \circ_i S)} = Z_{f(T) \circ_i f(S)}$ as desired.  This proves that $Sym(\mc P)$ is a planar algebra.  

The remaining properties of a spherical $C^*$-planar algebra follow easily from the corresponding properties of $\mc C_k(\mc P \otimes \mc P^{rev})$ (note that one has to be careful about the spin factors when checking sphericality).

Now  in Example \ref{keytangles}, we have shown that the identification of $Sym(\mc P)_n$ with $\mc M_0' \cap \mc M_n$ is compatible with the action of several classes of tangles.  But it is a well-known result of Jones  (see \cite[Theorem 4.2.1]{palg}) that the planar algebra structure on $(\mc M_0' \cap \mc M_n)_{n \geq 0}$ is uniquely determined by the action of these tangles.  So it follows that this identification gives an isomorphism of planar algebras between $Sym(\mc P)$ and $\mc P(\mc M_0 \subset \mc M_1)$, which completes the proof.
\end{proof}

\begin{remark}\label{m0vm1}
We have worked with the inclusion $M_0 \otimes M_0^{op} \subset M_0 \boxtimes M_0^{op}$ instead of the symmetric enveloping inclusion $M_1 \otimes M_1^{op} \subset M_1 \boxtimes M_1^{op}$ for simplicity.  However, we can now recover the planar algebra of the symmetric enveloping inclusion (or asymptotic inclusion) as follows.  Since $M_0 \otimes M_0^{op} \subset M_0 \boxtimes M_0^{op}$ is isomorphic to the compression by $e_0 \otimes e_0^{op}$ of the inclusion $M_2 \otimes M_2^{op} \subset M_2 \boxtimes M_2^{op}$, it follows that these inclusions have the same planar algebra.  The latter subfactor is isomorphic to the symmetric enveloping inclusion for $M_1 \subset M_2$ by \cite{cjs}.  The planar algebra of $M_1 \subset M_2$ is $\mc P^{op}$, the dual planar algebra of $\mc P$ in which the shadings are reversed \cite{palg}.  As discussed by Popa \cite{pop1}, the planar algebra of the symmetric enveloping inclusion of a subfactor depends only on the planar algebra of that subfactor.  By duality, it now follows that the planar algebra of the symmetric enveloping inclusion $M_1 \otimes M_1^{op} \subset M_1 \boxtimes M_1^{op}$ is isomorphic to $Sym(\mc P^{op})$.  
\end{remark}

\section{Fusion rules and the affine category}

In this section we compute the fusion rules for the asymptotic inclusion, recovering some results of Ocneanu and Evans-Kawahigashi \cite{evk}.  In particular, we show that the fusion rules for $\mc M_1-\mc M_1$-bimodules are described by the affine category of $\mc P$. We remark that it has recently been shown by Das, Ghosh and Gupta \cite{dgg2} that the category of affine Hilbert representations is equivalent to the Drinfeld center of the fusion category associated to $\mc P$ (see also \cite{dgg1}).  

\smallskip
\noindent\textbf{The affine category:}  We briefly recall the definition of the \textit{affine category} of a planar algebra $\mc P$, for further details see \cite{annular}, \cite{ghosh}. 

\begin{definition}
An \textit{affine $(n,m)$-tangle} $T$ is a planar tangle with outer box equal to $\{(x,y) \in \R^2: \max(|x|,|y|) = 2\}$, a distinguished inner box equal to $\{(x,y) \in \R^2 :\max(|x|,|y|) = 1\}$, with marked points $\{(-2, \frac{2i}{n}):  -n \leq i < n\}$ (resp. $\{(-1,\frac{i}{m}): -m \leq k < m\}$) on the boundary of the outer box (resp. distinguished inner box) and with distinguished intervals occuring at the top of the outer box and distinguished inner box.  
\end{definition}

Typically one must specify the shading near the boundary of the outer (resp. distinguished inner box) when $n = 0$ (resp. $m =0$).  Here we will only be concerned with the `positive' part of the affine category, where these regions are assumed to be unshaded.

If $\mc P$ is  a planar algebra we say that an affine $(n,m)$-tangle $T$ is $\mc P$-\textit{labelled} if we have an assignment of elements of $\mc P$ to each input box of $T$, except for the distinguished one.  Note that $T$ then determines a linear map $P_{m} \to P_n$ by assigning the element of $P_m$ to the distinguished input box and applying $Z_T$.

Define the (positive) affine category $Aff^+(\mc P)$ to have one object for each $k \geq 0$, and with morphisms equal to isotopy classes of $\mc P$-labelled affine tangles, where the isotopies are required to fix the outer and distinguished inner boxes.  If $T$ (resp. $S$) is an affine $(n,k)$ (resp. $(k,m)$) tangle, the composition $TS$ is the affine $(n,m)$-tangle obtained by scaling $S$ by a factor of $1/2$, composing the resulting diagram with $T$ and then rescaling the resulting tangle.  Let $F\mc A^+(\mc P)$ be the linearization of this category, i.e. the objects are the same but the morphisms are vector spaces with bases given by isotopy classes of affine tangles.  

\begin{remark}
The distinction between the affine category and the \textit{annular category} of \cite{annular} is that annular tangles are taken up to isotopies which are not required to fix the outer and distinguished inner boxes.
\end{remark}

For $k \geq 0$ we have a natural map $\psi_k$ from (isotopy classes of) labelled planar $(n+m+2k)$-tangles to (isotopy classes of) affine $(n,m)$-tangles given by:
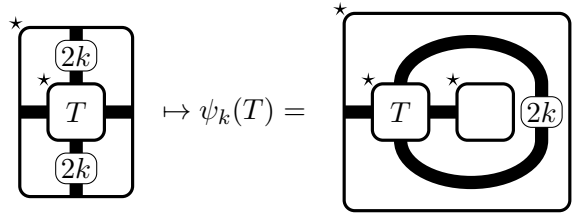
\begin{equation*}
\begin{tikzpicture}[scale=.75]
\draw[Box] (0,0) rectangle (1,1); \node at (.5,.5) {$T$}; \node[marked,above left] at (0,1) {};
\draw[thickline] (-.5,.5) -- (0,.5);
\draw[thickline] (1,.5) -- (1.5,.5);
\draw[thickline] (.5,1) -- node[rcount] {$2k$} (.5,2);
\draw[thickline] (.5,0) -- node[rcount] {$2k$} (.5,-1);
\draw[Box] (-.5,-1) rectangle (1.5,2); \node[marked,above left] at (-.5,2) {};
\begin{scope}[xshift=5.75cm]
\draw[Box] (-.5,-1.25) rectangle (3.5,2.25); \node[marked,above left] at (-.5,2.25) {};
\draw[Box] (0,0) rectangle (1,1); \node at (.5,.5){$T$}; \node[marked,above left] at (0,1) {};
\draw[Box] (1.5,0) rectangle (2.5,1); \node[marked,above left] at (1.5,1) {};
\draw[thickline] (.5,1) arc(180:0:1.25cm and .75cm) -- node[rcount] {$2k$} (3,0) arc(0:-180:1.25cm and .75cm);
\draw[thickline] (1,.5) -- (1.5,.5); \draw[thickline] (0,.5) -- (-.5,.5);
\node[left] at (-1,0.5) {$\mapsto \psi_k(T) = $};
\end{scope}
\end{tikzpicture} 
\end{equation*}
It is easy to see that every isotopy class of affine $(n,m)$-tangles is in the range of $\psi_k$ for $k$ sufficiently large.  

Define $\mc R \subset F\mc A^+(\mc P)$ by
\begin{equation*}
\mc R = \mathrm{span}\biggl\{\sum \psi_{k}(T_i): T_i \text{ are $\mc P$-labelled $(n+m+2k)$-tangles such that }\sum Z_{T_i} = 0\biggr\}.  
\end{equation*}

\begin{definition}
The (positive) \textit{affine algebroid} $\mc A^+(\mc P) = \{\mc A(\mc P)_{n,m}: n,m \geq 0\}$ is the quotient of $F\mc A^+(\mc P)$ by $\mc R$.
\end{definition}

In other words, $\mc A^+(\mc P)$ is the quotient of the universal affine algebroid of $\mc P$ by all relations which hold in a contractible disc.  It is easy to see that composition of affine tangles passes to $\mc A^+(\mc P)$.    It is clear from the definitions that $\psi_k$ gives a well-defined linear map $P_{n+m+2k} \to \mc A(\mc P)_{n,m}$.  Moreover, we have the following description of the kernel.

\begin{lemma}\label{affineiso}
The kernel of $\psi_k:P_{n+m+2k} \to \mc A(\mc P)_{n,m}$ is spanned by elements of the form
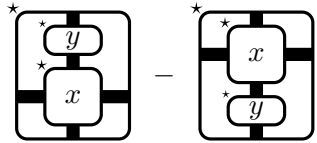
\begin{equation*}
\begin{tikzpicture}[scale=.75]
 \draw[Box] (0,0) rectangle (1,1); \node at (.5,.5) {$x$}; \node[marked,scale=.8,above left] at (0,1) {};
\draw[Box] (0,1.25) rectangle (1,1.75); \node at (.5,1.5) {$y$}; \node[marked,scale=.6,above left] at (0,1.75) {};

\foreach \y in {2,1.25,0} {\draw[thickline] (.5,\y) -- ++(0,-.25);}
\draw[thickline] (-.5,.5) -- (0,.5); \draw[thickline] (1,.5) -- (1.5,.5);

\draw[Box] (-.5,-.25) rectangle (1.5,2); \node[marked,scale=.9,above left] at (-.5,2) {};

\begin{scope}[xshift=3.25cm]
 \draw[Box] (0,.75) rectangle (1,1.75); \node at (.5,1.25) {$x$}; \node[marked,scale=.8,above left] at (0,1.75) {};
\draw[Box] (0,0) rectangle (1,.5); \node at (.5,.25) {$y$}; \node[marked,scale=.6,above left] at (0,.5) {};

\foreach \y in {2,.75,0} {\draw[thickline] (.5,\y) -- ++(0,-.25);}
\draw[thickline] (-.5,1.25) -- (0,1.25); \draw[thickline] (1,1.25) -- (1.5,1.25);

\draw[Box] (-.5,-.25) rectangle (1.5,2); \node[marked,scale=.9,above left] at (-.5,2) {};

\node[left] at (-.75,.875) {$-$};
\end{scope}
\end{tikzpicture}
\end{equation*}
for $x \in P_{2n+k}$ and $y \in P_{2k}$.
\end{lemma}

\begin{proof}
We prove this first for the universal planar algebra $\mc P(L)$ with labelling set $L = P$ (cf. \cite{palg}).  In this case $P_{n+m+2k}$ is replaced by the vector space with basis given by isotopy classes of $\mc P$-labelled $n+m+2k$-tangles, and $\mc A(\mc P)_{n,m}$ is replaced by the vector space with basis given by isotopy classes of $\mc P$-labelled affine $(n,m)$-tangles.  Since $\psi_k$ maps basis elements to basis elements, its kernel is spanned by elements of the form $T_1 - T_2$, where $T_1,T_2$ are $\mc P$-labelled planar $n+m+2k$ tangles such that $\psi_k(T_1)$ is isotopic to $\psi_k(T_2)$.  By a standard topological argument it follows that the kernel is spanned by elements of the form appearing in the statement of the lemma, where $x$ is a labelled $n+2k$-tangle and $y$ is a labelled $2k$-tangle.  Now taking the quotient by the kernel of the partition function $Z:\mc P(L) \to \mc P$, we see that the result holds for $\mc P$ as well.
\end{proof}

Note that by Proposition \ref{depthspan}, if $\mc P$ is finite-depth then $\psi_{k}$ is surjective if $2k \geq depth(\mc P)$.  In particular,  if $\mc P$ is finite-depth then $\mc A(\mc P)_{n,m}$ is finite-dimensional for every $n,m \geq 0$. 

A module over $\mc A^+(\mc P)$ is a graded vector space $V = (V_n)$ together with an action of affine tangles.  So to each affine $(n,m)$-tangle there is an associated linear map $V_m \to V_n$, which is compatible with composition of affine tangles.  In particular, $\mc P^+ = (P_n)_{n \geq 0}$ is a module over $\mc A^+(\mc P)$.  The \textit{weight} of a module is the least value of $n$ such that $V_n$ is non-trivial.

If each $V_n$ is a finite-dimensional Hilbert space, and we have
\begin{equation*}
\langle av, w \rangle = \langle v,a^* w\rangle
\end{equation*}
for all $v,w \in V$ and $a \in \mc A^+(\mc P)$, then we say that $V$ is a Hilbert module over $\mc A^+(\mc P)$.  For Hilbert modules the notions of \textit{indecomposability} and \textit{irreducibility} are equivalent, see \cite{annular}, \cite{ghosh}.

\medskip
\noindent\textbf{Affine representations and Ocneanu's tube algebra:}

We will now assume that $\mc P$ is finite-depth.  Let $2k \geq depth(\mc P)$.  For each $v \in \Gamma_+$, fix a minimal projection $t_v$ in the central component of $P_{2k}$ corresponding to $v$.  Recall that $\tr_{P_{2k}}(t_v) = \delta^{-2k}\mu_v$.   Let $t = \sum t_v$, then Ocneanu's \textit{tube algebra} is the compression
\begin{equation*}
 Tube(\mc P) = t(\mc A(\mc P)_k) t = \bigoplus_{v,w \in \Gamma_+} Tube(\mc P)_{v,w},
\end{equation*}
where we set $Tube(\mc P)_{v,w} = t_v(\mc A(\mc P)_k)t_w$.  

\begin{remark}
Ocneanu defined the tube algebra \cite{evk} using Turaev-Viro topological quantum field theory, but it is not hard to see that our definition is equivalent.  
\end{remark}

For $v \in \Gamma_+$ and $n \geq 0$, define $H_n(v) \subset P_{n+k}$ to be the range of the tangle:
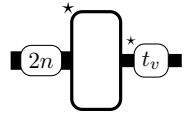
\begin{equation*}
 \begin{tikzpicture}[scale=.65]
  \draw[Box] (0,0) rectangle (1,2); \node[marked,above left, scale=.9] at (0,2) {};
\draw[verythickline] (0,1) --node[cpr,scale=.8] {$2n$} (-1.25,1); \draw[thickline] (1,1) --node[cpr,scale=.8] (t){$t_v$} (2.25,1); \node[marked,scale=.7,above left] at (t.north west) {};
 \end{tikzpicture}
\end{equation*}
$H_n(v)$ is a Hilbert space with inner product $\langle x,y \rangle_{H_n(v)} =  \frac{\delta^{2k}}{\mu_v}\langle x,y \rangle_{P_{n+k}}$.

Observe that 
\begin{equation*}
 \biggl(\bigoplus_{v,w \in \Gamma_+} Tube(\mc P)_{v,w} \otimes (H_n(v) \otimes H_m(w)^*)\biggr)_{n,m \geq 0}
\end{equation*}
has a natural algebroid structure given as follows: the product of $a \otimes x_1 \otimes y_1$ in $Tube(\mc P)_{v,w} \otimes (H_n(v) \otimes H_k(w)^*)$ and $b \otimes x_2 \otimes y_2$ in $Tube(\mc P)_{w',z} \otimes (H_k(w') \otimes H_m(z)^*)$ is given by
\begin{equation*}
 (a \otimes x_1 \otimes y_1)(b \otimes x_2 \otimes y_2) = \delta_{w,w'} \langle y_1,x_2 \rangle_{H_k(w)} \cdot (ab \otimes x_1 \otimes y_2).
\end{equation*}

For $n,m \geq 0$ define $\phi_{n,m}: \bigoplus_{v,w \in \Gamma_+} Tube(\mc P)_{v,w} \otimes (H_n(v) \otimes H_m(w)^*) \to \mc A(\mc P)_{n,m}$ by
\begin{equation*}
\begin{tikzpicture}[scale=.75]
\draw[Box] (-2.5,-1.25) rectangle (5.5,2.25); \node[marked,above left] at (-2.5,2.25) {};
\draw[Box] (0,0) rectangle (1,1); \node at (.5,.5){$a$}; \node[marked,above left,scale=.9] at (0,1) {};
\draw[Box] (2.25,0) rectangle (3,1); \node at (2.625,.5) {$y^*$}; \node[marked,above right,scale=.9] at (3,1) {};
\draw[Box] (-2,0) rectangle (-1.25,1); \node at (-1.625,.5) {$x$}; \node[marked,above left,scale=.9] at (-2,1) {};
\draw[Box] (3.5,0) rectangle (4.5,1); \node[marked,above left] at (3.5,1) {};
\draw[thickline] (.5,1) arc(180:90:1.25cm and .75cm) -- ++(2,0) arc(90:0:1.25cm and .75cm)--  (5,0) arc(0:-90:1.25cm and .75cm) -- ++(-2,0) arc(270:180:1.25cm and .75cm);

\draw[thickline] (1,.5) --node[cpr,scale=.75] (tw) {$t_w$} (2.25,.5);  \node[marked,scale=.6,above right] at (tw.north east) {};
\draw[thickline] (3,.5) -- (3.5,.5); \draw[thickline] (-2,.5) -- (-2.5,.5);
\draw[thickline] (0,.5) --node[cpr,scale=.75] (tv) {$t_v$} (-1.25,.5); \node[marked,scale=.6,above left] at (tv.north west) {};
\node[left] at (-3,0.5) {$\phi_{n,m}(a \otimes x \otimes y^*) = $};
\end{tikzpicture}
\end{equation*}
for $a \otimes x \otimes y^* \in Tube(\mc P)_{v,w} \otimes (H_n(v) \otimes H_m(w)^*)$. 

\begin{proposition}\label{tube-algebroid}
$\phi = (\phi_{n,m})$ gives an isomorphism of algebroids
\begin{equation*}
 \mc A^+(\mc P) \simeq \biggl(\bigoplus_{v,w \in \Gamma_+} Tube(\mc P)_{v,w} \otimes (H_n(v) \otimes H_m(w)^*) \biggr)_{n,m}.
\end{equation*}
\end{proposition}

\begin{proof}
That $(\phi_{n,m})$ is compatible with composition follows easily from the identity
\begin{equation*}
 \begin{tikzpicture}[scale=.65]
  \draw[Box] (0,0) rectangle (1,2); \node at (.5,1) {$y^*$}; \node[marked,scale=.8,above left] at (0,2) {};
\draw[Box] (1.5,0) rectangle (2.5,2); \node at (2,1) {$x$}; \node[marked,scale=.8,above left] at (1.5,2) {};
\draw[thickline] (1,1) -- (1.5,1);
\draw[thickline] (2.5,1) --node[cpr,scale=.75] (tw) {$t_w$} (4,1); \node[marked,scale=.6,above left] at (tw.north west) {};
\draw[thickline] (0,1) --node[cpr,scale=.75] (tv) {$t_v$} (-1.5,1); \node[marked,scale=.6,above left] at (tv.north west) {};

\node[right] at (4.5,1) {$= \delta_{v,w} \langle y,x \rangle_{H_n(v)} \cdot t_v,$};
 \end{tikzpicture}
\end{equation*}
which holds since $t_v,t_w$ are minimal projections.

It remains to show that $\phi_{n,m}$ is a linear isomorphism for each $n,m$.  Suppose we have an element 
\begin{equation*}
 z = \sum_{v,w \in \Gamma_+} a_{v,w} \otimes x_v \otimes y_w^* \in \bigoplus_{v,w \in \Gamma_+} Tube(\mc P)_{v,w} \otimes (H_n(v) \otimes H_m(w)^*).
\end{equation*}
Then we have 
\begin{equation*}
\phi_{n,m}( (t_v \otimes t_v \otimes x_v^*) z (t_w \otimes y_w \otimes t_w)) = \phi_{k,k}(a_{v,w} \otimes |x_v|^2\cdot t_v \otimes |y_w|^2\cdot t_w) = |x_v|^2 |y_w|^2 \cdot a_{v,w}.
\end{equation*}  
By the above it follows that $\phi_{n,m}(z) = 0$ implies $z =0$, so $\phi_{n,m}$ is injective.  Surjectivity follows from Proposition \ref{depthspan} and the fact that $1_{P_{2k}}$ is contained in the linear span of $\{xt_v y: x,y \in P_{2k}, v \in \Gamma_+\}$.
 
\end{proof}

We can now determine the irreducible Hilbert modules over $\mc A^+(\mc P)$.  It is clear that any minimal central projection of $Tube(\mc P)$ must be contained in $Tube(\mc P)_{v,v}$ for some $v \in \Gamma_+$.  Suppose that $p$ is such a projection, and let $V$ be a finite-dimensional Hilbert space on which $p(Tube(\mc P))p$ acts irreducibly.  Extend this action to $Tube(\mc P)$ by letting $1-p$ act by zero.  For $n \geq 0$ define $V_n = V \otimes H_n(v)$.  Then $\mc V = (V_n)_{n \geq 0}$ has a natural Hilbert module structure over
\begin{equation*}
 \biggl(\bigoplus_{w,z \in \Gamma_+} Tube(\mc P)_{w,z} \otimes (H_n(w) \otimes H_m(z)^*)\biggr)_{n,m}
\end{equation*}
We can use the isomorphism of Proposition \ref{tube-algebroid} to make $\mc V$ into a Hilbert module over $\mc A^+(\mc P)$.

\begin{theorem}
$\mc V = (V_n)_{n \geq 0}$ is an irreducible Hilbert module over $\mc A^+(\mc P)$, of weight $d(*,v) / 2$.  Moreover, any irreducible Hilbert module over $\mc A^+(\mc P)$ is of this form.
\end{theorem}

\begin{proof}
For fixed $n \geq 0$,  the tensor product of the actions of $Tube(\mc P)$ on $V$ and $H_n(v) \otimes H_n^*(v) \simeq \mc B(H_n(v))$ on $H_n(v)$ is clearly irreducible, and therefore the action of $\mc A(\mc P)_n$ on $V_n = V \otimes H_n(v)$ is irreducible as well.  It follows from \cite{annular} that $V = (V_n)$ is an irreducible Hilbert module over $\mc A^+(\mc P)$.  The weight of this module is the smallest value of $n$ for which $H_n(v)$ is non-trivial, which is equal to $d(*,v) / 2$.  

Now by \cite{annular}, any irreducible Hilbert module $\mc V = (V_n)$ over $\mc A^+(\mc P)$ must have the property that $V_n$ is an irreducible module over $\mc A(\mc P)_n$ for each $n \geq 0$.  Moreover, if $\mc W = (W_n)$ is another irreducible Hilbert module such that $V_n$ and $W_n$ are isomorphic $\mc  A(\mc P)_n$-modules for some $n$, then $\mc V$ and $\mc W$ are isomorphic modules over $\mc A^+(\mc P)$.  Since the isomorphism classes of irreducible modules over the multimatrix algebra $\mc A(\mc P)_n$ correspond to central projections, it follows from Proposition \ref{tube-algebroid} that every irreducible Hilbert module over $\mc A^+(\mc P)$ is of the form described above.
\end{proof}

\noindent\textbf{The principal and dual principal graphs:}

We will relate the lattice of higher relative commutants for $\mc M_0 \subset \mc M_1$ with the affine category, which will allow us to compute the principal and dual principal graphs.  First we need to further analyze the map $\psi_k$ defined above.  We will assume that $2k \geq depth(\mc P)$, so that $\psi_k$ is surjective.  

First we show that the kernel is determined by the projection $p = p^{(1)} \otimes p^{(2)}$ from Section \ref{sec:skein}.  For this we need the following lemma.

\begin{lemma}\label{pcommutator}
Let $x \in P_{2(n+k)}$ and $y \in P_{2k}$, then we have
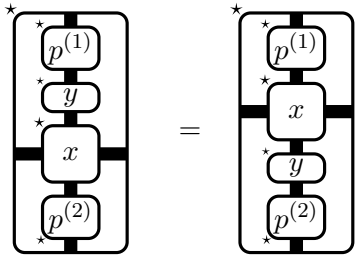
\begin{equation*}
\begin{tikzpicture}[scale=.75]
 \draw[Box] (0,0) rectangle (1,1); \node at (.5,.5) {$x$}; \node[marked,scale=.8,above left] at (0,1) {};
\draw[Box] (0,1.25) rectangle (1,1.75); \node at (.5,1.5) {$y$}; \node[marked,scale=.6,above left] at (0,1.75) {};
\draw[Box] (0,2) rectangle (1,2.75); \node at (.5,2.375) {$p^{(1)}$}; \node[marked,scale=.7,above left] at (0,2.75) {};
\draw[Box] (0,-.25) rectangle (1,-1); \node at (.5,-.625) {$p^{(2)}$}; \node[marked,scale=.7,below left] at (0,-1) {};

\foreach \y in {3,2,1.25,0,-1} {\draw[thickline] (.5,\y) -- ++(0,-.25);}
\draw[thickline] (-.5,.5) -- (0,.5); \draw[thickline] (1,.5) -- (1.5,.5);
\draw[Box] (-.5,-1.25) rectangle (1.5,3); \node[marked,scale=.9,above left] at (-.5,3) {};

\begin{scope}[xshift=4cm]
 \draw[Box] (0,.75) rectangle (1,1.75); \node at (.5,1.25) {$x$}; \node[marked,scale=.8,above left] at (0,1.75) {};
\draw[Box] (0,0) rectangle (1,.5); \node at (.5,.25) {$y$}; \node[marked,scale=.6,above left] at (0,.5) {};
\draw[Box] (0,2) rectangle (1,2.75); \node at (.5,2.375) {$p^{(1)}$}; \node[marked,scale=.7,above left] at (0,2.75) {};
\draw[Box] (0,-.25) rectangle (1,-1); \node at (.5,-.625) {$p^{(2)}$}; \node[marked,scale=.7,below left] at (0,-1) {};

\foreach \y in {3,2,.75,0,-1} {\draw[thickline] (.5,\y) -- ++(0,-.25);}
\draw[thickline] (-.5,1.25) -- (0,1.25); \draw[thickline] (1,1.25) -- (1.5,1.25);
\draw[Box] (-.5,-1.25) rectangle (1.5,3); \node[marked,scale=.9,above left] at (-.5,3) {};
\node[left] at (-1,.875) {$=$};
\end{scope}
\end{tikzpicture}
\end{equation*}

\end{lemma}

\begin{proof}
By Proposition \ref{depthspan}, it suffices to show that
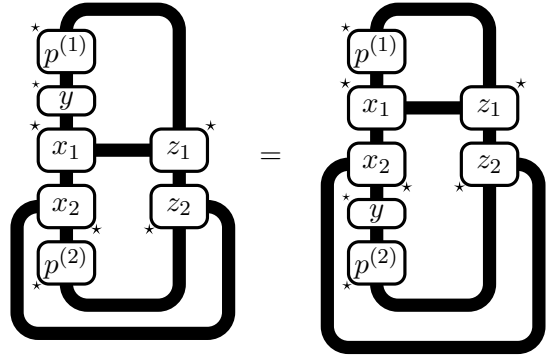
\begin{equation*}
\begin{tikzpicture}[scale=.75]
\draw[Box] (0,1.25) rectangle (1,1.75); \node at (.5,1.5) {$y$}; \node[marked,scale=.6,above left] at (0,1.75) {};
\draw[Box] (0,2) rectangle (1,2.75); \node at (.5,2.375) {$p^{(1)}$}; \node[marked,scale=.7,above left] at (0,2.75) {};
\draw[Box] (0,.25) rectangle (1,1); \node at (.5,.625) {$x_1$}; \node[marked,scale=.8,above left] at (0,1) {};
\draw[Box] (0,-.75) rectangle (1,0); \node at (.5,-.375) {$x_2$}; \node[marked,scale=.8,below right] at (1,-.75) {};

\draw[Box] (2,.25) rectangle (3,1); \node at (2.5,.625) {$z_1$}; \node[marked,scale=.8,above right] at (3,1) {};
\draw[Box] (2,-.75) rectangle (3,0); \node at (2.5,-.375) {$z_2$}; \node[marked,scale=.8,below left] at (2,-.75) {};

\draw[Box] (0,-1) rectangle (1,-1.75); \node at (.5,-1.375) {$p^{(2)}$}; \node[marked,scale=.7,below left] at (0,-1.75) {};

\foreach \y in {2,1.25,.25,-.75} {\draw[thickline] (.5,\y) -- ++(0,-.25);}
\draw[thickline] (2.5,.25) -- (2.5,0);

\draw[thickline] (.5,2.75) arc(180:90:.375cm) --++(1.25,0) arc(90:0:.375cm) -- ++(0,-1.75);
\draw[thickline] (.5,-1.75) arc(180:270:.375cm) -- ++(1.25,0) arc(-90:0:.375cm) -- ++(0,1);
\draw[thickline] (0,-.375) arc(90:180:.375cm) -- ++(0,-1.5) arc(180:270:.375cm) -- ++(3,0) arc(-90:0:.375cm) -- ++(0,1.5) arc(0:90:.375cm);
\draw[thickline] (1,.625) -- (2,.625);

\begin{scope}[xshift=5.5cm]
\draw[Box] (0,-.75) rectangle (1,-.25); \node at (.5,-.5) {$y$}; \node[marked,scale=.6,above left] at (0,-.25) {};
\draw[Box] (0,2) rectangle (1,2.75); \node at (.5,2.375) {$p^{(1)}$}; \node[marked,scale=.7,above left] at (0,2.75) {};
\draw[Box] (0,1) rectangle (1,1.75); \node at (.5,1.375) {$x_1$}; \node[marked,scale=.8,above left] at (0,1.75) {};
\draw[Box] (0,0) rectangle (1,.75); \node at (.5,.375) {$x_2$}; \node[marked,scale=.8,below right] at (1,0) {};

\draw[Box] (2,1) rectangle (3,1.75); \node at (2.5,1.375) {$z_1$}; \node[marked,scale=.8,above right] at (3,1.75) {};
\draw[Box] (2,0) rectangle (3,.75); \node at (2.5,.375) {$z_2$}; \node[marked,scale=.8,below left] at (2,0) {};

\draw[Box] (0,-1) rectangle (1,-1.75); \node at (.5,-1.375) {$p^{(2)}$}; \node[marked,scale=.7,below left] at (0,-1.75) {};

\foreach \y in {2,1,0,-.75} {\draw[thickline] (.5,\y) -- ++(0,-.25);}
\draw[thickline] (2.5,1) -- (2.5,.75);

\draw[thickline] (.5,2.75) arc(180:90:.375cm) --++(1.25,0) arc(90:0:.375cm) -- ++(0,-1);
\draw[thickline] (.5,-1.75) arc(180:270:.375cm) -- ++(1.25,0) arc(-90:0:.375cm) -- ++(0,1.75);
\draw[thickline] (0,.375) arc(90:180:.375cm) -- ++(0,-2.5) arc(180:270:.375cm) -- ++(3,0) arc(-90:0:.375cm) -- ++(0,2.5) arc(0:90:.375cm);
\draw[thickline] (1,1.375) -- (2,1.375);
\node[left] at (-1,.5) {$=$};

\end{scope}
\end{tikzpicture}
\end{equation*}
 for $x_1,x_2,z_1,z_2 \in P_{n+2k}$.  By Lemma \ref{2cleaver}, the left hand side is equal to
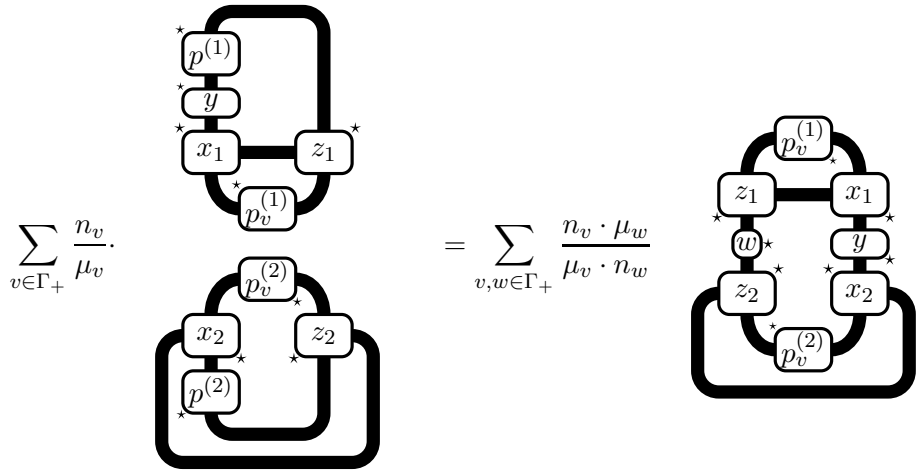
\begin{equation*}
\begin{tikzpicture}[scale=.75]
\draw[Box] (0,1.25) rectangle (1,1.75); \node at (.5,1.5) {$y$}; \node[marked,scale=.6,above left] at (0,1.75) {};
\draw[Box] (0,2) rectangle (1,2.75); \node at (.5,2.375) {$p^{(1)}$}; \node[marked,scale=.7,above left] at (0,2.75) {};
\draw[Box] (0,.25) rectangle (1,1); \node at (.5,.625) {$x_1$}; \node[marked,scale=.8,above left] at (0,1) {};
\draw[Box] (2,.25) rectangle (3,1); \node at (2.5,.625) {$z_1$}; \node[marked,scale=.8,above right] at (3,1) {};
\draw[thickline] (.5,2.75) arc(180:90:.375cm) --++(1.25,0) arc(90:0:.375cm) -- ++(0,-1.75);
\draw[thickline] (1,.625) -- (2,.625);

\begin{scope}[yshift=-2.25cm]
\draw[Box] (2,-.75) rectangle (3,0); \node at (2.5,-.375) {$z_2$}; \node[marked,scale=.8,below left] at (2,-.75) {};
\draw[Box] (0,-.75) rectangle (1,0); \node at (.5,-.375) {$x_2$}; \node[marked,scale=.8,below right] at (1,-.75) {};
\draw[thickline] (.5,-1.75) arc(180:270:.375cm) -- ++(1.25,0) arc(-90:0:.375cm) -- ++(0,1);
\draw[thickline] (0,-.375) arc(90:180:.375cm) -- ++(0,-1.5) arc(180:270:.375cm) -- ++(3,0) arc(-90:0:.375cm) -- ++(0,1.5) arc(0:90:.375cm);
\draw[Box] (0,-1) rectangle (1,-1.75); \node at (.5,-1.375) {$p^{(2)}$}; \node[marked,scale=.7,below left] at (0,-1.75) {};
\end{scope}

\foreach \y in {2,1.25,-3} {\draw[thickline] (.5,\y) -- ++(0,-.25);}

\draw[Box] (1,-.75) rectangle (2,0); \node at (1.5,-.375) {$p_v^{(1)}$}; \node[marked,scale=.7,above left] at (1,0) {};
\draw[Box] (1,-2) rectangle (2,-1.25); \node at (1.5,-1.625) {$p_v^{(2)}$}; \node[marked,scale=.7,below right] at (2,-2) {};

\draw[thickline] (.5,.25) arc(180:270:.5cm and .625cm); \draw[thickline] (2.5,.25) arc(0:-90:.5cm and .625cm);
\draw[thickline] (.5,-2.25) arc(180:90:.5 and .625); \draw[thickline] (2.5,-2.25) arc(0:90:.5 and .625);

\node[left] at (-1,-1.25) {$\displaystyle\sum_{v \in \Gamma_+} \frac{n_v}{\mu_v} \cdot$};

\begin{scope}[xshift = 9.5cm,yshift = -1.5cm]
\draw[Box] (0,1) rectangle (1,1.75); \node at (.5,1.375) {$z_1$}; \node[marked,scale=.8,below left] at (0,1) {};
\draw[Box] (2,1) rectangle (3,1.75); \node at (2.5,1.375) {$x_1$}; \node[marked,scale=.8, below right] at (3,1) {};
\draw[Box] (.25,.25) rectangle (.75,.75); \node at (.5,.5) {$w$}; \node[marked,scale=.8,right=.05] at (.75,.5) {};
\draw[Box] (2,.25) rectangle (3,.75); \node at (2.5,.5) {$y$}; \node[marked,scale=.8,below right] at (3,.25) {};
\draw[Box] (0,0) rectangle (1,.-.75); \node at (.5,.-.375) {$z_2$}; \node[marked,scale=.8,above right] at (1,0) {};
\draw[Box](2,0) rectangle (3,-.75); \node at (2.5,-.375) {$x_2$}; \node[marked,scale=.8,above left] at (2,0) {};
\draw[Box] (1,2) rectangle (2,2.75); \node at (1.5,2.375) {$p_v^{(1)}$}; \node[marked,scale=.6,below right] at (2,2) {};
\draw[Box] (1,-1) rectangle (2,-1.75); \node at (1.5,-1.375) {$p_v^{(2)}$}; \node[marked,scale=.6,above left] at (1,-1) {};

\foreach \x in {.5,2.5} {\foreach \y in {1,.25} {   
\draw[thickline] (\x,\y) -- ++(0,-.25);  }}

\draw[thickline] (.5,1.75) arc(180:90:.5 and .625); \draw[thickline] (2.5,1.75) arc(0:90:.5 and .625);
\draw[thickline] (.5,-.75) arc(180:270:.5 and .625); \draw[thickline] (2.5,.-.75) arc(0:-90:.5 and .625);
\draw[thickline] (1,1.375) -- (2,1.375);
\draw[thickline] (0,-.375) arc(90:180:.375cm) --++(0,-1) arc(180:270:.375cm) -- ++(3,0) arc(-90:0:.375) --++(0,1) arc(0:90:.375);

\node[left] at (-1,.25) {$= \displaystyle\sum_{v,w \in \Gamma_+} \frac{n_v\cdot\mu_w}{\mu_v\cdot n_w}$};
\end{scope}

\end{tikzpicture}
\end{equation*}
Where the equality above follows from rotating the top and bottom components on the left hand side by 180 degrees and then applying Lemma \ref{2cleaver} again.  But if we applied the same procedure to the right hand side of the first equation it is clear that we would obtain the same diagram, so the result follows.
\end{proof}

Given $x \in P_{2(n+k)}$, define $p \cdot x$ by
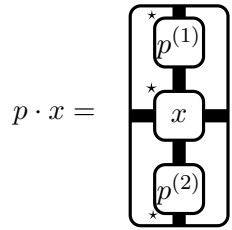
\begin{equation*}
\begin{tikzpicture}[scale=.65]
\draw[Box] (0,0) rectangle (1,1); \node at (.5,.5) {$x$}; \node[marked,above left,scale=.8] at (0,1) {};
\draw[Box] (0,1.5) rectangle (1,2.5); \node at(.5,2) {$p^{(1)}$}; \node[marked,above left, scale=.7] at (0,2.5) {};
\draw[Box] (0,-.5) rectangle (1,-1.5); \node at (.5,-1) {$p^{(2)}$}; \node[marked,below left,scale=.7] at (0,-1.5) {};

\foreach \y in {0,1.5} {\draw[thickline] (.5,\y) -- ++(0,-.5);}
\foreach \y in {-1.5,2.75} {\draw[thickline] (.5,\y) -- ++(0,-.25);}
\foreach \x in {-.5,1}{\draw[thickline] (\x,.5) --++(.5,0); }

\draw[Box] (-.5,-1.75) rectangle (1.5,2.75);

\node[left] at (-1,.5) {$p \cdot x =$};
\end{tikzpicture}
\end{equation*}

\begin{proposition}
If $x \in P_{2(n+k)}$ then $\psi_k(x) = \psi_k(p\cdot x)$.  Moreover, if $\psi_k(x) = 0$ then $p \cdot x = 0$.
\end{proposition}
 
\begin{proof}
The first statement follows immediately from Proposition \ref{p-skein} (4).  The second follows from combining Lemmas \ref{affineiso} and \ref{pcommutator}.
\end{proof}

It follows that $\psi_k$ restricts to a linear isomorphism of $\{p \cdot x: x \in P_{2(n+k)}\}$ onto $\mc A(\mc P)_n$.  We now want to pull back the multiplication from $\mc A(\mc P)_n$.  To do this it is convenient to modify the map $\psi_k$.  So we define $\tilde \psi_k:P_{2(n+k)} \to \mc A(\mc P)_n$ by
\begin{equation*}
\begin{tikzpicture}[scale=.75]
\draw[Box] (-.5,-1.25) rectangle (3.6,2.25); \node[marked,above left] at (-.5,2.25) {};
\draw[Box] (0,0) rectangle (1,1); \node at (.5,.5){$x$}; \node[marked,above left] at (0,1) {};
\draw[Box] (1.5,0) rectangle (2.5,1); \node[marked,above left] at (1.5,1) {};
\draw[thickline] (.5,1) arc(180:0:1.25cm and .75cm) -- node[cpr](v) {$v$} (3,0) arc(0:-180:1.25cm and .75cm); \node[right=.03,scale=.6,marked] at (v.east) {};
\draw[thickline] (1,.5) -- (1.5,.5); \draw[thickline] (0,.5) -- (-.5,.5);
\node[left] at (-1,0.3) {$\tilde\psi_k(x) = \displaystyle\sum_{v \in \Gamma_+} \sqrt{\frac{\mu_v}{n_v}}\;\cdot$};
\end{tikzpicture} 
\end{equation*}
It is clear that $\tilde \psi_k$ has the same kernel as $\psi_k$, so $\tilde \psi_k$ is a linear isomorphism when restricted to $\{p \cdot x: x \in P_{2(n+k)}\}$.  For $n$ even we also define $\tilde \psi_k^{op}:P_{2(n+k)} \to \mc A(\mc P)_n$ by
\begin{equation*}
\begin{tikzpicture}[scale=.75]
\draw[Box] (0,0) rectangle (1,1); \node[marked,above left, scale=.9] at (0,1) {};
\draw[Box] (2,0) rectangle (3,1); \node[marked,above left, scale=.9] at (2,1) {}; \node at (2.5,.5) {$x$};

\draw[thickline] (2,.75) arc(270:180:.375cm) arc(0:90:.375cm) -- ++(-1.25,0) arc(90:270:.375cm);
\draw[thickline] (2,.25) arc(90:180:.375cm) arc(0:-90:.375cm) --++(-1.25,0) arc(270:90:.375cm);

\draw[thickline] (3,.75) arc(-90:90:.45cm and .75cm) -- ++(-3.75,0) arc(90:180:.75cm) arc(0:-90:.75cm);
\draw[thickline] (3,.25) arc(90:-90:.45cm and .75cm) --++(-3.75,0) arc(270:180:.75cm) arc(0:90:.75cm);

\draw[thickline] (2.5,1) arc(0:90:.9cm) -- ++(-1.75,0) arc(90:180:.9cm) --node[cpr] (v) {$v$} ++(0,-1) arc(180:270:.9cm) -- ++(1.75,0) arc(-90:0:.9cm); \node[marked,scale=.6,left=.05] at (v.west) {};

\draw[Box] (-2.25,-1.625) rectangle (4,2.625); \node[marked,scale=.9,above left] at (-2.25,2.625) {};

\node[left] at(-2.5,.5) {$\tilde \psi_k^{op}(x) = \displaystyle \sum_{v \in \Gamma_+} \sqrt{\frac{\mu_v}{n_v}}\;\cdot$};
\end{tikzpicture}
\end{equation*}
Note that we require $n$ even because of the shading.  It is clear that $\tilde \psi_k^{op}$ is also a linear isomorphism when restricted to $\{p\cdot x: x \in P_{2(n+k)}\}$.

Now for $x,y \in P_{2(n+k)}$ define
\begin{equation*}
\begin{tikzpicture}[scale=.65]
\draw[Box] (0,0) rectangle (1,1); \node at (.5,.5) {$x$}; \node[marked,scale=.8,above left] at (0,1) {};
\draw[Box] (2,0) rectangle (3,1); \node at (2.5,.5) {$y$}; \node[marked,scale=.8,above left] at (2,1) {};
\draw[Box] (1,1.25) rectangle (2,2); \node[scale=.9] at (1.5,1.625) {$q^{(1)}$}; \node[marked,scale=.6,above left] at (1,2) {};

\draw[Box] (1,-.25) rectangle (2,-1); \node[scale=.9] at (1.5,-.625) {$q^{(2)}$}; \node[marked,scale=.6,below left] at (1,-1) {};
\draw[thickline] (.5,1) arc(180:90:.5cm and .625cm); \draw[thickline] (2,1.625) arc(90:0:.5cm and .625cm);
\draw[thickline] (.5,0) arc(180:270:.5cm and .625cm); \draw[thickline] (2,-.625) arc(-90:0:.5cm and .625cm);

\draw[thickline] (1.5,2) --++(0,.25); \draw[thickline] (1.5,-1) --++(0,-.25);

\draw[thickline] (-.5,.5) -- (0,.5); \draw[thickline] (1,.5) -- (2,.5); \draw[thickline] (3,.5) -- (3.5,.5);

\draw[Box] (-.5,-1.25) rectangle (3.5,2.25); \node[marked,scale=.9,above left] at (-.5,2.25) {};

\node[left] at (-1,.5) {$x \star_q y = $};
\end{tikzpicture}
\end{equation*}

\begin{proposition}
For $x, y \in P_{2(n+k)}$ we have 
\begin{align*}
\tilde \psi_k(x \star_q y) &= \tilde \psi_k(x) \cdot \tilde \psi_k(y)\\
\tilde \psi_k^{op}(y \star_q x) &= \tilde \psi_k^{op}(x) \cdot \tilde \psi_k^{op}(y).
\end{align*}
\end{proposition}

\begin{proof}
 This follows from Proposition \ref{q-connect}.
\end{proof}

\begin{corollary}\label{affine_lift}
We have an embeddings $\mc A(\mc P)_{n} \hookrightarrow P_{2n+4k}$ and $\mc A(\mc P)_{n} \hookrightarrow P_{2n+4k}^{op}$ given by
\begin{equation*}
 \begin{tikzpicture}[scale=.65]
\draw[Box] (0,0) rectangle (1,1); \node at (.5,.5) {$x$}; \node[marked,above left,scale=.8] at (0,1) {};
\draw[Box] (0,1.5) rectangle (1,2.5); \node at(.5,2) {$q^{(1)}$}; \node[marked,above left, scale=.7] at (0,2.5) {};
\draw[Box] (0,-.5) rectangle (1,-1.5); \node at (.5,-1) {$q^{(2)}$}; \node[marked,below left,scale=.7] at (0,-1.5) {};

\foreach \y in {0,1.5} {\draw[thickline] (.5,\y) -- ++(0,-.5);}
\foreach \x in {-.5,1}{\foreach \y in {-1,.5,2} {\draw[thickline] (\x,\y) --++(.5,0); }}

\draw[Box] (-.5,-1.75) rectangle (1.5,2.75);

\node[left] at (-1,.5) {$\tilde \psi_k(x) \mapsto $};

\begin{scope}[xshift=8.25cm]
\draw[Box] (0,0) rectangle (1,1); \node at (.5,.5) {$x$}; \node[marked,above left,scale=.8] at (0,1) {};
\draw[Box] (0,1.5) rectangle (1,2.5); \node at(.5,2) {$q^{(1)}$}; \node[marked,above left, scale=.7] at (0,2.5) {};
\draw[Box] (0,-.5) rectangle (1,-1.5); \node at (.5,-1) {$q^{(2)}$}; \node[marked,below left,scale=.7] at (0,-1.5) {};

\foreach \y in {0,1.5} {\draw[thickline] (.5,\y) -- ++(0,-.5);}
\foreach \x in {-.5,1}{\foreach \y in {-1,.5,2} {\draw[thickline] (\x,\y) --++(.5,0); }}

\draw[Box] (-.5,-1.75) rectangle (1.5,2.75);

\node[left] at (-1,.5) {and $\;\;\;\;\tilde \psi_k^{op}(x) \mapsto $};
\end{scope}

\end{tikzpicture}
\end{equation*}
for $x \in P_{2(n+k)}$.
\end{corollary}

\begin{proof}
This follows from Corollary \ref{2q-rotate} and the previous proposition.
\end{proof}

\begin{theorem}\label{commsquare}
For $n \geq 2$ we have an isomorphism of commuting squares:
\begin{equation*}
\begin{bmatrix}
\mc M_0' \cap \mc M_{2n-1} & \subset & \mc M_0' \cap \mc M_{2n}\\
\cup & &\cup\\
 \mc M_1' \cap \mc M_{2n-1} &\subset & \mc M_1' \cap \mc M_{2n}
\end{bmatrix} \simeq \begin{bmatrix}
p_{n-1}(P_{4(n-1)k})p_{n-1} & \subset & p_n(P_{2nk} \otimes P_{2nk}^{op})p_n \\
\cup & & \cup\\
p_{n-2}(\mc A(\mc P)_{2(n-2)k})p_{n-2} &\subset &p_{n-1}(P_{4(n-1)k}^{op})p_{n-1} 
\end{bmatrix}
\end{equation*}

\end{theorem}

\begin{proof}
This isomorphisms follow from Proposition \ref{commutants} and Theorem \ref{symmlift}.  The compatibility of the inclusions can be easily checked using the description of $Sym(\mc P)$ from the previous section, the details are left to the reader.
\end{proof}

\begin{remark}
We also have an isomorphism of the dual squares
\begin{equation*}
\begin{bmatrix}
\mc M_0' \cap \mc M_{2n} & \subset & \mc M_0' \cap \mc M_{2n+1}\\
\cup & &\cup\\
 \mc M_1' \cap \mc M_{2n} &\subset & \mc M_1' \cap \mc M_{2n+1}
\end{bmatrix} \simeq \begin{bmatrix}
p_n(P_{2nk} \otimes P_{2nk}^{op})p_n & \subset & p_{n}(P_{4nk})p_{n}  \\
\cup & & \cup\\
p_{n-1}(P_{4(n-1)k}^{op})p_{n-1} &\subset & p_{n-1}(\mc A(\mc P)_{2(n-1)k})p_{n-1} 
\end{bmatrix}
\end{equation*}
where the inclusion $P^{op}_{4(n-1)k} \subset \mc A(\mc P)_{2(n-1)k}$ is given by:
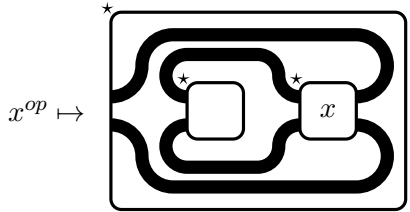
\begin{equation*}
\begin{tikzpicture}[scale=.75]
\draw[Box] (0,0) rectangle (1,1); \node[marked,above left, scale=.9] at (0,1) {};
\draw[Box] (2,0) rectangle (3,1); \node[marked,above left, scale=.9] at (2,1) {}; \node at (2.5,.5) {$x$};

\draw[thickline] (2,.75) arc(270:180:.375cm) arc(0:90:.375cm) -- ++(-1.25,0) arc(90:270:.375cm);
\draw[thickline] (2,.25) arc(90:180:.375cm) arc(0:-90:.375cm) --++(-1.25,0) arc(270:90:.375cm);

\draw[thickline] (3,.75) arc(-90:90:.55cm ) -- ++(-3.25,0) arc(90:180:.55cm) arc(0:-90:.55cm);
\draw[thickline] (3,.25) arc(90:-90:.55cm) --++(-3.25,0) arc(270:180:.55cm) arc(0:90:.55cm);

\draw[Box] (-1.35,-1.25) rectangle (3.875,2.25); \node[marked,scale=.9,above left] at (-1.35,2.25) {};

\node[left] at(-1.625,.5) {$x^{op} \mapsto$};
\end{tikzpicture}
\end{equation*}

\end{remark}

We can now compute the principal and dual principal graphs by determining the Bratelli diagrams for the above inclusions.

\begin{proposition}\label{brat}
We have the following Bratelli diagrams.
\begin{enumerate}
 \item For $n \geq 2$ the Bratelli diagram of $P_{4(n-1)k} \subset P_{2nk} \otimes P_{2nk}^{op}$ is as follows. The even vertices are indexed by pairs of $M_0-M_0$ bimodules $(X_v,\overline X_w)$ for $v,w \in \Gamma_+$.  The odd vertices are indexed by $M_0-M_0$ bimodules $X_z$, $z \in \Gamma_+$.  The number of edges between $(X_v,\overline X_w)$ and $X_z$ is equal to the multiplicity of $X_z$ in $X_v \otimes \overline X_w$, i.e. the dimension of the range of the partially labelled tangle:
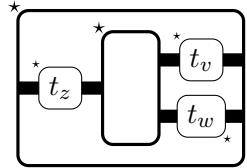
\begin{equation*}
\begin{tikzpicture}[scale=.75]
\draw[Box] (0,0) rectangle (1,2); \node[marked,scale=.9,above left] at (0,2) {};
\draw[thickline] (1,1.5) --node[cpr] (tv) {$t_v$} ++(1.5,0); \node[marked,scale=.6,above left] at (tv.north west) {};
\draw[thickline] (1,.5) --node[cpr] (tw) {$t_w$} ++(1.5,0); \node[marked,scale=.6,below right] at (tw.south east) {};
\draw[thickline] (0,1) --node[cpr] (tz) {$t_z$} ++(-1.5,0); \node[marked,scale=.6,above left] at (tz.north west) {};
\draw[Box] (-1.5,-.375) rectangle (2.5,2.375); \node[marked,scale=.9,above left] at (-1.5,2.375) {};
\end{tikzpicture}
\end{equation*}
\item For $n \geq 1$ the Bratelli diagram of $  P_{4nk}^{op} \subset \mc A(\mc P)_{2nk}$ is as follows.  The odd vertices are indexed by $M_0-M_0$ bimodules $\overline X_z$ for $z \in \Gamma_+$.  The even vertices are indexed by irreducible Hilbert modules over $\mc A^+(\mc P)$, or equivalently by minimal central projections in Ocneanu's tube algebra $Tube(\mc P)$.  Given an irreducible Hilbert module $V$ over $\mc A^+(\mc P)$, let $\pi_V$ be a minimal projection in the central component of $Tube(\mc P)$ corresponding to $V$ via Theorem \ref{affineiso}.  Then the number of edges between $\overline X_z$ and $V$ is equal to the dimension of the range of the following partially labelled affine tangle, viewed as a linear map $P_{4k} \to \mc A(\mc P)_{k}$.
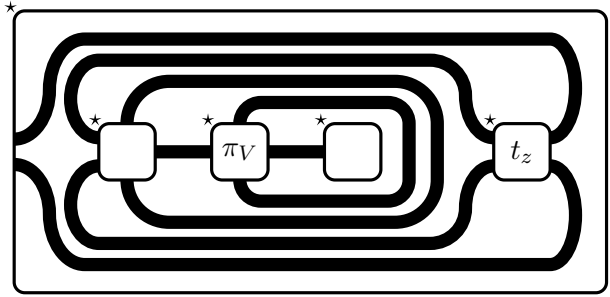
\begin{equation*}
 \begin{tikzpicture}[scale=.75]
\draw[Box] (-2,0) rectangle (-1,1); \node at (-1.5,.5) {$\pi_V$}; \node[marked,above left,scale=.9] at (-2,1) {};
\draw[Box](-4,0) rectangle (-3,1); \node[marked,above left, scale=.9] at (-4,1) {};
\draw[Box] (0,0) rectangle (1,1); \node[marked,above left, scale=.9] at (0,1) {};
\draw[Box] (3,0) rectangle (4,1); \node[marked,above left, scale=.9] at (3,1) {}; \node at (3.5,.5) {$t_z$};

\draw[thickline] (3,.75) arc(270:180:.5cm and .75cm) arc(0:90:.625cm) -- ++(-5.875,0) arc(90:270:.5cm and .6875cm);
\draw[thickline] (3,.25) arc(90:180:.5cm and .75cm) arc(0:-90:.625cm) --++(-5.875,0) arc(270:90:.5 and .6875);

\draw[thickline] (4,.75) arc(-90:90:.45cm and .875cm) -- ++(-8.25,0) arc(90:180:.625cm and .9cm) arc(0:-90:.625cm and .875cm);
\draw[thickline] (4,.25) arc(90:-90:.45cm and .875cm) --++(-8.25,0) arc(270:180:.625cm and .9cm) arc(0:90:.625cm and .875cm);

\draw[thickline] (-3,.5) -- (-2,.5); \draw[thickline] (-1,.5) -- (0,.5);

\draw[thickline] (-1.5,1) arc(180:90:.375cm) --++(2.25,0) arc(90:0:.375cm) --++(0,-1) arc(0:-90:.375cm) --++(-2.25,0) arc(270:180:.375cm);
\draw[thickline] (-3.5,1) arc(180:90:.75cm) --++(4,0) arc(90:0:.75cm) --++(0,-1) arc(0:-90:.75cm) --++(-4,0) arc(270:180:.75cm);

\draw[Box] (-5.5,-2) rectangle (5,3); \node[marked,scale=.9,above left] at (-5.5,3) {};
\end{tikzpicture}
\end{equation*}

\end{enumerate}

\end{proposition}

\begin{proof}
Recall that if $i:A \to B$ is a unital inclusion of multi-matrix algebras and $p \in A$ (resp. $q \in B$) is a minimal projection in the central component labelled $v$ (resp. $w$) then the number of edges between $v$ and $w$ in the Bratelli diagram of the inclusion is equal to the dimension of the subspace $\{i(p)bq: b \in B\} \subset B$.  The results then follow from the previous proposition and the remark which followed.  
\end{proof}

\begin{remark}
It follows from Theorem \ref{commsquare} that the principal (resp. dual principal) graph of $M_0 \otimes M_0^{op} \subset M_0 \boxtimes M_0^{op}$ is the connected component of $(X_*,X_*)$ (resp. $X_*$) of the Bratelli diagram described in (1) (resp. (2)) above. 

The vertices in the Bratelli diagrams above which are not in the principal or dual principal graphs do not correspond to bimodules arising from the Jones tower $\mc M_n$.   However, these `missing' bimodules are easily constructed.  Indeed, recall that we have $\mc M_3 \simeq p(M_{2k} \boxtimes M_{2k}^{op})p$.  This gives $M_{2k} \boxtimes M_{2k}^{op}$ a $\mc M_i$-$\mc M_j$ bimodule structure for $i,j = 0,1$.  By adapting the arguments above it can be proved that there is an isomorphism of commuting squares: 
\begin{equation*}
\begin{bmatrix}
 \Hom_{\mc M_0,\mc M_1}\bigl(L^2\bigl(M_{2k} \boxtimes M_{2k}^{op}\bigr)\bigr) & \subset & \Hom_{\mc M_0,\mc M_0}\bigl(L^2\bigl(M_{2k} \boxtimes M_{2k}^{op}\bigr)\bigr)\\
\cup & &\cup\\
 \Hom_{\mc M_1,\mc M_1}\bigl(L^2\bigl(M_{2k} \boxtimes M_{2k}^{op}\bigr)\bigr) &\subset &  \Hom_{\mc M_1,\mc M_0}\bigl(L^2\bigl(M_{2k} \boxtimes M_{2k}^{op}\bigr)\bigr)
\end{bmatrix} \simeq \begin{bmatrix}
P_{8k} & \subset & P_{6k} \otimes P_{6k}^{op} \\
\cup & & \cup\\
\mc A(\mc P)_{2k} &\subset &P_{8k}^{op}
\end{bmatrix}
\end{equation*}
It follows that the fusion graphs of irreducible bimodules which appear in $L^2(M_{2k} \boxtimes M_{2k})$ are precisely the Bratelli diagrams of Proposition \ref{brat}.
\end{remark}

\medskip
\noindent\textbf{Fusion rules for$\mc M_1$-$\mc M_1$ bimodules:}

It is well-known that the quantum double of a fusion category is always braided. In the subfactor setting, it was likewise shown by Ocneanu \cite{ocn} and Evans-Kawahigashi \cite{evk} that the fusion category of bimodules arising from the asymptotic inclusion is braided (see also \cite{izumi,mug}).  Below we show how the fusion rules and braiding for $\mc M_1-\mc M_1$ bimodules can be seen in our framework.

Define $v_n \in \mc A(\mc P)_{2nk}$ by
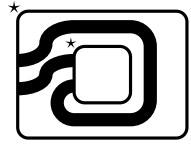
\begin{equation*}
\begin{tikzpicture}[scale=.75]
\draw[Box] (0,0) rectangle (1,1); \node[marked,scale=.8,above left] at (0,1) {};
\draw[Box] (-1,-.625) rectangle (2,1.625);  \node[marked,scale=.9,above left] at (-1,1.625) {};

\draw[thickline] (-1,.25) arc(-90:0:.5cm and .25cm) arc(180:90:.5cm and .25cm);
\draw[thickline] (-1,.75) arc(-90:0:.5 and .3) arc(180:90:.5 and .3) -- ++(1,0) arc(90:0:.35) --++(0,-.95) arc(0:-90:.35) -- ++(-1,0) arc(270:90:.3cm);
\end{tikzpicture}
\end{equation*}

For $x \in P_{2(n+k)}$ define $x^{op}$ by
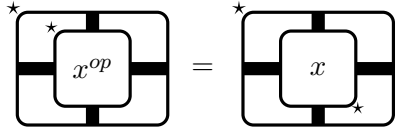
\begin{equation*}
 \begin{tikzpicture}
  \draw[Box] (0,0) rectangle (1,1); \node at (.5,.5) {$x^{op}$}; \node[marked,above left, scale =.9] at (0,1) {};
\draw[thickline] (-.5,.5) -- (0,.5); \draw[thickline] (1,.5) -- (1.5,.5);
\draw[thickline] (.5,1) -- (.5,1.25); \draw[thickline] (.5,0) -- (.5,-.25);
\draw[Box](-.5,-.25) rectangle (1.5,1.25); \node[marked,above left] at (-.5,1.25) {};

\begin{scope}[xshift=3cm]
   \draw[Box] (0,0) rectangle (1,1); \node at (.5,.5) {$x$}; \node[marked,below right, scale =.9] at (1,0) {};
\draw[thickline] (-.5,.5) -- (0,.5); \draw[thickline] (1,.5) -- (1.5,.5);
\draw[thickline] (.5,1) -- (.5,1.25); \draw[thickline] (.5,0) -- (.5,-.25);
\draw[Box](-.5,-.25) rectangle (1.5,1.25); \node[marked,above left] at (-.5,1.25) {};
\node[left] at (-.75,.5) {$=$};
\end{scope}

 \end{tikzpicture}
\end{equation*}
Observe that we have
\begin{align*}
 \tilde \psi_k(x^{op}) &= v_n^*\bigl(\tilde \psi_k(x)\bigr)v_n\\
\tilde \psi_k^{op}(x^{op}) &= v_n^*\bigl(\tilde \psi_k^{op}(x)\bigr)v_n
\end{align*}

Define $i_n: \mc A(\mc P)_{2nk} \otimes \mc A(\mc P)_{2nk}\hookrightarrow \mc A(\mc P)_{4(n+1)k}$ by
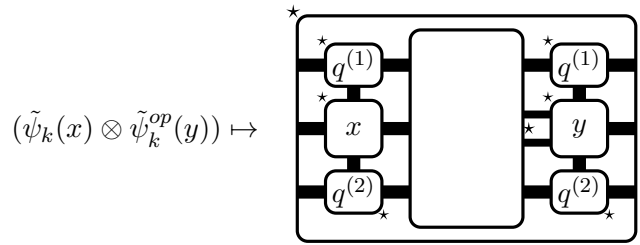
\begin{equation*}
\begin{tikzpicture}[scale=.75]
\draw[Box] (0,0) rectangle (1,1); \node at (.5,.5) {$x$}; \node[marked,above left,scale=.8] at (0,1) {};
\draw[Box] (0,1.25) rectangle (1,2); \node at (.5,1.625) {$q^{(1)}$}; \node[marked,above left, scale=.8] at (0,2) {};
\draw[Box] (0,-.25) rectangle (1,-1); \node at (.5,-.625) {$q^{(2)}$}; \node[marked,below right,scale=.8] at (1,-1) {};
\foreach \x in {-.5,1} {\foreach \y in {-.625,.5,1.625} {\draw[thickline] (\x,\y) -- ++(.5,0);}}
\draw[thickline] (.5,-.25) --++(0,.25); \draw[thickline] (.5,1) --++(0,.25);

\begin{scope}[xshift=4cm]
 \draw[Box] (0,0) rectangle (1,1); \node at (.5,.5) {$y$}; \node[marked,above left,scale=.8] at (0,1) {};
\draw[Box] (0,1.25) rectangle (1,2); \node at (.5,1.625) {$q^{(1)}$}; \node[marked,above left, scale=.8] at (0,2) {};
\draw[Box] (0,-.25) rectangle (1,-1); \node at (.5,-.625) {$q^{(2)}$}; \node[marked,below right,scale=.8] at (1,-1) {};
\foreach \x in {-.5,1} {\foreach \y in {-.625,1.625} {\draw[thickline] (\x,\y) -- ++(.5,0);}}
\draw[thickline] (.5,-.25) --++(0,.25); \draw[thickline] (.5,1) --++(0,.25);

\draw[thickline] (1,.5) --++(.5,0);
\draw[medthick] (0,.75) --++(-.5,0);
\draw[medthick] (0,.25) --++(-.5,0);
\end{scope}

\draw[Box] (1.5,-1.25) rectangle (3.5,2.25); \node[marked,right=.03,scale=.9] at (3.5,.5) {};
\draw[Box] (-.5,-1.5) rectangle (5.5,2.5); \node[marked,above left] at (-.5,2.5) {};

\node[left] at (-1,.5) {$(\tilde \psi_k(x) \otimes \tilde \psi_k^{op}(y)) \mapsto$};
\end{tikzpicture}
\end{equation*}

We then have
\begin{equation*}
v_{2(n+1)}^*\bigl(i_n\bigl(\tilde \psi_k(x) \otimes \tilde \psi_k^{op}(y^{op})\bigr)\bigr)v_{2(n+1)} = i_n\bigl(\tilde \psi_k(y) \otimes \tilde \psi_k^{op}(x^{op})\bigr),
\end{equation*}
which is easy to see by drawing the appropriate diagram.  It follows that if we set 
\begin{equation*}
u_n = v_{2(n+1)}\cdot i_n(v_n \otimes v_n)
\end{equation*}
then we have
\begin{equation*}
 u_n^*(i_n(x \otimes y))u_n = i_n(y \otimes x)
\end{equation*}
for any $x,y \in \mc A(\mc P)_{2nk}$. 

Recall that we have $L^2(\mc M_{2n}) \simeq L^2(\mc M_{n}) \otimes_{\mc M_0} L^2(\mc M_{n})$ as $\mc M_1-\mc M_1$ bimodules.  Moreover, the inclusion $\Hom_{\mc M_1,\mc M_1}(L^2(\mc M_{n})) \otimes \Hom_{\mc M_1,\mc M_1}(L^2(\mc M_{n})) \hookrightarrow \Hom_{\mc M_1,\mc M_1}(L^2(\mc M_{n}))$ is given by the following tangle (the top-left corner is shaded):
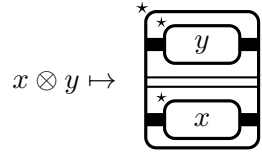
\begin{equation*}
\begin{tikzpicture}[scale=.5]
 \draw[Box] (0,0) rectangle (2,1); \node at (1,.5) {$x$}; \node[marked,above left,scale=.8] at (0,1) {}; 
\draw[Box] (0,2) rectangle (2,3); \node at (1,2.5) {$y$}; \node[marked,above left,scale=.8] at (0,3) {};

\foreach \y in {1.375,1.625} {\draw[thick] (-.5,\y) -- ++(3,0);}
\foreach \y in {.5,2.5} {\foreach \x in {-.5,2} {\draw[thickline] (\x,\y) --++(.5,0); }}

\draw[Box] (-.5,-.25) rectangle (2.5,3.375); \node[marked,above left, scale=.9] at (-.5,3.375) {};
\node[left] at (-1,1.5) {$x \otimes y \mapsto$};
\end{tikzpicture}
\end{equation*}
\begin{proposition}\label{tensinclude}
The following diagram is commutative:
\begin{equation*}
\xymatrix{\mc A(\mc P)_{2(n-2)k} \otimes \mc A(\mc P)_{2(n-2)k} \ar[r]^{i_n} & \mc A(\mc P)_{4(n-1)k}\\
\Hom_{\mc M_1,\mc M_1} L^2(\mc M_{n}) \otimes \Hom_{\mc M_1,\mc M_1}L^2(\mc M_{n}) \ar[u] \ar[r] & \Hom_{\mc M_1,\mc M_1} L^2(\mc M_{2n}) \ar[u]
}
\end{equation*}

\end{proposition}

\begin{proof}
This can be verified by computing the tangle above in the planar algebra $Sym(\mc P)$, using the description from the previous section.
\end{proof}

We are now able to recover the following description of the fusion rules for $\mc M_1-\mc M_1$-bimodules due to Ocneanu \cite{ocn} and Evans-Kawahigashi \cite{evk}.

\begin{corollary}
Let $V,W,Z$ be irreducible $\mc M_1-\mc M_1$-bimodules (corresponding to irreducible Hilbert modules over $\mc A^+(\mc P)$).  Let $\pi_V,\pi_W,\pi_Z$ be minimal projections in the central components of $Tube(\mc P)$ corresponding to $V,W,Z$.  Let $\widetilde \pi_V,\widetilde \pi_W \in P_{4k}$ be such that $\pi_V = \psi_k(\widetilde \pi_V)$ and $\pi_W = \psi_k^{op}(\widetilde \pi_W)$.  Then the multiplicity of $Z$ in $V \otimes_{\mc M_1} W$ is the dimension of the range of the following tangle, viewed as a linear map from $P_{8k}$ to $\mc A(\mc P)_{6k}$:
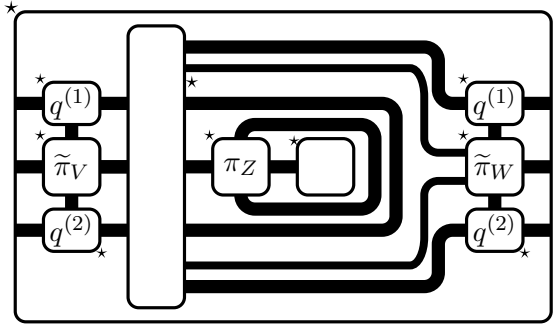
\begin{equation*}
\begin{tikzpicture}[scale=.75]
\draw[Box] (0,0) rectangle (1,1); \node at (.5,.5) {$\widetilde \pi_V$}; \node[marked,above left,scale=.8] at (0,1) {};
\draw[Box] (0,1.25) rectangle (1,2); \node at (.5,1.625) {$q^{(1)}$}; \node[marked,above left, scale=.8] at (0,2) {};
\draw[Box] (0,-.25) rectangle (1,-1); \node at (.5,-.625) {$q^{(2)}$}; \node[marked,below right,scale=.8] at (1,-1) {};
\foreach \x in {-.5,1} {\foreach \y in {-.625,.5,1.625} {\draw[thickline] (\x,\y) -- ++(.5,0);}}
\draw[thickline] (.5,-.25) --++(0,.25); \draw[thickline] (.5,1) --++(0,.25);

\begin{scope}[xshift=7.5cm]
 \draw[Box] (0,0) rectangle (1,1); \node at (.5,.5) {$\widetilde \pi_W$}; \node[marked,above left,scale=.8] at (0,1) {};
\draw[Box] (0,1.25) rectangle (1,2); \node at (.5,1.625) {$q^{(1)}$}; \node[marked,above left, scale=.8] at (0,2) {};
\draw[Box] (0,-.25) rectangle (1,-1); \node at (.5,-.625) {$q^{(2)}$}; \node[marked,below right,scale=.8] at (1,-1) {};
\foreach \x in {1} {\foreach \y in {-.625,1.625} {\draw[thickline] (\x,\y) -- ++(.5,0);}}
\draw[thickline] (.5,-.25) --++(0,.25); \draw[thickline] (.5,1) --++(0,.25);

\draw[thickline] (1,.5) --++(.5,0);
\draw[medthick] (0,.75) --++(-.5,0) arc(270:180:.25cm and .25cm) --++(0,.75) arc(0:90:.25cm and .5cm) -- ++(-4,0);
\draw[medthick] (0,.25) --++(-.5,0) arc(90:180:.25cm and .25cm) --++(0,-.75) arc(0:-90:.25cm and .5cm) --++(-4,0);
\draw[thickline] (0,1.625) --++(-.25,0) arc(270:180:.25cm) --++(0,.5) arc(0:90:.25cm) --++(-4.25,0);
\draw[thickline] (0,-.625) --++(-.25,0) arc(90:180:.25cm) --++(0,-.5) arc(0:-90:.25cm) --++(-4.25,0);
\end{scope}

\draw[Box] (1.5,-2) rectangle (2.5,3); \node[marked,scale=.9,right=.05] at (2.5,2) {};
\draw[thickline] (2.5,1.625) --++(3.5,0) arc(90:0:.25cm) --++(0,-1.75) arc(0:-90:.25cm) --++(-3.5,0);

\draw[Box] (3,0) rectangle (4,1); \node at (3.5,.5) {$\pi_Z$}; \node[marked,above left,scale=.8] at (3,1) {};
\draw[Box] (4.5,0) rectangle (5.5,1); \node[marked,above left,scale=.8] at (4.5,.9) {};

\foreach \x in {1,2.5,4} {\draw[thickline] (\x,.5) -- ++(.5,0);}
\draw[thickline] (3.5,1) arc(180:90:.25cm) -- ++(1.875,0) arc(90:0:.25cm) -- ++ (0,-1) arc(0:-90:.25cm) -- ++(-1.875,0) arc(270:180:.25cm);

\draw[Box] (-.5,-2.25) rectangle (9,3.25); \node[marked,above left] at (-.5,3.25) {};
\end{tikzpicture}
\end{equation*}
Moreover, the fusion category of $\mc M_1 - \mc M_1$ bimodules is braided.  \qed
\end{corollary}

\begin{proof}
For the first statement we compute the Bratelli diagram of the inclusion 
\begin{equation*}
\Hom_{\mc M_1,\mc M_1}(L^2(\mc M_{n})) \otimes \Hom_{\mc M_1,\mc M_1}(L^2(\mc M_{n})) \hookrightarrow \Hom_{\mc M_1,\mc M_1}(L^2(\mc M_{n})),
\end{equation*}
using the description from the proof of Proposition \ref{brat}.  As discussed above, the unitary operator $p_{n-1}u_np_{n-1} \in \Hom_{\mc M_1,\mc M_1}(L^2(\mc M_{2n}))$ implements the braiding for this inclusion.
\end{proof}

\def\cprime{$'$}

\end{document}